\documentclass[onefignum,onetabnum]{siamart251216}

\title{Classification of Smooth Alignable Voss Surfaces}

\author{Arvin Rasoulzadeh\thanks{%
METROM Mechatronische Maschinen GmbH,
Schönaicher Str.~6, 09232 Hartmannsdorf, Germany
(\email{arvin.rasoulzadeh@metrom.com}).}}

\headers{Smooth Alignable Voss Surfaces}{A. Rasoulzadeh}

\usepackage{amsfonts}
\usepackage{amssymb}
\usepackage{graphicx}
\usepackage{soul}
\usepackage{bigints}
\usepackage{overpic}
\usepackage{tikz}
\usepackage{array}
\usepackage{multirow}
\usepackage{mathrsfs}
\usepackage{comment}
\usepackage{dsfont}
\usepackage{relsize}
\usepackage{subcaption}

\usetikzlibrary{matrix,arrows.meta,calc}

\newsiamremark{remark}{Remark}
\newsiamremark{example}{Example}
\newsiamthm{lemdef}{Lemma and Definition}
\newsiamthm{property}{Property}

\newlength{\ml}
{%
    \begingroup
    \addtocounter{theorem}{-1}%
    \let\oldthetheorem\thetheorem
    \renewcommand{\thetheorem}{\oldthetheorem$'$}%
    \begin{theorem}[#1]%
}%
{%
    \end{theorem}%
    \endgroup
}

\renewcommand{\epsilon}{\varepsilon}
\renewcommand{\phi}{\varphi}

\newcommand{\sign}{\operatorname{sign}}

\DeclareMathOperator{\sn}{sn}
\DeclareMathOperator{\cn}{cn}
\DeclareMathOperator{\dn}{dn}

\def\R{\mathbb R}
\def\Z{\mathbb Z}

\def\d{\mathrm d}

\newcommand{\const}{\mathrm{const}}
\newcommand{\id}{\operatorname{id}}
\newcommand{\Hess}{\operatorname{Hess}}

\newcommand{\eqrefi}[2]{\textup{(\ref{#1})}\textsubscript{#2}}

\ifpdf
\hypersetup{
  pdftitle={Classification of Smooth Alignable Voss Surfaces},
  pdfauthor={Arvin Rasoulzadeh}
}
\fi

\begin{document}

\maketitle


\begin{abstract}
    Alignable nets are grid structures that can collapse to a planar strip. This property simplifies on-site assembly and enables compact transport and storage. These structures are then deployed by a scissor motion at vertices in a desired location. In this article, we classify all non-developable surfaces supporting an alignable net that additionally have the geodesic-conjugate net property, namely, the alignable Voss surfaces. Using the surface elements and adapted frames, we show that these surfaces arise as rotation fields of either helicoidal pseudospherical surfaces or Amsler surfaces, resulting in two distinct classes.\\
    The first class consists of helicoidal Voss surfaces. Exploiting this helicoidal structure, we combine a kinematic interpretation with the Bour family to derive explicit immersion formulas depending on the isometric deformation and alignability parameters. Furthermore, we examine several subclasses and an extension of this result in detail. In particular, we study the notable case of infinitely many geodesic-conjugate nets on a helicoid and provide a kinematic explanation for why this phenomenon appears in computations.\\
    For the second class, we obtain the classification at the level of the fundamental forms and show that they correspond to certain $3$-parameter families of non-helicoidal intrinsic surfaces of revolution. These families are governed by general Painlev\'{e} III solutions including both regular and nonregular solutions. Finally, a surprising feature of this class is that its members support geodesic isothermal-conjugate nets absent from Eisenhart’s classification of the aforementioned nets, thereby identifying a missing case in that classification.
\end{abstract}


\begin{keywords}
Voss surfaces, alignable nets, isometric deformations, Bour family,
intrinsic surfaces of revolution
\end{keywords}


\begin{MSCcodes}
    53A05, 53A17
\end{MSCcodes}


\section{Introduction}

Alignable nets arise in the broader context of deployable structures, i.e., systems designed to be fabricated and transported in a compact configuration and subsequently deployed on site \cite{FenciCurrie2017,Howell2001,SultanSkelton2003,YouPellegrino1997}. In architectural engineering, a particularly successful line of work concerns bending-active gridshells, in which a nearly planar grid is assembled and then formed into a curved load-bearing surface. Prominent examples range from the Mannheim Multihalle \cite{HappoldLiddell1975,Otto1974} to the later Downland Gridshell \cite{HarrisRomerKellyJohnson2003}.

From the geometric viewpoint, many bending-active gridshell constructions are modeled by discrete Chebyshev nets, whose continuous counterparts have a long history and subtle existence theory \cite{DoutheEtAl2017,MassonMonasse2017}. However, more recently, elastic geodesic grids have also appeared in the study of planar-to-spatial deployment paths and their parametrized families \cite{PillweinEtAl2021MultiPatch,PillweinEtAl2020,PillweinMusialski2021}.

In this landscape, Tellier introduced and developed the notion of \emph{alignable nets}. These are structures whose intrinsic geometry admits a collapse into a planar strip, thereby enabling compact storage and transport while retaining a controlled deployed shape \cite{tellierIASS,Tellier2022a,Tellier2022b}. Since then, the alignable nets appear in multiple works \cite{DeployableDavide,AlignablePauly,Morphology}.
Finally, a recent computational and theoretical treatment is provided by Pellis (see \cite{AlignablePellis}), who develops the Cartan moving-frame theory for alignable lamella gridshells and derives explicit shape restrictions for such nets. In particular, he shows that geodesic alignable nets are limited to surfaces isometric to surfaces of revolution.

\begin{figure*}[t!]
  \centering
  \begin{subfigure}[b]{0.24\textwidth}
    \centering
    \begin{overpic}[height=30mm, width=\textwidth]{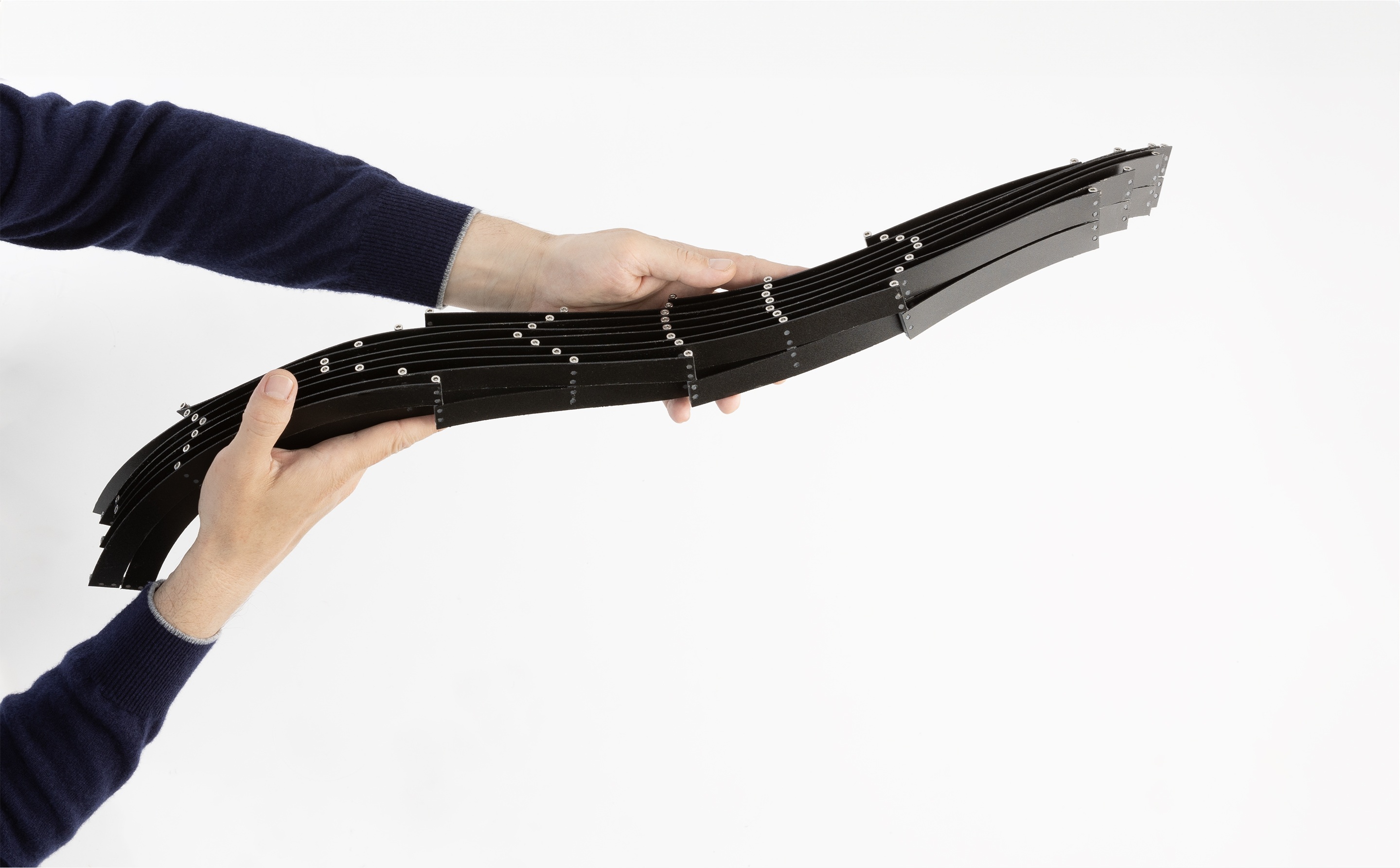}
    \end{overpic}
  \end{subfigure}
  \hfill
  \begin{subfigure}[b]{0.24\textwidth}
    \centering
    \begin{overpic}[height=30mm, width=\textwidth]{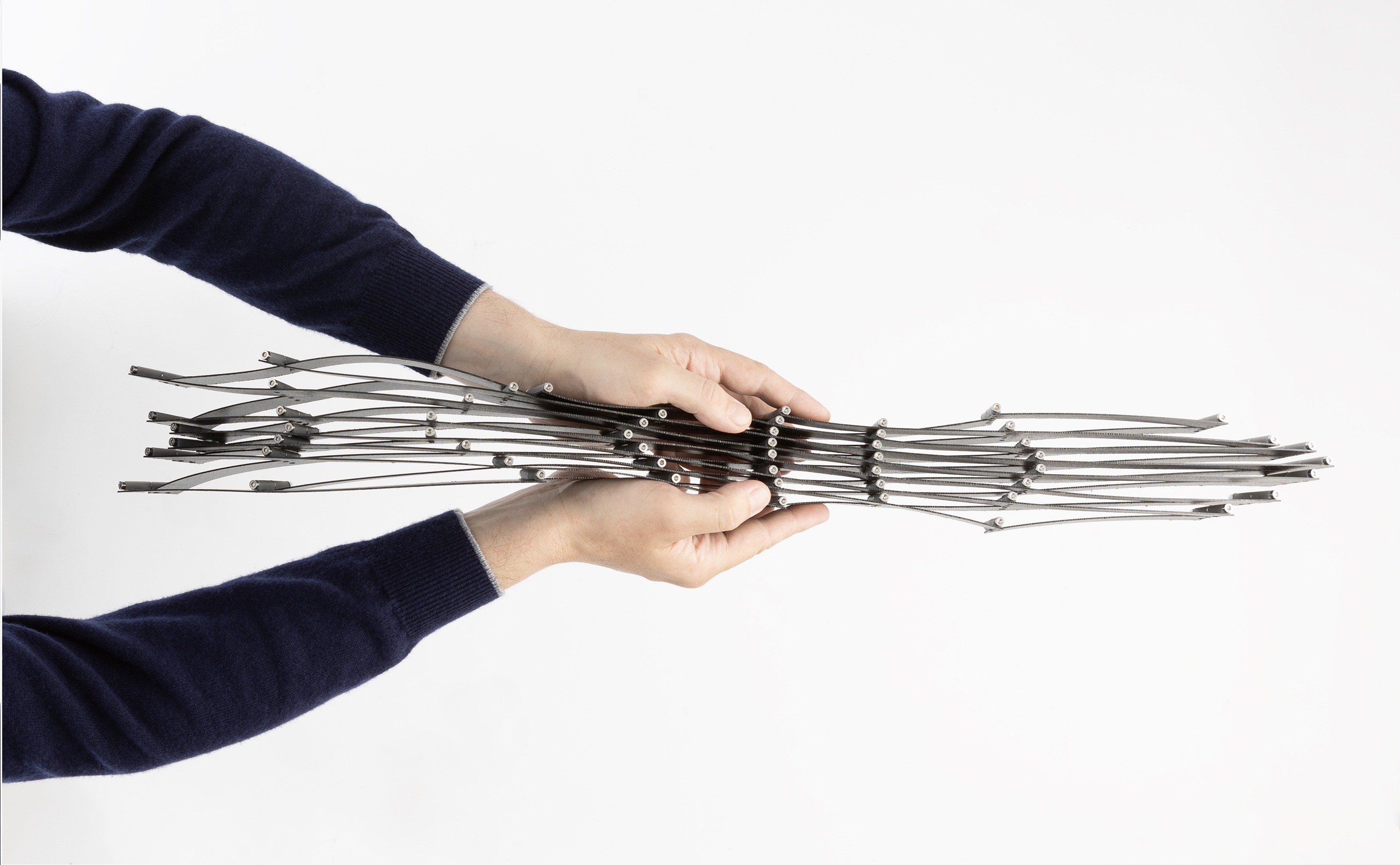}
    \end{overpic}
  \end{subfigure}
  \hfill
  \begin{subfigure}[b]{0.24\textwidth}
    \centering
    \begin{overpic}[height=30mm, width=\textwidth]{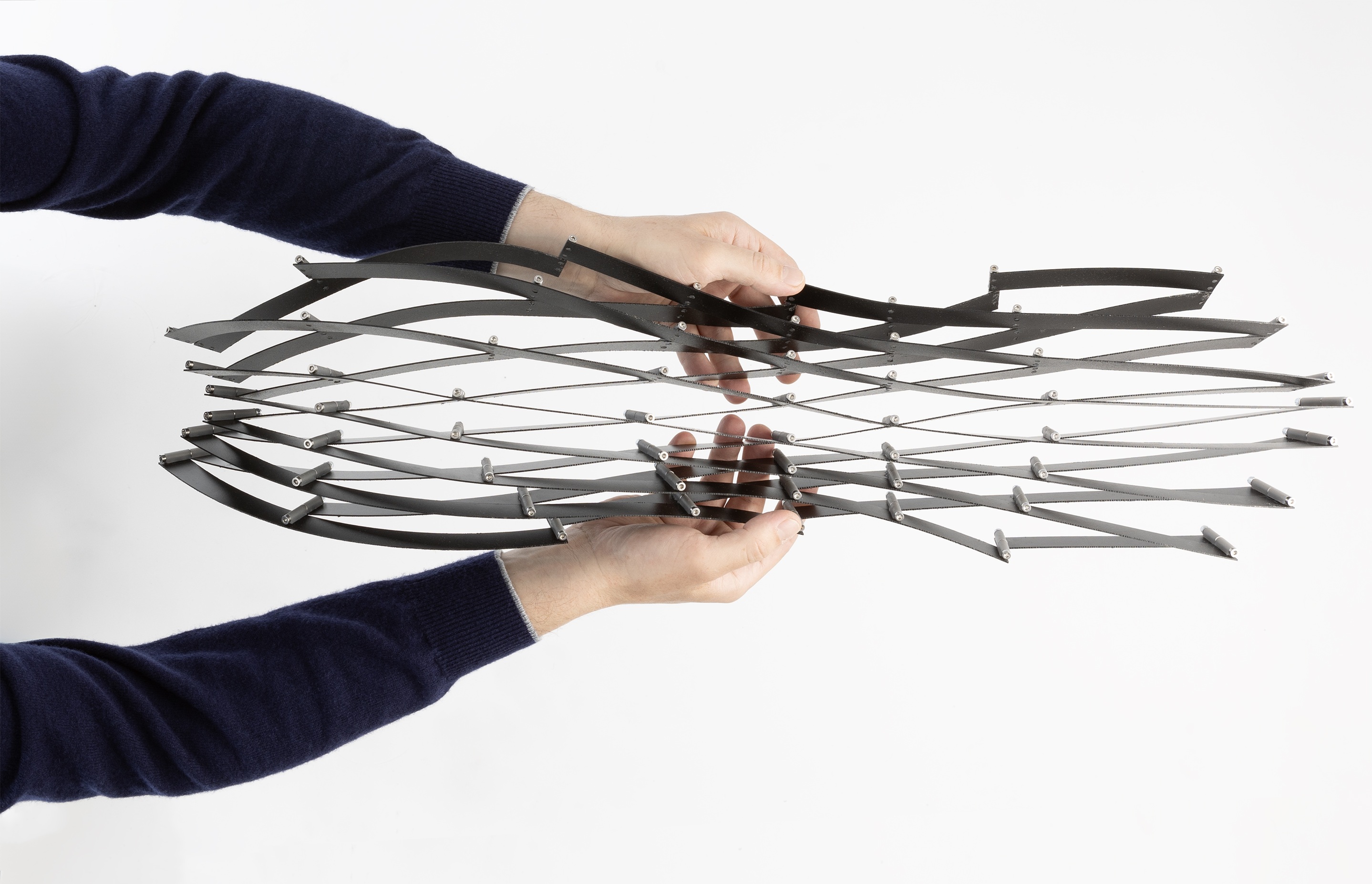}
    \end{overpic}
  \end{subfigure}
  \hfill
  \begin{subfigure}[b]{0.24\textwidth}
    \centering
    \begin{overpic}[height=30mm, width=\textwidth]{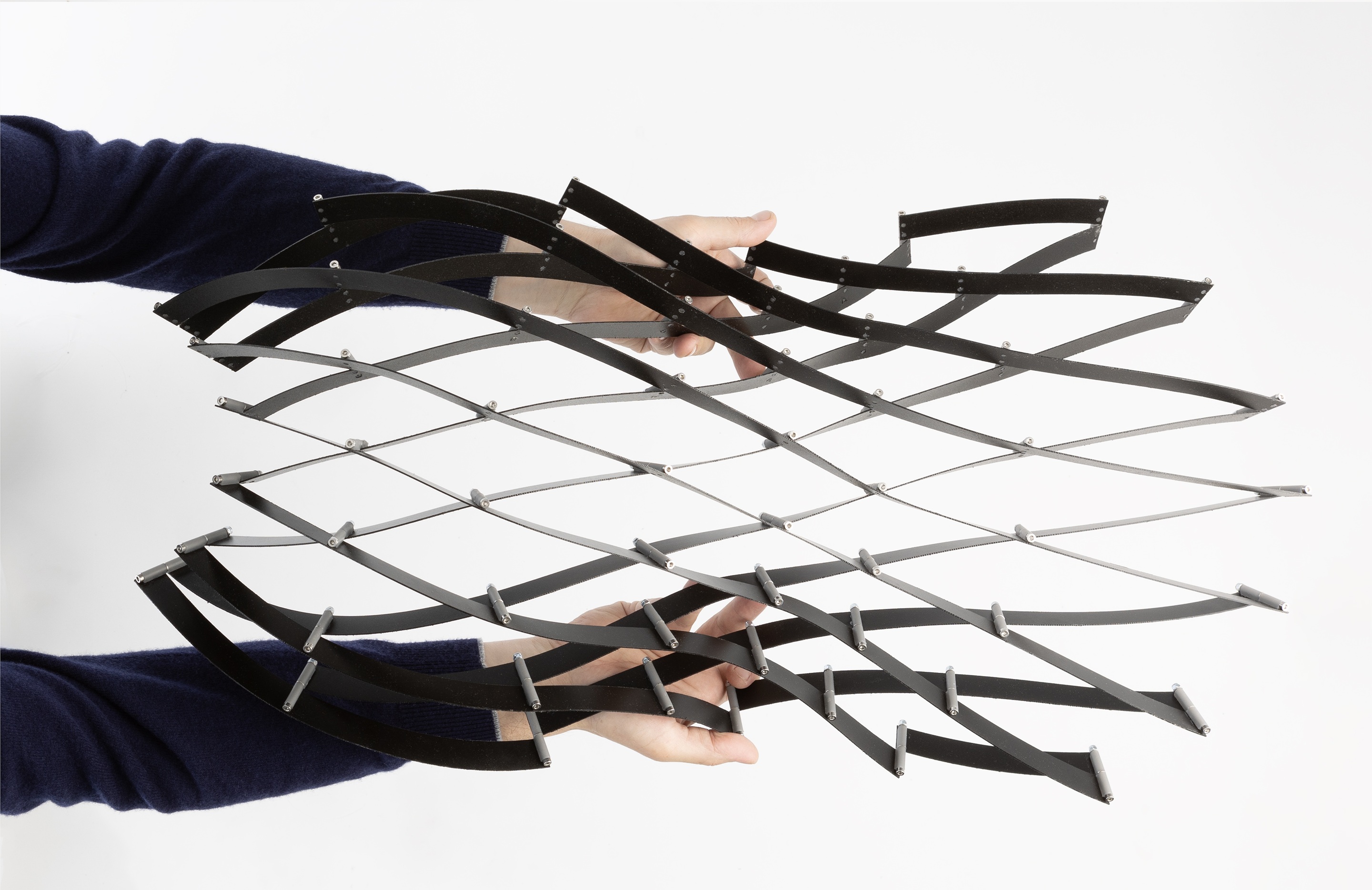}
    \end{overpic}
  \end{subfigure}
  \caption{Physical prototype of an alignable net shown in its deployment sequence. Image courtesy of Davide Pellis (see \cite{AlignablePellis}).}
  \label{fig:alignable:examples}
\end{figure*}
\subsection{Review on Voss surfaces} 
Most likely Liouville was the first to encounter the surfaces that are now called \emph{Voss surfaces}. In one of his studies he announces that, \emph{with the exception of developable surfaces, no surface supports two orthogonal families of geodesics forming a conjugate net} \cite[p.~95]{Voss}. Nevertheless, the subject did not receive much attention until Aurel Voss’ 1888 memoir \cite{Voss}, where these surfaces are introduced explicitly as \emph{geodesic-conjugate surfaces} and linked to pseudospherical surfaces. Beyond establishing the connection to the theory of constant negative curvature surfaces, Voss highlights the right helicoid as an example carrying infinitely many geodesic-conjugate nets. Then he derives conditions for surfaces of revolution to support such nets, leading to an ODE for the profile curve \cite{Voss}.

\paragraph{Related works to this article}
An early systematic treatment closely related to the present work is due to Gambier \cite{Gambier18}. He identifies two families of Voss surfaces isometric to surfaces of revolution, a helicoidal family and a second, non-helicoidal family. For the helicoidal case, he establishes the connection with helicoidal pseudospherical surfaces and studies the corresponding conjugate net preserving isometric deformations, including explicit parametrizations of the surface of revolution case. For the second family, he derives the governing angle equation without recognizing this equation as a reduction of Painlev\'{e} III. Lin and Conte \cite{LinConte29} subsequently identify this reduction with Painlev\'{e} III and discuss its relevant solutions. The classical theory was also developed extensively by Eisenhart, both in his treatise \cite{EisenhartTreatise} and in a series of papers \cite{EisenhartAssociate,Eisenhart1914,Eisenhart1915,Eisenhart1917}. Of particular relevance here is \cite{EisenhartAssociate}, where Eisenhart addresses the classification of Voss surfaces supporting geodesic isothermal-conjugate nets. As shown in this article, however, this classification does not account for all such nets.

\paragraph{Other classical strands \& contemporary works}
Other classical developments approach Voss surfaces from complementary directions. Guichard’s work \cite{Guichard1890b,Guichard1890a}, later surveyed by Salkowski \cite{Salkowski1924} and further developed by Gambier \cite{Gambier19}, studies the subject through congruences and transformation theory. Related questions concerning isometric deformations preserving a conjugate net and minimal Voss surfaces were treated respectively by Finikoff \cite{Finikoff17}, Razzaboni \cite{Razzaboni38}, and Tachauer \cite{Tachauer} and further explicit examples connected with classical surface geometry appear in Diller \cite{Diller}.

The link established by Voss between geodesic-conjugate nets and pseudospherical surfaces naturally invites integrable systems methods. Babich constructed Voss surfaces from the Lax pair of the sine-Gordon equation \cite{BabichVoss}. A recent contribution by Marvan reinterprets Guichard's transformation via recursion operators for sine-Gordon symmetries and generates Voss surfaces using their \emph{support functions} \cite{Marvan}. Finally, Voss surfaces have gained renewed attention through discretization and geometric modelling, where discrete Voss surfaces serve as design spaces for elastic gridshells and related structures \cite{Doliwa1999,MontagneEtAl2021VossSurfaces,sauer,Tellier2022a,Tellier2022b}. The author also extends the theory of Voss surfaces to the semi-discrete world in \cite{izmestiev}.

\subsection{Contribution and outline of the article}

This article addresses the question \emph{which Voss surfaces are alignable?} To answer this question, we provide a full classification of all non-developable alignable Voss surfaces. The article is organized as follows: 

Section~\ref{sec:preliminary} collects the notation and basic notions used throughout the article, including surface elements, frames, isometric deformations and infinitesimal isometric deformations.

Section~\ref{sec:V:surface:as:Rotation:Fields:of:K:surfaces} is devoted to the analysis of Voss surfaces via the pseudospherical surfaces for which they arise as rotation fields. 
This viewpoint allows us to convert questions about the Voss surfaces into the language of pseudospherical surfaces, a class that has been extensively studied and for which a broad range of techniques is available. The main results of this section are Theorems \ref{theorem:I:II:III:V:surface} and \ref{theorem:isometric:deformation}.

The central classification result is obtained in Section~\ref{sec:reciproal:parallel:K:surfaces}. There we determine all pseudospherical surfaces whose rotation fields can be alignable Voss surfaces. Theorem~\ref{theorem:main} states that the alignable Voss surfaces can only appear as the rotation fields of \emph{helicoidal pseudospherical surfaces} and \emph{Amsler surfaces}. 
This dichotomy gives rise to the two classes studied in the remainder of the article and connects the present results with the classical work of Gambier and the later analysis of Lin and Conte \cite{Gambier18,LinConte29}.

Section~\ref{sec:1st:kind} treats the first class, consisting of the rotation fields of helicoidal pseudospherical surfaces. A kinematic interpretation of Theorem~\ref{theorem:main}, namely, Theorem~\ref{theorem:main:first} then yields the formulas of immersion of this class of Voss surfaces. These formulas depend on the alignability and isometric deformation parameters and are presented in Theorem~\ref{theorem:main:immersion:V}. The section finishes with multiple remarks on Voss surfaces. 

Section~\ref{sec:2nd:kind} addresses the alignable Voss surfaces that arise as rotation fields of Amsler surfaces. In this case, Theorem~\ref{theorem:classification:second} gives a complete description at the level of the fundamental forms. The resulting surfaces form $3$-parameter families governed by general Painlev\'{e} III solutions. Surprisingly, this class supports geodesic isothermal-conjugate nets that are not accounted for in Eisenhart's 1906 classification \cite{EisenhartAssociate}, in which he claims to have classified all such nets. Subsection~\ref{subsec:eisenhart} is devoted to this issue and identifies the source of the omission. Finally, the section concludes with an analysis of the global symmetries of this class of Voss surfaces.

\section{Notations and Conventions}\label{sec:preliminary}
In this article, a net is an immersion \(\psi\colon\Omega\to\mathbb{R}^3\) of a simply connected open set \(\Omega\subset\mathbb{R}^2\). We denote by $I_\psi$ the first \emph{fundamental form} induced by $\psi$ and we refer to \(\Omega\) as the \emph{parameter domain} of $\psi$. The basic objects of this article are certain surface elements, for which we adopt Kühnel’s notion \cite{Kuhnel}:

\begin{definition}\label{def:surface:element}
    Two nets $\psi\colon\Omega\to\R^3$ and $\tilde{\psi}\colon\tilde{\Omega}\to\R^3$, are called equivalent if there is a diffeomorphism $\phi\colon\Omega\rightarrow\tilde{\Omega}$ such that $\psi=\tilde{\psi}\circ\phi$. A \emph{surface element} is an equivalence
    class of nets. 
\end{definition}
As evident from the above definition, the notion of a \emph{surface element} enables transitions between different nets. Consequently, we denote a surface element by $P = (\psi,\Omega)$ with the net $\psi\colon\Omega\to\R^3$ as its \emph{representative}. We refer to any other representative $(\tilde{\psi},\tilde{\Omega})$ in its equivalence class as a \emph{reparametrization} of $P$. Simultaneously, if a net belongs to the surface element $P$, we say that $P$ supports that net. Throughout the article, we occasionally identify a surface element $P$ with the image $\psi(\Omega)$ of one of its representatives whenever no confusion is likely to arise.

To carry out calculations on surface elements, we use frames. A \emph{frame} for $P$ is a 2-tuple $(E_1,E_2) \in \mathfrak{X}(\Omega)\times \mathfrak{X}(\Omega)$ of nowhere vanishing, linearly independent smooth vector fields $E_1$ and $E_2$. 
We call a frame $(E_1,E_2)$ a \emph{coordinate frame} if and only if $E_1$ and $E_2$ commute at every point of $\Omega$ that is, $\mathcal{L}_{E_1} E_2 = 0$ (see \cite[page~234]{Lee}). 
We refer to the integral curves of the vector fields $E_i$ simply as $E_i$-lines.
We denote by $(\theta^{\,1},\theta^{\,2})\in\mathfrak{X}^\ast(\Omega)\times\mathfrak{X}^\ast(\Omega)$ the \emph{dual coframe}, i.e., $\theta^i(E_j)=\delta^i_j$ for $i,j=1,2$. 
Since we shall repeatedly manipulate second fundamental forms of surface elements, for a surface element $P$ in the dual coframe $(\theta^1,\theta^2)$ we have:
\begin{equation*}
    II_P = L \theta^1\otimes\theta^1 + 2 M \theta^1\odot\theta^2 + N\theta^2\otimes\theta^2,
    \quad\text{where}\quad
    {\theta}^{\,1}\odot{\theta}^{\,2} := \frac{1}{2}({\theta}^{\,1}\otimes{\theta}^{\,2}+{\theta}^{\,2}\otimes{\theta}^{\,1}).
\end{equation*}
When dealing with coordinate frames, we occasionally express them in $(u,v)$ parameters and represent $(E_1,E_2) = (\partial_u, \partial_v)$ and $(\theta^{\,1},\theta^{\,2}) = (\mathrm{d}u, \mathrm{d}v)$. 
Throughout the article, \(A\propto B\) means that \(A=\rho B\) for some nowhere vanishing scalar factor \(\rho\). We use this notation when the omitted factor has no effect on the conclusion being drawn, for instance when only the vanishing or non-vanishing of the expression is relevant.

For a frame \((E_1,E_2)\) that is orthonormal, set \(e_i:=\d\psi(E_i)\), \(i=1,2\), and \(\nu:=e_1\times e_2\). Then, for every \(Z\in\mathfrak{X}(\Omega)\), the directional derivatives satisfy
\begin{gather}\label{eq:fram:field:change:12}
        D_Z e_1 = \Theta_{12}(Z)\,e_2 + \Theta_{1n}(Z)\,\nu,\qquad
        D_Z e_2 = \Theta_{21}(Z)\,e_1 + \Theta_{2n}(Z)\,\nu,\\
        D_Z \nu = \Theta_{n1}(Z)\,e_1 + \Theta_{n2}(Z)\,e_2,
\end{gather}
where $-\Theta_{ji} = \Theta_{ij} \in \mathfrak{X}^\ast(\Omega)$ are called the \emph{connection forms}. In particular, $\Theta_{12} = \kappa^g_1\,\theta^1 + \kappa^g_2\,\theta^2$ where $\kappa^g_i$ is the \emph{geodesic curvature} (for details see \cite{Kuhnel,AlignablePellis}).
The Codazzi and Gauss equations take the following form respectively:
\begin{equation}\label{eq:codazzi:cartan}
    \begin{gathered}
        \d\Theta_{1n} = \Theta_{12}\wedge\Theta_{2n},\qquad
        \d\Theta_{2n} = -\Theta_{12}\wedge\Theta_{1n},\\
        \d\Theta_{12} = \Theta_{1n}\wedge\Theta_{n2}
        = -\mathrm{K}\,\theta^1\wedge\theta^2.
    \end{gathered}
\end{equation}
where $\mathrm{K}$ is the Gaussian curvature. 
\subsection{Isometric Deformations}
We briefly review isometric deformations and their related notions that are repeatedly used in this article.
\begin{definition}[\cite{Kuhnel}]
    An \emph{isometric deformation} of a net $\psi\colon\Omega\to\R^3$ is a smooth family of nets $\psi^t\colon\Omega\to\R^3$ where $t \in (-\epsilon, \epsilon)$ with $\psi^0 = \psi$ such that $I^t = I^0$ for all $t$.
    We call an isometric deformation \emph{non-trivial} if no two $\psi^{t_1}$, $\psi^{t_2}$ are related by a rigid motion, i.e., there is no isometry $\phi \colon \R^3 \to \R^3$ such that $\psi^{t_2} = \phi \circ \psi^{t_1}$.

    Accordingly, two surface elements $Q$ and $P$ are said to be \emph{isometric} if they have representatives $(\psi^t,\Omega)$ and $({\psi},\Omega)$ such that ${\psi}^t$ is an isometric deformation of $\psi$.
\end{definition}
In \cite{IzmestievTsurface}, numerous examples of isometric deformations are presented in the context of T-surfaces.
\subsection{Infinitesimal Isometric Deformation and its Rotation Field}\label{sec:IID:rotation:field:review}
Let $\psi\colon\Omega\to\R^3$ be any net and let $\psi^{\,t}\colon\Omega\to\R^3$ be its isometric deformation.
Consider its velocity field $\xi\colon\Omega\to\R^3$ at $t = 0$:
\begin{equation*}
    \xi(p) := \left. \frac{\d}{\d t} \right|_{t=0} \psi^t(p),\qquad\qquad p\in \Omega,
\end{equation*}
viewed as a vector field along the surface element $(\psi,\Omega)$. For each $t$, denote by $I^t(Z_1,Z_2) = \langle \d\psi^t(Z_1), \d\psi^t(Z_2)\rangle$ the first fundamental form of $\psi^t$, where $Z_1,Z_2$ are vector fields on $\Omega$ and $\langle -,- \rangle$ is the Euclidean inner product on $\mathbb{R}^3$. The isometry condition is $I^t = I^0$ for all $t$, hence
\begin{equation}\label{eq:constant:I}
    \left. \frac{\d}{\d t} \right|_{t=0} I^{t}(Z_1,Z_2) = 0,\qquad\qquad \forall\, Z_1, Z_2 \in \mathfrak{X}(\Omega).
\end{equation}
Using the Euclidean connection $D$ on $\mathbb{R}^3$, we have $\frac{\d}{\d t}\Big|_{t=0}\d\psi^t(Z_1) = D_{Z_1}\,\xi = \d\xi(Z_1)$. Consequently, \eqref{eq:constant:I} implies
\begin{equation}\label{eqn:IID}
    \langle \d\xi(Z_1), \d\psi(Z_2) \rangle + \langle \d\xi(Z_2), \d\psi(Z_1) \rangle = 0.
\end{equation}

\begin{definition}
    Any vector field $\xi \colon \Omega \to \R^3$ satisfying equations \eqref{eqn:IID} is called an \emph{infinitesimal isometric deformation}, IID for short, of the net $\psi\colon\Omega\to\R^3$.
    An IID is called \emph{trivial}, if it is induced by a Killing field on $\R^3$, namely, $\xi = \eta \times \psi + \zeta$ where $\eta,\zeta \in \R^3$.
\end{definition}
While the initial velocity field of any isometric deformation is an IID, some IIDs do not arise as initial velocity fields of isometric deformations (see \cite{Spivak5, Darja}). Furthermore, even though the rotation field and the IID are smooth maps they need not be immersions and hence they do not necessary form nets.

\begin{theorem}\label{theorem:IID:fundamnetals}
    For every IID $\xi\colon \Omega \to \R^3$ of a net $\psi\colon\Omega\to\R^3$ there is a unique pair of vector fields $\eta, \zeta \colon \Omega \to \R^3$ such that the following equations are satisfied:
    \begin{equation}\label{eq:gather}
    \xi = \eta \times \psi + \zeta,\qquad\qquad
    \d\xi = \eta \times \d\psi,\qquad\qquad
    \d\zeta = \psi \times \d\eta.
    \end{equation}
\end{theorem}
The fields $\eta$ and $\zeta$ are called the \emph{rotation field} and the \emph{translation field} of $\xi$, respectively.
\eqrefi{eq:gather}{2} says that $\eta$ is the axis of the infinitesimal rotation of the tangent plane of $(\psi,\Omega)$.
The IID $\xi$ is trivial if and only if the rotation and the translation fields are constant. 

The following lemma \& definition is adapted from the results of Chapter~4 of \cite{Darja}. 

\begin{lemdef}\label{lem:parallel:tangent:space}
    Let $\eta\colon\Omega\to\R^3$ be a rotation field of a net $\psi\colon\Omega\to\R^3$. Then for every $p \in \Omega$, one has $\mathrm{Im}\,\mathrm{d}\eta_p \subseteq \mathrm{Im}\,\mathrm{d}\psi_p$. The map $A : T\Omega \rightarrow T\Omega$ defined pointwise by
        \begin{equation}\label{eq:rotation:operator}
            A_p = \mathrm{d}\psi^{-1}_p \circ \mathrm{d}\eta_p,
        \end{equation}
        is called the rotation operator. $A$ is traceless and it is self-adjoint with respect to the second fundamental form of $\psi$. Finally, the first and second fundamental forms of $\eta$ fulfill the following relations:
        \begin{equation}\label{eq:I:II:K}
            I_\eta(Z_1,Z_2) = I_\psi(A(Z_1),A(Z_2)),\qquad
            II_\eta(Z_1,Z_2) = II_\psi(Z_1,A(Z_2)),
        \end{equation}
        where $Z_1,Z_2 \in \mathfrak{X}(\Omega)$.
\end{lemdef}

For a net $\psi:\Omega\to\R^3$, let $\mathfrak{R}(\psi)$ denote the set of all its rotation fields.
\begin{proposition}\label{prop:vector:space}
    Let $\psi:\Omega\to\R^3$ be a net, then $\mathfrak{R}(\psi)$ is a vector space.
\end{proposition}
\begin{proof}
    Every $\eta\in \mathfrak{R}(\psi)$ is a smooth map $\Omega\to\mathbb{R}^3$, and hence, with respect to the standard basis of $\mathbb{R}^3$, may be written as $\eta=(\eta^1,\eta^2,\eta^3)$ with $\eta^i\in C^\infty(\Omega)$. Thus $\mathfrak{R}(\psi)\subset C^\infty(\Omega,\mathbb{R}^3)\cong \prod_{i=1}^3 C^\infty(\Omega)$.
    Moreover, $0\in\mathfrak{R}(\psi)$, since any constant translation $\xi$ is an IID and satisfies $\d\xi=0=0\times\d\psi$. Let now $\eta_1,\eta_2\in\mathfrak{R}(\psi)$, with corresponding IIDs $\xi_1,\xi_2$. For any $\alpha,\beta\in\mathbb{R}$, the field $\alpha\xi_1+\beta\xi_2$ is again an IID, and $\d(\alpha\xi_1+\beta\xi_2) = (\alpha\eta_1+\beta\eta_2)\times\d\psi$.
    By uniqueness of the rotation field, $\alpha\eta_1+\beta\eta_2\in\mathfrak{R}(\psi)$. Hence, $\mathfrak{R}(\psi)$ is a linear subspace of $\prod_{i=1}^3 C^\infty(\Omega)$.
\end{proof}
If $(\tilde{\psi},\tilde{\Omega})$ is another representative of $P = (\psi,\Omega)$, so that $\tilde{\psi} =\psi\circ\phi$ for a diffeomorphism $\phi:\Omega\to\tilde{\Omega}$, then $\eta\mapsto\eta\circ\phi$ defines a linear isomorphism $\mathfrak{R}(\psi)\cong\mathfrak{R}(\tilde{\psi})$. We therefore write $\mathfrak{R}(P)$ for the space of rotation fields associated with the surface element $P$.

We next recall the following cyclic relation between an immersion, its IID and the corresponding rotation field (see \cite[\S XI]{EisenhartTreatise} and \cite{Darja}):
\begin{proposition}\label{prop:cyclic}
    Let $\psi,\xi,\eta\colon\Omega\to\mathbb\R^3$ be nets, with $\xi$ an IID of $\psi$ and $\eta$ its corresponding rotation field. 
    For any permutation $(f_1,f_2,f_3)$ of $(\psi,\eta,\xi)$, $f_1$ and $f_2$ are conjugate nets if and only if $f_3$ is an asymptotic net.
\end{proposition}

Finally, before closing the overview, we mention the following proposition: 
\begin{proposition}[\cite{EisenhartTreatise}]\label{prop:rotation:field:iff}
     Let $\psi,\eta\colon\Omega\to\R^3$ be two nets and let $\mathbf{II}$ and $\mathbf{III}$ be the matrices of second and third fundamental forms respectively. Furthermore, define
     \begin{equation*}
        \mathrm{Det}(\mathbf{A},\mathbf{B}) := \frac{1}{2}\left(\det\left(\mathbf{A} + \mathbf{B}\right) - \det\left(\mathbf{A}\right) - \det\left(\mathbf{B}\right)\right),
    \end{equation*}
    where $\mathbf{A}$ and $\mathbf{B}$ are two $2\times 2$ matrices. Then $\eta$ is a rotation field of $\psi$ if and only if $\mathrm{Det}(\mathbf{II}_\psi,\mathbf{II}_\eta)=0$ and $\mathbf{III}_\psi = \mathbf{III}_\eta$.
\end{proposition}
\section{V-surfaces as Rotation Fields of K-surfaces}\label{sec:V:surface:as:Rotation:Fields:of:K:surfaces}
Pseudospherical surfaces are perhaps the most studied class of surfaces in the history of differential geometry. In this manuscript we refer to them as K-surfaces. A tradition that has roots in the development of discrete differential geometry by Sauer \cite{sauer}.
\begin{definition}\label{def:K:surface}
    A surface element $K$ is called a K-surface if it has a constant Gaussian curvature. 
\end{definition}
In this article, all K-surfaces are assumed to have Gaussian curvature $-1$.
At each point of $K$ there are two distinguished tangent directions along which the normal curvature vanishes. These are called the \emph{asymptotic directions}.  
Up to a change of sign, we may choose unit tangent vectors $E_{1}$ and $E_{2}$ along these directions. The ordered pair $(E_{1},E_{2})$ will be referred to as a \emph{K-frame}. Its dual coframe $({\theta}^{1},{\theta}^{2})$ will be called the \emph{K-coframe}.  
The integral lines of the asymptotic directions are then called \emph{asymptotic lines}. 
On K-surfaces it is a classical fact that asymptotic nets and their Gauss maps have \emph{Chebyshev net property}. Accordingly, a reparametrization of $K$ by such a net is referred to as a \emph{K-net} \cite{EisenhartTreatise}.
\begin{remark}\label{remark:K:frame:coordinates}
    A K-frame is a coordinate frame. In fact, from \cite[Lemma~2.10]{izmestiev} we know that for a K-frame $\nabla_{E_1}E_2 =\nabla_{E_2}E_1 =0$ and consequently $\mathcal{L}_{E_1}E_2 =0$. 
\end{remark}
The remark above allows us, whenever convenient, to work in $(u,v)$ coordinates by identifying the K-frame with the coordinate frame $(\partial_u,\partial_v)$.

A major classical result describing K-surfaces at the level of the fundamental forms, is as follows in a tensor language (see \cite{EisenhartTreatise} for details):
\begin{theorem}\label{theorem:classify:K-nets}
    Let $K$ be a K-surface with Gaussian curvature $\mathrm{K} = -1$ and the K-frame $(E_1,E_2)$. Furthermore, let $\omega\in C^\infty(\Omega,(0,\pi))$ be the angle between $E_1$ and $E_2$. Then the fundamental forms of $K$ in the K-frame are
    \begin{equation}\label{eq:I:II:psi}
        {I} = {\theta}^{\,1}\otimes{\theta}^{\,1} + 2\cos{(\omega)}\,{\theta}^{\,1}\odot{\theta}^{\,2} + {\theta}^{\,2}\otimes{\theta}^{\,2},\qquad
        {II} = 2\sin(\omega)\,{\theta}^{\,1}\odot{\theta}^{\,2},
    \end{equation}
    with $\omega$ satisfying the sine-Gordon equation $\Hess \omega\,(E_1, E_2) = \sin{(\omega)}$.
\end{theorem}
From now on, we refer to $\omega$ as asymptotic angle function.
\subsection{V-surfaces \& V-nets}\label{sec:V:surfaces:V:nets}
Originally, a \emph{Voss surface} was thought of as the image of a \emph{geodesic-conjugate net}, i.e., a conjugate net whose coordinate lines are geodesics (see \cite[p.~3]{izmestiev} and \cite{Voss}). In this article, however, we adopt the equivalent perspective of viewing a Voss surface as an \emph{immersion rotation field} of a K-surface. From now on, we refer to a Voss surface by the term \emph{V-surface}.
\begin{definition}\label{def:V:surface}
    A surface element $V$ is called a V-surface if there exists a K-surface $K$ and representatives $\psi,\eta\colon\Omega\to\mathbb{R}^3$ of $K$ and $V$, respectively, such that $\eta$ is a rotation field of $\psi$.
\end{definition}

\begin{remark}\label{remark:V:in:R}
    Let $K = (\psi,\Omega)$ be a K-surface. Then by Definition~\ref{def:V:surface}, every immersion $\eta \in \mathfrak{R}(K)$ represents a V-surface.
\end{remark}

\begin{definition}\label{def:V:net}
    We call a net $\eta\colon\Omega\to\R^3$ with nowhere vanishing Gaussian curvature, a V-net if it satisfies $\eta_{uv}(u,v)\in \operatorname{span}\{\eta_u(u,v),\,\eta_v(u,v)\}$ for all $(u,v)\in \Omega$ and each coordinate line is a geodesic of $\eta(\Omega)$.
\end{definition}

\begin{remark}\label{remark:box:reparametrization}
    Note that under a \emph{box reparametrization} $(u,v)\mapsto (f(u),g(v))$, where $f(u)$ and $g(v)$ are smooth, and $f_u\,g_v\neq 0$ for all $(u,v)\in \Omega$, a V-net remains a V-net.
\end{remark}

The following proposition establishes the relation between V-nets and V-surfaces.
\begin{proposition}\label{prop:V:net}
    Every V-surface supports a V-net. Conversely, the surface element represented by a V-net is a V-surface. 
\end{proposition}
\begin{proof}
    Let $V = (\eta, \Omega )$ be a V-surface with the Gauss map $\nu:\Omega\to\mathbb{S}^2$. Then there exists a K-surface $K = (\psi,\Omega)$ such that $V\in\mathfrak{R}(K)$. There exists a K-net reparametrization $(\tilde{\psi},\tilde{\Omega})$ of $K$ by the diffeomorphism $\phi: \Omega \rightarrow \tilde{\Omega}$. Now, we claim that $\tilde{\eta} := \eta \circ \phi^{-1}$ is a V-net. Since the relations of \eqref{eq:I:II:K} are invariant with respect to a reparametrization, using Lemma\,\ref{lem:parallel:tangent:space} and Proposition\,\ref{prop:rotation:field:iff}, we get $\mathbf{III}_{\tilde{\eta}} = \mathbf{III}_{\tilde{\psi}}$ and $\mathrm{Det} (\mathbf{II}_{\tilde{\eta}}, \mathbf{II}_{\tilde{\psi}}) = 0$.
    Then the first equation implies that $\tilde{\nu} := \nu \circ \phi^{-1}$ is a Chebyshev net while the second one implies $M_{\tilde{\eta}}\,M_{\tilde{\psi}} = 0$ which itself implies that $M_{\tilde{\eta}} = 0$. Consequently, $\tilde{\eta}$ is a conjugate net with Chebyshev Gauss map and is therefore, by \cite[Thm.2.8]{izmestiev}, a V-net.

    Conversely, let $V$ be the surface element arising from the V-net $\eta\colon\Omega\to\R^3$ with $\nu$ as its Gauss map. Then $\nu$ is an immersion and by \cite[Thm.2.8]{izmestiev} it is a Chebyshev net. Consequently, $\nu_{uv} \parallel \nu$. Therefore, the Lelieuvre one-form $(\nu \times \nu_u)\,\d u - (\nu\times\nu_v)\,\d v$ is closed. Since $\Omega$ is simply connected, it has a primitive $\psi$, characterized up to translation by
    \begin{equation*}
        \psi_u = \nu \times \nu_u,\qquad\qquad
        \psi_v = \nu_v \times \nu.
    \end{equation*}
    Because $\nu$ is an immersion, $\psi$ is an immersion. Its Gauss map is $\pm\nu$, and the coordinate directions are asymptotic. Moreover, the Chebyshev property of $\nu$ gives $\mathrm{K}_\psi = -1$, so $\psi$ is a K-net. Now $\eta$ and $\psi$ have the same third fundamental form, since their Gauss maps agree up to sign. Furthermore, the given coordinates are conjugate for $\eta$ and asymptotic for $\psi$, hence $\mathrm{Det}(\mathbf{II}_\eta,\mathbf{II}_\psi) = 0$. Proposition~\ref{prop:rotation:field:iff} therefore implies that \(\eta\) is a rotation field of \(\psi\). Consequently, the surface element represented by \(\eta\) is a V-surface.
\end{proof}
\begin{remark}
    Note that, throughout this article, V-surfaces, by construction, have nowhere vanishing Gaussian curvatures. Consequently, developable cases such as geodesic-conjugate parametrizations of cylindrical surfaces are excluded from our considerations. 
\end{remark}
The proof of Proposition~\ref{prop:V:net} shows that, choosing a K-net representative
$\psi\colon\Omega\to\R^3$ on a K-surface $K$ induces a corresponding V-net representative $\eta\colon\Omega\to\R^3$ on a target V-surface $V$ in $\mathfrak{R}(K)$. In the light of this observation, let $(E_1,E_2)$ be the K-frame with the K-coframe $(\theta^1,\theta^2)$ and let $A$ be the rotation operator. From Lemma \& Definition~\ref{lem:parallel:tangent:space} we know that $A$ is traceless and hence $A(E_1)=a\,E_1+c\,E_2$ and $A(E_2)=b\,E_1-a\,E_2$.
Then the conjugate property of the induced V-net together with \eqref{eq:I:II:K} gives
\begin{equation*}
    0=II_\eta(E_1,E_2)=II_\psi(E_1,A(E_2))=-a\,II_\psi(E_1,E_2),
\end{equation*}
and therefore $a= 0$. Since $II_\psi(E_1,E_2)\neq0$, the rotation operator takes the form
\begin{equation}\label{eq:A:K:frame}
    A=b\,E_1\otimes\theta^2+c\,E_2\otimes\theta^1,
\end{equation}
with $bc\neq0$. Moreover, \eqref{eq:I:II:K} gives $I_\eta(E_1,E_1)=c^2$ and $I_\eta(E_2,E_2)=b^2$. Therefore, defining $(\bar E_1,\bar E_2) = (c^{-1}E_1,b^{-1}E_2)$ causes the \emph{reciprocal-parallel relation}: 
\begin{equation*}
    \d\eta(\bar E_1)=\d\psi(E_2),\qquad\qquad
    \d\eta(\bar E_2)=\d\psi(E_1).
\end{equation*}
This implies that the image of $(\bar E_1,\bar E_2)$ under $\d\eta$ consists of unit tangent vector fields to the coordinate lines of the V-net $\eta$.
Thus, we call $(\bar E_1,\bar E_2)$ the induced V-frame.
As a result of this discussion, we get the following relations between the K-frame and its induced V-frame along with their corresponding coframes:
\begin{equation}\label{eq:e1:e3:bare1:bare3}
    E_1=c\,\bar E_1,\qquad\qquad
    E_2=b\,\bar E_2,\qquad\qquad
    \bar\theta^1=c\,\theta^1,\qquad\qquad
    \bar\theta^2=b\,\theta^2.
\end{equation}
It is noteworthy, that a V-frame is not a coordinate frame. 
\begin{remark}
    In later sections, we shall see that a V-surface can also serve as the rotation field of surfaces other than K-surfaces (see Remarks~\ref{remark:V:rotation:field:V:I} and \ref{remark:V:rotation:field:V:II}), or even of more than one K-surface (see Remark~\ref{remark:infinitely:many:V:nets}). However, as one can see from the proof of Proposition~\ref{prop:V:net}, once a V-surface and a fixed V-net on it are given, there is a unique K-surface of which V is a rotation field inducing the prescribed V-net (see \cite{izmestiev,sauer}).
\end{remark}

The following theorem brings together results from \cite{EisenhartAssociate,Gambier18,sauer} in a tensor language, which we will use repeatedly in this work.
\begin{theorem}\label{theorem:I:II:III:V:surface}
    Let $V$ be a V-surface and let $K$ be a K-surface such that $V\in \mathfrak{R}(K)$. Furthermore, let $(E_1,E_2)$ be the K-frame on $K$ with $\omega\in C^\infty(\Omega,(0,\pi))$ denoting the asymptotic angle function. Then we have:
    \begin{align}
    {I}_V &= c^2\,{\theta}^{\,1}\otimes{\theta}^{\,1} + 2 bc\,\cos(\omega)\,{{\theta}}^1\odot{{\theta}}^2 + b^2\,{\theta}^{\,2}\otimes{\theta}^{\,2},\label{eq:I:sigma}\\
    {II}_V &=  \sin(\omega)\,\left(c\,{\theta}^{\,1}\otimes{\theta}^{\,1} + b\,{\theta}^{\,2}\otimes{\theta}^{\,2}\right),\label{eq:II:sigma}
    \end{align}
    where $b,c\in C^\infty(\Omega,\R^\ast)$ and they solve the following system of PDEs:
    \begin{gather}
            -\mathrm{d}b(E_1) = c\csc{(\omega)}\,\mathrm{d}\omega(E_2) + b\cot{(\omega)}\,\mathrm{d}\omega(E_1),\label{eq:codazzi:VV}\\
            -\mathrm{d}c(E_2) = c\cot{(\omega)}\,\mathrm{d}\omega(E_2) + b\csc{(\omega)}\,\mathrm{d}\omega(E_1),\label{eq:codazzi:V}\\
            \Hess_{K}\omega\,(E_1, E_2) = \sin{(\omega)},\label{eq:gauss:V}
    \end{gather}
    where \eqref{eq:codazzi:VV}-\eqref{eq:codazzi:V} are the Codazzi equations while \eqref{eq:gauss:V} is the Gauss equation with $\Hess_{K}$ being the Hessian with respect to the Levi-Civita connection of $K$.
    The Gaussian and mean curvatures of $V$ are
    \begin{equation}\label{eq:Gaussian:mean:v:net:A:net}
        {\mathrm{H}_{V} = \frac{b + c}{2\,bc\sin(\omega)}},\quad\quad\quad\quad \mathrm{K}_{V} = \frac{1}{bc}.
    \end{equation}
\end{theorem}
\begin{proof}
    For a modern proof, we refer the reader to the supplementary material.
\end{proof}

For a fixed K-net $\psi\colon\Omega\to\R^3$, Theorem~\ref{theorem:I:II:III:V:surface} assigns to every immersion rotation field of $\psi$ a pair of nowhere vanishing smooth functions $(b,c)$ satisfying \eqref{eq:codazzi:VV}-\eqref{eq:codazzi:V}. Conversely, every such pair determines an immersion $\eta\colon\Omega\to\R^3$, unique up to translation, through $\d\eta=\d\psi\circ A$ and \eqref{eq:A:K:frame}. In fact, \eqref{eq:codazzi:VV}-\eqref{eq:codazzi:V} are precisely the compatibility conditions $\d(\d\psi\circ A)=0$, while $bc\ne0$ ensures that $\eta$ is an immersion \cite[§155–156]{EisenhartTreatise}. Since $A$ is traceless, $\eta$ is a rotation field of $\psi$ \cite[Lemma 4.10]{Darja}, and represents a V-net with the fundamental forms of \eqref{eq:I:sigma} and \eqref{eq:II:sigma}.

We call $(b,c)$ a \emph{solution pair}. In coordinates with $(E_1,E_2)=(\partial_u,\partial_v)$, \eqref{eq:codazzi:VV}-\eqref{eq:codazzi:V} form a coupled linear system of first-order PDEs. Smooth \emph{characteristic initial data} $c(u,v_0)$ and $b(u_0,v)$ determine a unique local solution near $(u_0,v_0)$ (see the discussion at \cite[p.~103,125--128]{sauer}).

\begin{remark}\label{remark:bc:sign}
    Since $b,c \neq 0$, the sign of $bc$ is constant on $\Omega$. By \eqrefi{eq:Gaussian:mean:v:net:A:net}{2}, the space of \emph{solution pairs} of \eqref{eq:codazzi:V} therefore split into two components
    \begin{equation*}
        (b^+,c^+) \colon\quad b^+c^+>0,
        \qquad\qquad
        (b^-,c^-) \colon\quad b^-c^-<0,
    \end{equation*}
    where each correspond to V-surfaces of positive and negative Gaussian curvature respectively.
\end{remark}
\subsection{Isometric Deformation of V-surfaces vs. Spectral Deformation of K-surfaces}
The following definition and lemma are a slight variation of \cite[Thm~10]{BobenkoQuaternion}.
\begin{definition}
    Let $(\psi,\Omega)$ be a K-surface with $\omega$ as the asymptotic angle function.
    Its \emph{spectral deformation} is a K-surface $(\psi^t,\Omega)$ with $t \in (-\epsilon, \epsilon)$ and $\psi^0 = \psi$ such that $II^t = II^0$ and $\omega^t = \omega^0$.
\end{definition}
The family of spectral deformations, $\{K^t\}$ with $K^t := (\psi^t,\Omega)$, is called the \emph{associate family}.
\begin{lemma}\label{lem:spectral:deformation}
    Let $K^t$ and $K$ be two K-surfaces with $(E^t_1,E^t_2)$ and $(E_1,E_2)$ as their corresponding K-frames. $K^t$ is a spectral deformation of $K$ if and only if 
    \begin{equation}\label{eq:spectral:def}
        E_1^t=\lambda^{-1}E_1,\qquad\qquad E^t_2=\lambda E_2,
    \end{equation}
    where $\lambda=\exp(t)$.
\end{lemma}
\begin{proof}
    For the proof, we refer the reader to the supplementary material.
\end{proof}
From \cite{EisenhartTreatise} it is known that every V-net admits a $1$-parameter family of isometric deformations preserving the V-net property (see \cite[Sec.~2.5]{izmestiev} for a modern treatment). 
Below, we show that the isometric deformation of $V$ is tied to the spectral deformation of the corresponding K-surface (see \cite{sauer}).
\begin{theorem}\label{theorem:isometric:deformation}
    Let $V^0$ and $K^0$ be a V-surface and a K-surface respectively with $V^0\in \mathfrak{R}(K^0)$. Fix the V-net on $V^0$ induced by the K-net of $K^0$. Then a V-surface $V$ belongs to the isometric deformation family of $V^0$ preserving the aforementioned V-net if and only if its corresponding K-surface $K$, inducing the same V-net, belongs to the associate family of $K^0$.
    
    Equivalently, the deformation parameters of the two $1$-parameter families can be chosen compatibly so that the families are written as $\{V^t\}$ and $\{K^t\}$ with $V^0$ and $K^0$ as their respective members at $t=0$, and $V^t\in \mathfrak{R}(K^t)$ for every admissible $t$.
\end{theorem}
 
\begin{proof}
    For the proof, we refer the reader to the supplementary material.
\end{proof}
\section{Identification of the Corresponding K-surfaces}\label{sec:reciproal:parallel:K:surfaces}

\subsection{Intrinsic Surfaces of Revolution \& Helicoidal Surfaces}
\begin{definition}\label{def:geodesic:parallel:surf}
    A surface element $P = (\psi,\Omega)$ possessing a nowhere vanishing Killing field $Y\in\mathfrak{X}(\Omega)$ is called an \emph{intrinsic surface of revolution}, for short ISR, relative to $Y$.
\end{definition}
On an ISR, let $E_1$ be the unit vector field orthogonal to $Y$ such that $(E_1,Y)$ is positively oriented. Since $Y$ is Killing, $\mathcal{L}_{E_1}Y=0$, so $(E_1,Y)$ is a coordinate frame. Set $G=I(Y,Y)$ and $E_2=Y/\sqrt{G}$. We call $(E_1,E_2)$ the \emph{ISR frame} adapted to $Y$, and its dual $(\theta^1,\theta^2)$ the \emph{ISR coframe}.

Historically, the ISR were studied by Weingarten in \cite{Weingarten} and much later in \cite{Smyth} for CMC surfaces.
Recently, this class of surfaces are investigated again in theory, see \cite{Weber,Kaya}, and applications \cite{geodesicparallelnets}. The philosophy behind choosing such name is that intrinsic surfaces of revolution are \emph{isometric to surfaces of revolution} \cite{EisenhartTreatise}.

An important subclass of ISR is the class of \emph{helicoidal surfaces} which are a natural generalization of \emph{helicoids} and \emph{surfaces of revolution} (see \cite[page~103]{DocarmoDifferential}). These surfaces are generated by $1$-parameter subgroups of $\mathrm{SE}(3)$ called screw motions. A screw motion is generated by $\bar{\xi} \in \mathfrak{se}(3)$ which is of the form $\bar{\xi}_p = \eta \times p + \zeta$ for every point $p$ of $\mathbb{R}^3$ and with $\eta, \zeta \in \mathbb{R}^3$. In the special case where $\zeta = 0$, the screw motion is called rotation.
\begin{definition}
    A surface in Euclidean space $\mathbb{R}^3$ is called a \emph{helicoidal surface} if it is invariant under a screw motion. In the special case where the screw motion is a rotation the aforementioned surface is called a surface of revolution.
\end{definition}
In our surface element setting, the terms helicoidal surface, surface of revolution and helicoid are used in a local sense and refer to regular portions of the corresponding surfaces.

The following theorem, which is a reformulated version of Exercise 23 in \cite[p.~188]{EisenhartTreatise} in a coordinate-free shape, facilitates the identification of helicoidal surfaces.
\begin{theorem}\label{theorem:B:nets}
    A surface element $P=(\psi,\Omega)$ is a helicoidal surface relative to $Y$ if and only if there exists a nowhere vanishing vector field $Y\in\mathfrak{X}(\Omega)$ such that
    \begin{equation}\label{eq:I:II:killing}
        \mathcal{L}_Y I = 0,\quad\quad\quad\quad
        \mathcal{L}_Y II = 0.
    \end{equation}
    If $(E_1,E_2)$ is the ISR frame adapted to $Y$, then $P$ is a surface of revolution if and only if $II(E_1,E_2)=0$, and is a helicoid if and only if $II(E_1,E_1)=II(E_2,E_2)=0$.
\end{theorem}
We call a vector field $Y$ satisfying \eqref{eq:I:II:killing} a \emph{$I$-$II$-Killing field}. 

\subsection{Amsler Surfaces}
There are surfaces for which only \eqrefi{eq:I:II:killing}{2} holds. We call a vector field $Y$ satisfying \eqrefi{eq:I:II:killing}{2} a \emph{$II$-Killing field}. The aforementioned class contains a large variety of surfaces, including the minimal surfaces. However, in this article, we are interested in a certain subclass of pseudospherical surfaces with this property, namely the \emph{Amsler surfaces}. To identify these surfaces, we propose the following definition, which is a coordinate-free equivalent of the definition given in \cite[§ 3.7.2]{Popov}, tailored to our surface element approach:

\begin{definition}\label{def:K:surf:patch:Amsler}
    We call a K-surface $(\psi,\Omega)$ an Amsler surface if it admits a nowhere vanishing $II$-Killing field $Y\in\mathfrak{X}(\Omega)$ such that $\mathcal{L}_Y I \neq 0$ and $\d\omega(Y) = 0$ with $\omega$ as the asymptotic angle function.
\end{definition}

In Section~\ref{sec:2nd:kind}, we review Amsler surfaces, including their well-known regular subclass in greater detail.

\subsection{Alignable V-surfaces \& their Corresponding K-surfaces}
\begin{definition}[\cite{AlignablePellis}]\label{def:smooth:alignability}
    Let \(\psi:\Omega\to\mathbb R^3\) be a net. For every coordinate rectangle
    \begin{equation*}
        \mathscr{R} = \left[u_0,u_1\right]\times\left[v_0,v_1\right],\qquad u_0<u_1,\quad v_0<v_1,
    \end{equation*}
    the parametrized boundary $\psi|_{\partial \mathscr{R}}$ is called a \emph{net loop}. Label its corners $A,B,C,D$ in cyclic order, starting at $(u_0,v_0)$. The loop is alignable if $l_{AB} + l_{BC} = l_{AD} + l_{DC}$, where the lengths are measured along the corresponding coordinate lines. The net is \emph{alignable} if every such net loop is alignable.
\end{definition}

In an alignable net by controlling the angle at one of the vertices and changing it the surface locally goes through a scissor motion at all the vertices resulting it globally to collapse to a curve\footnote{In physical applications, this curve is represented by a planar strip.}. To determine whether a given net is alignable, the following criterion is introduced in \cite{AlignablePellis}:
\begin{lemma}\label{lem:mahak:davide}
    Let $\psi$ be a net with the frame $(E_1,E_2)$ and dual coframe $(\theta^1, \theta^2)$. Then $\psi$ is an alignable net if and only if $\mathrm{d}\theta^1 + \mathrm{d}\theta^2 = 0$.
\end{lemma}
In this article, our focus is on V-surfaces which possess an alignable V-net. 
\begin{definition}
    A V-surface is called alignable if it supports an alignable V-net.
\end{definition}

Now, a tailored version of Lemma~\ref{lem:mahak:davide} for our purpose is as follows:
\begin{lemma}\label{lem:alignability:cond:II}
    Let $\eta\colon\Omega\to\R^3$ be a V-net corresponding to a solution pair $(b, c)$ of \eqref{eq:codazzi:VV}-\eqref{eq:codazzi:V}. Then we have the following statement:
    \begin{equation}\label{eq:alignability:cond}
        \text{1.  $\eta$ is alignable}\quad\Leftrightarrow\quad \text{2.  }\mathrm{d}b(E_1) = \mathrm{d}c(E_2) \quad\Leftrightarrow\quad \text{3.  }b\,\mathrm{d}\omega(E_1) = c\,\mathrm{d}\omega(E_2).
    \end{equation}
\end{lemma}

\begin{proof}
    Let $(b, c)$ be a solution pair corresponding to an alignable V-net.
    \item[$(1)\!\Leftrightarrow\!(2)$] Using Lemma~\ref{lem:mahak:davide} and the relations \ref{eq:e1:e3:bare1:bare3} we get that $\eta$ is an alignable net if and only if
    \begin{equation*}
       0 = \d\bar{\theta}^1 + \d\bar{\theta}^2 = \left( \d b(E_1) - \d c(E_2)\right)\,\theta^1\wedge\theta^2. 
    \end{equation*}
    \item[$(2)\!\Leftrightarrow\!(3)$] Substituting \eqrefi{eq:alignability:cond}{2}/\eqrefi{eq:alignability:cond}{3} into \eqref{eq:codazzi:VV}-\eqref{eq:codazzi:V} results in
    \eqrefi{eq:alignability:cond}{3}/\eqrefi{eq:alignability:cond}{2}.
\end{proof}
Having developed the necessary tools, we arrive at the first major result of this article.
\begin{theorem}\label{theorem:main}
    $V$ is an alignable V-surface if and only if there exists a K-surface $K$, with $V\in \mathfrak{R}(K)$ such that
    \begin{itemize}
        \item $K$ is either a helicoidal K-surface or an Amsler surface,
        \item if $Y$ is the $II$-Killing field of $K$ then it bisects the V-frame induced on $V$.
    \end{itemize}
\end{theorem}
\begin{proof}
    \textit{Necessity.} Using \eqrefi{eq:alignability:cond}{3} in the \emph{Codazzi equations} of \eqref{eq:codazzi:VV}-\eqref{eq:codazzi:V}, and recalling that $b,c\neq0$, we obtain
    \begin{equation}\label{eq:codazzi:II}
        \begin{aligned}
            -\mathrm{d}b(E_1)
            &= b\,\mathrm{d}\omega(E_1)\,
            \Bigl(\csc(\omega)+\cot(\omega)\Bigr)
            \ \implies
            &\mathrm{d}\omega(E_1)\cot\left(\frac{\omega}{2}\right)
            &= -\frac{1}{b}\,\mathrm{d}b(E_1),\\
            -\mathrm{d}c(E_2)
            &= c\,\mathrm{d}\omega(E_2)\,
            \Bigl(\csc(\omega)+\cot(\omega)\Bigr)
            \ \implies
            &\mathrm{d}\omega(E_2)\cot\left(\frac{\omega}{2}\right)
            &= -\frac{1}{c}\,\mathrm{d}c(E_2).
        \end{aligned}
    \end{equation}
    Simplifying \eqref{eq:codazzi:II} further gives
    \begin{equation}\label{eq:codazzi:III}
        \begin{aligned}
        \mathrm{d}\left(\log|b|\right)(E_1)
        &=
        -\mathrm{d}\left(\log\left(1-\cos(\omega)\right)\right)(E_1),\\
        \mathrm{d}\left(\log|c|\right)(E_2)
        &=
        -\mathrm{d}\left(\log\left(1-\cos(\omega)\right)\right)(E_2).
        \end{aligned}
    \end{equation}
    Setting $\varrho:=\log\left(\frac{|b|}{|c|}\right)$ equations \eqref{eq:codazzi:III} give
    \begin{equation}\label{eq:cod}
        \begin{aligned}
            E_1(\log|c|)+E_1(\varrho)
            &=
            -E_1\left(\log(1-\cos\omega)\right),\\
            E_2(\log|c|)
            &=
            -E_2\left(\log(1-\cos\omega)\right).
        \end{aligned}
    \end{equation}
    Taking the derivative $E_2$ from the first and $E_1$ from the second of the equations of \eqref{eq:cod} and putting the r.h.s equal gives ${E_2} {E_1} \varrho = 0$. Integrating along the commuting flows of $E_1$, $E_2$ yields $\varrho = \log(\alpha) - \log(\beta)$ where $\alpha$ and $\beta$ are locally non-vanishing smooth functions with
    \begin{equation}\label{eq:E1:beta:E3:alpha}
        E_1(\beta) = E_2(\alpha) = 0.
    \end{equation}
    This gives $\frac{|b|}{|c|}=\frac{\alpha}{\beta}$, which, together with \eqrefi{eq:alignability:cond}{3}, implies
    \begin{equation}\label{eq:Y:omega}
        \alpha\,\mathrm{d}\omega(E_1)
        -
        \epsilon\beta\,\mathrm{d}\omega(E_2)
        =
        \mathrm{d}\omega(Y)
        =
        0,
    \end{equation}
    where $\epsilon:=\sign(bc)=\sign(\mathrm K_V)$ and $Y:=\alpha E_1-\epsilon\beta E_2$.
    Now, consider the vector field $X := \left(\alpha E_1 + \epsilon\beta E_2\right)$. Since $\omega$ is a solution of the sine-Gordon equation, we rewrite the aforementioned equation according to these two vector fields and since $\d\omega(Y) = 0$ we get:
    \begin{equation*}
        \sin{(\omega)} = \Hess_K\omega\,(E_1,E_2) = E_1(E_2\omega) = \frac{1}{4\epsilon\alpha\beta}{(X^2-Y^2)}\,\omega = \frac{1}{4\epsilon\alpha\beta}X^2\,\omega.
    \end{equation*}
    Now, since $X$ and $Y$ commute (i.e. $\mathcal{L}_XY = 0$), applying $Y$ to the sides of the above equation results in $0 = Y\left(\alpha\beta\right)$. This, by resorting to \eqref{eq:E1:beta:E3:alpha}, yields
    \begin{equation}\label{eq:E1:alpha:E3:beta:const}
        E_1(\alpha) = \epsilon\,E_2(\beta) = \const =: \mu,
    \end{equation}
    where the constant $\mu$ is a consequence of separability of $\alpha$ and $\beta$. \eqref{eq:E1:beta:E3:alpha}, \eqref{eq:Y:omega} and \eqref{eq:E1:alpha:E3:beta:const} imply
    \begin{equation}\label{eq:LYI:LYII}
            \mathcal{L}_Y I = 2\mu\,(\theta^1\otimes\theta^1 - \theta^2\otimes\theta^2),\qquad\qquad
            \mathcal{L}_Y II = 0.
    \end{equation}
    Consequently, in the light of Theorem~\ref{theorem:B:nets} and Definition~\ref{def:K:surf:patch:Amsler}, \eqref{eq:Y:omega} and \eqref{eq:LYI:LYII} along with $\mu = 0$ / $\mu \neq 0$ imply that $K$ is a helicoidal K-surface / an Amsler surface. Now, since $\frac{|b|}{|c|}=\frac{\alpha}{\beta}$ and $\epsilon=\sign(bc)=\sign(\mathrm K_V)$, we obtain
    \begin{equation}\label{eq:b:beta:epsilon:alpha:c}
        \beta\,b = \epsilon\,\alpha\,c.
    \end{equation}
    Since $Y = \alpha E_1 - \epsilon\beta E_2 = \alpha c\bar{E}_1 - \epsilon\beta b\bar{E}_2$, \eqref{eq:b:beta:epsilon:alpha:c} implies that $Y$ bisects the V-frame.
    
    \textit{Sufficiency.} Let $Y = \alpha E_1 - \beta E_2$ be the $II$-Killing field of $K$ for some $\alpha,\beta\in C^\infty(\Omega)$. Then using \eqref{eq:e1:e3:bare1:bare3} we have 
    \begin{equation}\label{Y:bar:E}
        Y = \alpha E_1 - \beta E_2 = \alpha\,c\,\bar{E}_1 - \beta\,b\,\bar{E}_2 = \rho (\bar{E}_1 - \epsilon \bar{E}_2).
    \end{equation}
    where $\rho = |\alpha\,c| = |\beta\,b|$ and $\epsilon =\pm 1$ and the last equation is obtained using the fact that $Y$ bisects the V-frame.
    We first observe that $\d\omega(Y) = 0$. If $K$ is an Amsler surface, this follows from Definition~\ref{def:K:surf:patch:Amsler}. If $K$ is helicoidal relative to $Y$, then Theorem~\ref{theorem:B:nets} gives $\mathcal{L}_Y I = \mathcal{L}_Y II = 0$ which upon expansion using Cartan's magic formula (see \cite[p.~372]{Lee}) implies $\d\omega(Y) = 0$. Now, since $\rho \neq 0$, \eqref{Y:bar:E} implies 
    \begin{equation*}
        0 = \d\omega(\bar{E}_1 - \epsilon \bar{E}_2) = \frac{1}{bc}\d\omega(b\,E_1 - \epsilon c E_2).
    \end{equation*}
    For $\epsilon = 1$, Lemma~\ref{lem:alignability:cond:II} immediately implies that $V$ is alignable. For $\epsilon = -1$, reverse the asymptotic direction $E_2$. Then we have \(E_2'=-E_2\), \(b'=-b\), \(c'=-c\), and \(\omega'=\pi-\omega\), so the preceding relation becomes $b^\prime \d\omega^\prime(E_1)=c^\prime \d\omega^\prime(E_2)$ and again by Lemma~\ref{lem:alignability:cond:II} $V$ is alignable.
\end{proof}

\begin{definition}\label{def:kind}
    An alignable V-surface (V-net) is called of the 
    \begin{itemize}
        \item first kind if it is arising as the rotation field of a helicoidal K-surface,
        \item second kind if it is arising as the rotation field of an Amsler surface.
    \end{itemize}
    Furthermore, we call an alignable V-surface relative to $Y$ when $Y$ is the $II$-Killing field of its corresponding K-surface. Finally, we call an alignable V-surface \emph{positive-alignable} or \emph{negative-alignable} according to whether its Gaussian curvature is positive or negative, respectively (see Remark~\ref{remark:bc:sign}).
\end{definition}

From the proof of Theorem~\ref{theorem:main} we infer the following remark:
\begin{remark}\label{remark:Y}
    Let $K$ be a helicoidal K-surface or an Amsler surface with the asymptotic angle function $\omega$. Then, w.l.o.g its $II$-Killing field is of the form
    \begin{equation}\label{eq:Y}
            Y = \alpha E_1 - \beta E_2,
    \end{equation}
    where \(E_2(\alpha)=E_1(\beta)=0\) and \(E_1(\alpha)=E_2(\beta)=\mu\in\mathbb{R}\). Here \(\mu=0\) corresponds to the helicoidal case, while \(\mu\neq 0\) corresponds to the Amsler case. Furthermore, $\d\omega(Y) = 0$.
\end{remark}

\begin{corollary}\label{corollary:alignble:is:geodesic:parallel:surf}
    Let $V$ be an alignable V-surface relative to $Y$. Then $V$ is an ISR relative to $Y$. Particularly, if $V$ is of the first kind then it is a helicoidal surface relative to $Y$.
\end{corollary}
\begin{proof}
    Let $V$ be an alignable V-surface relative to $Y$ as described in Remark~\ref{remark:Y}. Since $Y$ bisects the V-frame we get $Y = \alpha E_1 - \beta E_2 = \alpha c\bar{E}_1 - \beta b\bar{E}_2$ and hence $\alpha c = \epsilon\beta b$ with $\epsilon = \pm 1$.
    Now, applying $E_2$ (respectively $E_1$) to the latter and using \eqrefi{eq:alignability:cond}{2} combined with Remark~\ref{remark:Y} gives
    \begin{equation}\label{eq:Yb:Yc:one}
        \begin{aligned}
            Y(b) = \mu b,\quad
            &Y(c) = -\mu c,\quad
            &\mathcal{L}_Y E_1 &= -\mu E_1,\quad
            \mathcal{L}_Y E_2 = \mu E_2,\\
            &\mathcal{L}_Y \theta^1 = \mu \theta^1,\quad
            &\mathcal{L}_Y \theta^2 &= -\mu \theta^2.
        \end{aligned}
    \end{equation}
    Taking Lie derivatives of \eqref{eq:I:sigma} and \eqref{eq:II:sigma} using \eqref{eq:Yb:Yc:one} and Remark~\ref{remark:Y} yields
    \begin{equation*}
        \begin{aligned}
        \mathcal{L}_Y I_V
        &= 2c\bigl(Y(c)+\mu c\bigr)\,\theta^1\otimes\theta^1
        + 2Y(bc\cos(\omega))\,\theta^1\odot\theta^2
        + 2b\bigl(Y(b)-\mu b\bigr)\,\theta^2\otimes\theta^2
        = 0,\\
        \mathcal{L}_Y II_V
        &= c\left(Y(\omega)\,\cos(\omega) + \mu\sin(\omega)\right)\theta^1\otimes\theta^1 +
        b\left(Y(\omega)\,\cos(\omega) - \mu\sin(\omega)\right)\theta^2\otimes\theta^2.
        \end{aligned}
    \end{equation*}
    Evidently, $V$ is an ISR. If $V$ is an alignable of the first kind then the corresponding K-surface is a helicoidal K-surface and hence $\mu = 0$, making $V$ by the above equation, a helicoidal surface.
\end{proof}
\section{Alignable V-surfaces of the First Kind}\label{sec:1st:kind}
In the {supplementary material}, we discussed the problem of deriving immersion formulas for isometric deformations of V-nets and explained why this task is generally \emph{difficult}, even when the immersion of one member of the isometric family is already known. Despite this observation, we show that in the case of alignable V-surfaces of the first kind, resorting to a kinematic interpretation enables us to obtain their full isometric family. To this end, we take a closer look at isometries mapping one helicoidal surface to another.
\begin{definition}\label{def:Bour:deformation}
    Let $\psi^\tau\colon\Omega\to\R^3$ be an isometric deformation of a helicoidal surface parametrized by $\psi\colon\Omega\to\R^3$ relative to $Y$. We call $\psi^\tau$ a Bour isometric deformation if $\mathcal{L}_Y II^\tau = 0$ for all admissible $\tau$.

    Accordingly, a surface element $Q$ is called Bour isometric to a helicoidal surface $P$ if they admit representatives $(\psi^{\tau},\Omega)$ and $(\psi,\Omega)$ such that $\psi^\tau$ is a Bour isometric deformation of $\psi$.
\end{definition}
The following proposition is a variation of Bour’s method at the level of fundamental forms (see \cite{DoCarmoHelicoidal}).
\begin{proposition}\label{prop:bour:family}
    Let $P$ be a helicoidal surface relative to $Y$ and denote its adapted ISR frame by $(E_1,E_2)$. Then $P$ is contained in a $2$-parameter family of Bour isometric surfaces $\{P^{s,t}\}$, called Bour family, where the members of this family are characterized by the following fundamental forms in the ISR coframe:
    \begin{equation}\label{eq:I:II:u}
        I = \theta^1\otimes\theta^1 + \theta^2\otimes\theta^2,\qquad
        {II}^{s,t} = L\,\theta^1\otimes\theta^1 + 2\,M\,\theta^1 \odot \theta^2 + N\,\theta^2\otimes\theta^2,
    \end{equation}
    with the components of the second fundamental form being:
    \begin{equation}\label{eq:L:N:u}
        L=\pm\frac{G_1^2+4t^2-2G\,G_{11}}{2G\,\sqrt{\,4sG-4t^2-G_1^2\,}},\qquad
        M=\frac{t}{G},\qquad
        N=\pm\frac{1}{2G}\sqrt{\,4sG-4t^2-G_1^2\,},
    \end{equation}
    where $G = I(Y,Y)$, $G_1 := E_1(G)$ and $G_{11}:=E_1(E_1(G))$. Furthermore, $P^{s,0}$ are surfaces of revolution and their ISR frame directions are principal directions.
\end{proposition}

\begin{proof}
    See Appendix~\ref{appendix:Bour:I:II}.
\end{proof}

\begin{definition}
    Let $\psi^\tau\colon\Omega\to\R^3$ be a $1$-parameter Bour isometric deformation relative to $Y$. 
    Let $\chi^g$ and $\kappa^n$ be the geodesic torsion and normal curvature of $Y$-lines. Then $\psi^\tau$ is called a wrapping isometric deformation if it preserves $\chi^g$ and a sliding isometric deformation if it preserves $(\chi^g)^2 + (\kappa^n)^2$. 
\end{definition}
\begin{figure*}[t!]
    \centering
    \begin{subfigure}[b]{.32\textwidth}
        \centering
        \includegraphics[height = 35mm]{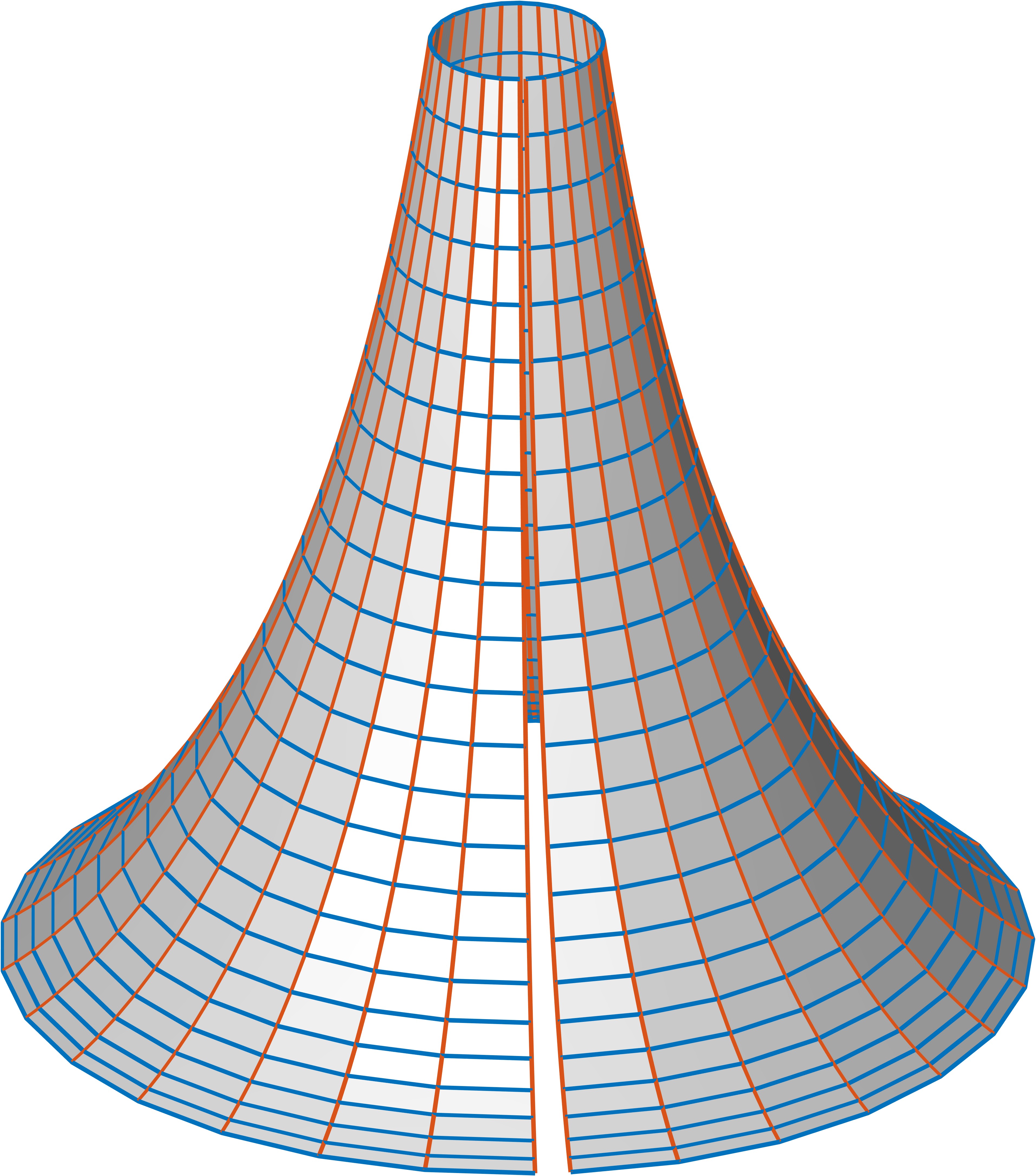}
    \end{subfigure}\hfill
    \begin{subfigure}[b]{.32\textwidth}
        \centering
        \includegraphics[height = 35mm]{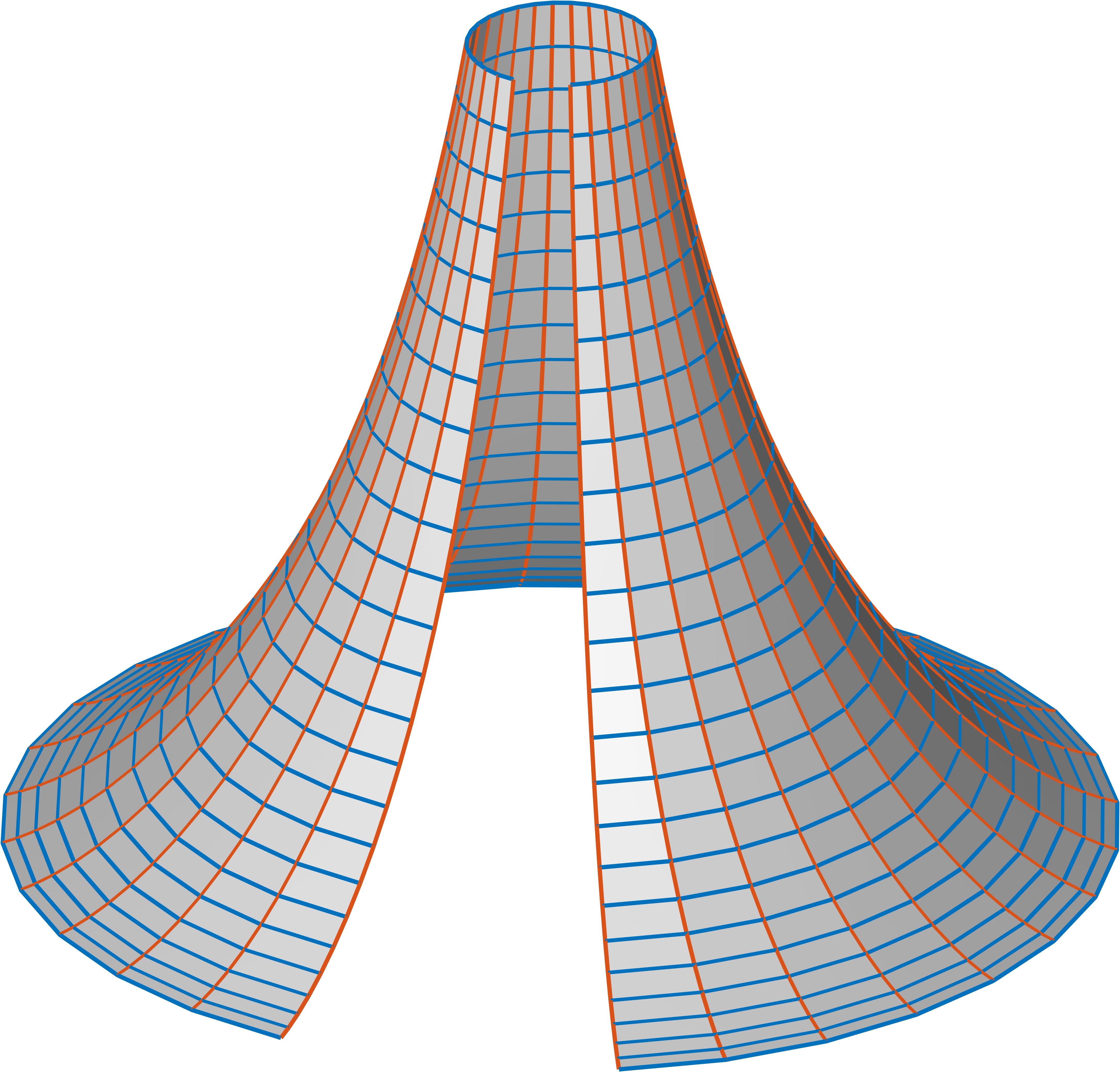}
    \end{subfigure}\hfill
    \begin{subfigure}[b]{.32\textwidth}
        \centering
        \includegraphics[height = 40mm]{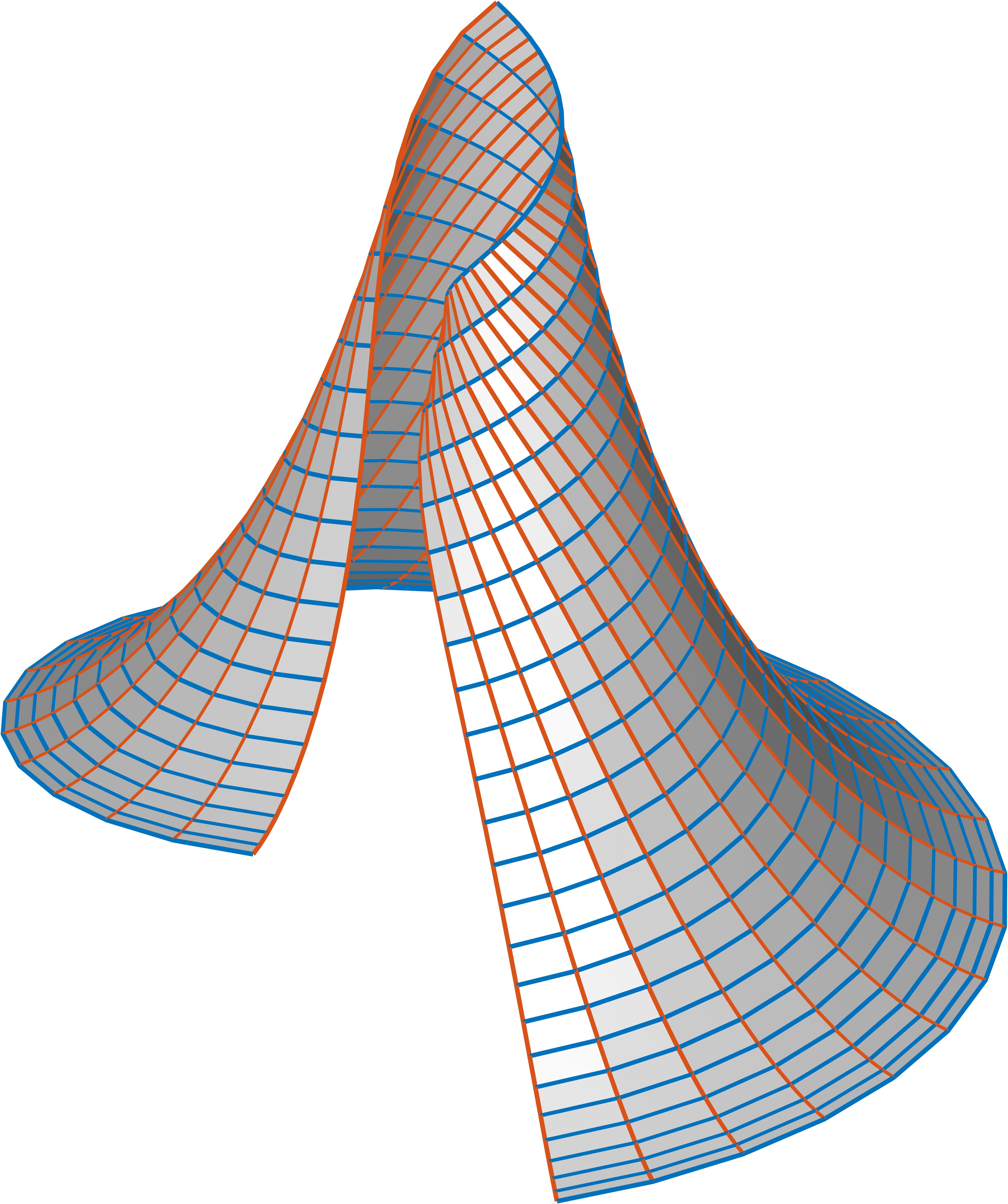}
    \end{subfigure}
    \caption{Illustration of the Bour isometric deformation of a parabolic pseudosphere. From left to middle, the deformation is realized by wrapping and from middle to right, it is realized by sliding.}
    \label{fig:pseudosphere}
\end{figure*}
A \emph{wrapping isometric deformation} preserves the twisting of the helicoidal surface normal about the tangent direction of $\d\psi^\tau(Y)$ while a \emph{sliding isometric deformation} preserves the total rate of change of the surface normal along the aforementioned direction. 
\begin{remark}
    Within the Bour family, wrapping isometric deformations are precisely the $1$-parameter deformations with $t$ constant, while sliding isometric deformations are precisely those with $s$ constant. They are therefore the subfamilies $\{P^{s,t_0}\}$ and $\{P^{s_0,t}\}$ respectively, where $(s_0,t_0)$ are fixed and the varying parameters take admissible values (see Fig.~\ref{fig:pseudosphere}).
\end{remark}
From now on, we refer to the instantaneous velocity fields of the sliding and wrapping isometric deformations by \emph{instantaneous sliding} and \emph{wrapping} respectively. 

\begin{proposition}\label{prop:categorizing:IIDs}
    Let $(\psi,\Omega)$ be a surface of revolution with nowhere vanishing Gaussian curvature and $\psi$ as its curvature line net. Let $\xi\colon\Omega\to\R^3$ be an IID along $(\psi,\Omega)$ with the corresponding rotation field $\eta$. Then the following statements are equivalent:
    \begin{enumerate}
        \item $\xi$ is locally an instantaneous sliding (wrapping) of $\psi$,
        \item $\eta$ is a conjugate (asymptotic) net.
    \end{enumerate}
    Finally, if $\xi$ and $\tilde\xi$ are two such IIDs of the same type (either wrapping or sliding), then 
    \begin{equation}\label{eq:rot:comp}
        \xi=C\tilde{\xi}+r\times\psi+q,\qquad
        \eta = C\tilde\eta+r,
    \end{equation}
    where $C\in \R^\ast$ and $r,q\in\R^3$ are constant. 
\end{proposition}

\begin{proof}
    See Appendix~\ref{appendix:categorizing:IIDs}.
\end{proof}
Before going forward to the main result we consider the following two lemmas:
\begin{lemma}\label{lem:K:surf:revol}
    Let $K$ be a helicoidal K-surface relative to $Y$. Then $K$ is a K-surface of revolution if and only if $Y$ bisects its K-frame.
\end{lemma}
\begin{proof}
    Let $K$ be a helicoidal K-surface relative to $Y$ with the K-frame $(E_1,E_2)$. W.l.o.g let $Y$ be as in Remark~\ref{remark:Y}. Then the ISR frame $(E_3,E_4)$ adapted to $Y$ is 
    \begin{equation*}\label{eq:E5:E6}
        (E_3,E_4) \propto \Bigl((\alpha\cos{(\omega)} - \beta)\,E_1 - (\alpha - \beta\cos{(\omega)})\,E_2,\alpha E_1 - \beta E_2\Bigr),
     \end{equation*}
    with $\d\alpha = \d\beta = 0$ and $\d\omega(Y) = 0$. 
    Now, if $K$ is a surface of revolution we get $0 = II(E_3,E_4) \propto \sin(\omega)\,(\beta^2 - \alpha^2)$ which implies $\beta = \pm\alpha$. Therefore, $Y$ bisects the K-frame.
    
    Conversely, let $Y$ be as in Remark~\ref{remark:Y}. Then since it bisects the K-frame we get $\alpha = \pm\beta$. This implies that the ISR frame is $(E_3,E_4)\propto(E_2 + E_1,E_2 - E_1)$. Now, using \eqrefi{eq:I:II:psi}{2} we get $II(E_3,E_4) = 0$ and hence by Theorem~\ref{theorem:B:nets} $K$ is a K-surface of revolution.   
\end{proof}

The following lemma is a rigorous restatement from \cite{Gambier18}:

\begin{lemma}\label{lem:three:equivalencies}
    Let $\{K^t\}$ be an associate family of K-surfaces. Then the following statements are equivalent:
    \begin{enumerate}
        \item A member of $\{K^t\}$ is a heliocidal K-surface relative to $Y$,
        \item All members of $\{K^t\}$ are heliocidal K-surfaces relative to $Y$,
        \item $\{K^t\}$ contains a unique K-surface of revolution relative to $Y$.
    \end{enumerate}
\end{lemma}
\begin{proof}
    $(1)\!\Rightarrow\!(2)$: Assume that $K^0$ is helicoidal relative to $Y$. Then by Remark~\ref{remark:Y} we have $Y = \alpha E_1 - \beta E_2$ where $\alpha$ and $\beta$ are constants and $\d\omega(Y) = 0$. For any $K^t$ in the associate family we have the relations \eqref{eq:spectral:def}. Hence, in the K-frame of $K^t$, we get $Y = \lambda\alpha E^t_1 - \lambda^{-1}\beta E^t_2$. The coefficients are again constant, and therefore $Y$ commutes with $E_1^t$ and $E_2^t$. This in addition to the fact that $\omega = \omega^t$ and hence $\d\omega(Y) = \d\omega^t(Y) = 0$ gives $\mathcal{L}_Y I^t = \mathcal{L}_Y II^t = 0$.
    
    $(1)\!\Rightarrow\!(3)$: Let us denote $Y$ in the frame $(E^t_1,E^t_2)$ as before. Now, Lemma~\ref{lem:K:surf:revol} implies $|\lambda\alpha|=|\lambda^{-1}\beta|$. Since $\lambda>0$, there is exactly one such $\lambda$. Hence, the associate family contains a unique K-surface of revolution relative to $Y$.
    
    Finally, $(2)\Rightarrow(1)$ and $(3)\Rightarrow(1)$ are immediate. 
\end{proof}

The following theorem is a kinematic interpretation of Theorem~\ref{theorem:main} for the specific case of the first kind. As we will see in Section~\ref{sec:immersions}, this interpretation allows us to build formulas of immersion for this class of V-surfaces.

\begin{theorem}\label{theorem:main:first}
    Let $\{V^t\}$ be the isometric family preserving a fixed V-net, and let $\{K^t\}$ be the corresponding associate family. Throughout, alignability is understood with respect to the V-net induced by $K^t$. Then the following statements are equivalent:
    \begin{enumerate}
        \item A member of $\{V^t\}$ is positive- (negative-) alignable of the first kind,
        \item Every member of $\{V^t\}$ is positive- (negative-) alignable of the first kind,
        \item There exists an admissible parameter $t_0$ such that:
        \begin{itemize}
            \item $K^{t_0}$ is a K-surface of revolution,
            \item Locally, $V^{t_0}$ is represented by the rotation field of an instantaneous sliding (wrapping) of $K^{t_0}$.
        \end{itemize}
    \end{enumerate}
    Moreover, whenever these conditions hold, $\{V^t\}$ is a $1$-parameter subfamily of the Bour family of $V^{t_0}$, relative to the $I$-$II$-Killing field of $K^{t_0}$.
\end{theorem}
\begin{proof}
    $(1)\!\Rightarrow\!(2)$: Since the isometric deformation preserves the first fundamental form, it preserves the alignability. This together with Lemma~\ref{lem:three:equivalencies} and the invariance of Gaussian curvature under isometries, shows that every $V^t$ is positive- (negative-) alignable of the first kind.
    \item[$(2)\!\Rightarrow\!(1)$:] The implication is immediate.
    \item[$(1)\!\Rightarrow\!(3)$:] By Lemma~\ref{lem:three:equivalencies}, the associate family $\{K^t\}$ contains a unique K-surface of revolution relative to $Y$. Denote this member by $K^{t_0}$ and its corresponding V-surface in $\{V^t\}$ by $V^{t_0}$. The isometric deformation preserves the induced V-net and its alignability. Thus $V^{t_0}$ is alignable of the first kind, with the same Gaussian curvature. Let $(E_1,E_2)$ be the K-frame of $K^{t_0}$ and let $(\bar E_1,\bar E_2)$ be the induced V-frame of $V^{t_0}$. Since $K^{t_0}$ is a surface of revolution relative to $Y$, the vector field $Y$ bisects its K-frame (see Lemma~\ref{lem:K:surf:revol}). Moreover, Theorem~\ref{theorem:main} implies that $Y$ bisects the induced V-frame as well. Remark~\ref{remark:Y} and the relations \eqref{eq:e1:e3:bare1:bare3} therefore give $c=\varepsilon\,b$ with $\varepsilon=\operatorname{sgn}(bc)$. Now, consider the normalized frame $(E_3,E_4)$ which bisects $(E_1,E_2)$. There, the second fundamental form of $V^{t_0}$ becomes 
    \begin{equation}\label{eq:three:II}
        II^{t_0} = 
        L\,\left((1+\varepsilon)\,\left(\mu_1^2\,\theta^3\otimes\theta^3 + \mu_2^2\,\theta^4\otimes\theta^4\right) + 
        2\,(\varepsilon - 1)\,\mu_1\mu_2\,\theta^3\odot\theta^4\right),
    \end{equation}
    where $\mu_1 = \frac{1}{2}\sec(\frac{\omega}{2})$, $\mu_2 = \frac{1}{2}\csc(\frac{\omega}{2})$ and $L = b\,\sin{(\omega)}$. 
    The frame $(E_3,E_4)$ bisects the K-frame $(E_1,E_2)$ of $K^{t_0}$ and hence correspond to its principal directions. 
    Now, on a sufficiently small simply connected domain $U$, choose local compatible nets $\psi,\eta\colon U\to\mathbb R^3$ adapted to $(E_3,E_4)$ for $K^{t_0}$ and $V^{t_0}$ respectively. Let $\xi\colon U\to\mathbb R^3$ be the smooth IID along $\psi$ whose rotation field is $\eta$.
    Then, the following alternatives take place
    \begin{enumerate}
        \item If $\varepsilon=1$, $\eta$ is a conjugate net and consequently by Proposition~\ref{prop:categorizing:IIDs}, $\xi$ is an instantaneous sliding of $(\psi,U)$,
        \item If $\varepsilon=-1$, $\eta$ is an asymptotic net and consequently by Proposition~\ref{prop:categorizing:IIDs}, $\xi$ is an instantaneous wrapping of $(\psi,U)$.
    \end{enumerate}
    \item[$(3)\!\Rightarrow\!(1)$]: Assume that $K^{t_0}$ is the K-surface of revolution with the K-frame $(E_1,E_2)$. The induced V-frame on the V-surface $V^{t_0}$ consists of conjugate directions and hence its second fundamental form is
    \begin{equation*}
        II^{t_0} = L\,\theta^1\otimes\theta^1 + N\,\theta^2\otimes\theta^2,\qquad
        L =c\sin(\omega),\qquad
        N = b\sin(\omega),
    \end{equation*}
    with $bc\neq 0$.
    Now, pass to the bisecting normalized frame $(E_3,E_4)$ of $(E_1,E_2)$. In its dual coframe, $(\theta^3,\theta^4)$, we get
    \begin{equation}\label{eq:II:eta:theta:56}
       II^{t_0} = \mu_1^2\,(L + N)\,\theta^3\otimes\theta^3 + 2\mu_1\mu_2\,(L - N)\,{\theta}^3\odot{\theta}^4 + \mu_2^2\,(L + N)\,\theta^4\otimes\theta^4,
    \end{equation}
    where $\mu_1 = \frac{1}{2}\sec{(\frac{\omega}{2})}$ and $\mu_2 = \frac{1}{2}\csc{(\frac{\omega}{2})}$. Since $(E_3,E_4)$ bisect the asymptotic directions they are mapped to the principal directions on $K^{t_0}$. 
    Choose compatible representatives $\psi,\eta\colon\Omega\to\mathbb R^3$ of $K^{t_0}$ and $V^{t_0}$, adapted to $(E_3,E_4)$, respectively. By condition~(3), each point of $\Omega$ has a sufficiently small simply connected open neighborhood $U\subseteq\Omega$ on which $\eta|_U$ is the rotation field of an instantaneous sliding (wrapping) $\xi\colon U\to\mathbb R^3$ of $\psi|_U$. Applying Proposition~\ref{prop:categorizing:IIDs} gives the following alternatives:
    \begin{enumerate}
        \item If $\xi$ is an instantaneous sliding, then $\eta|_U$ is a conjugate net.
        Thus, \eqref{eq:II:eta:theta:56} gives $L=N$ on $U$.
        \item If $\xi$ is an instantaneous wrapping, then $\eta|_U$ is an asymptotic net.
        Thus, \eqref{eq:II:eta:theta:56} gives $L=-N$ on $U$.
    \end{enumerate}
    Since condition~(3) holds locally throughout $\Omega$, these identities hold throughout $\Omega$. As $\sin\omega\neq0$, they imply $b=c$ in the sliding case and $b=-c$ in the wrapping case. In particular, $|b|=|c|$. Now, denote the $I$-$II$-Killing field of $K^{t_0}$ by $Y$. By Lemma~\ref{lem:K:surf:revol}, $Y$ bisects the K-frame of $K^{t_0}$. The relations \eqref{eq:e1:e3:bare1:bare3} along with $|c| = |b|$ show that $Y$ also bisects the induced V-frame on $V^{t_0}$. Therefore, Theorem~\ref{theorem:main} implies that $V^{t_0}$ is alignable with respect to its induced V-net. Since $K^{t_0}$ is a surface of revolution, this alignability is of the first kind. Furthermore, \eqrefi{eq:Gaussian:mean:v:net:A:net}{2} shows that $V^{t_0}$ is positive- (negative-) alignable in the sliding (wrapping) case.
    
    For the final assertion, take $Y$ to be the $I$-$II$-Killing field of $K^{t_0}$. The argument for $(1)\Rightarrow(2)$ shows that every $V^t$ is alignable of the first kind relative to this same field. By Corollary~\ref{corollary:alignble:is:geodesic:parallel:surf}, every $V^t$ is therefore helicoidal relative to $Y$, so $\mathcal{L}_Y II^t=0$ for every admissible $t$. Hence, $\{V^t\}$ is a $1$-parameter subfamily of the Bour family of $V^{t_0}$ relative to $Y$.
\end{proof}
\subsection{Formulas of Immersion}\label{sec:immersions}
Theorem~\ref{theorem:main:first} states that an alignable V-surface of the first kind can be constructed by starting from a K-surface of revolution. Infinitesimally sliding or wrapping that surface about its axis of symmetry (see Fig~\ref{fig:pseudosphere}) produces a rotation field, which determines an alignable helicoidal V-surface $V^0$. Then the aforementioned theorem asserts that all remaining V-surfaces of $\{V^t\}$ arise from $V^0$ by a $1$-parameter family of Bour isometric deformations of $V^0$.
Using this kinematic interpretation, we proceed to derive explicit immersion formulas for the alignable V-nets of the first kind. To put this roadmap into action, we start from a K-surface of revolution. 

K-surfaces of revolution are classically characterized by their profile curves and hence are a standard topic in differential geometry (see, e.g., \cite{DocarmoDifferential}). Here, however, we seek immersion formulas for the entire family of K-nets of revolution as a function of the alignability parameter \(k\). This requires an asymptotic parametrization in which the dependence on $k$ is explicit. Lemma~\ref{lem:K:surf:revol} tells us that a helicoidal K-surface is a K-surface of revolution if the K-frame is bisected by the $I$-$II$-Killing field. This, in the light of Remark~\ref{remark:Y}, gives $\d\omega(E_1) = \d\omega(E_2)$. After setting $(E_1,E_2)=(\partial_u,\partial_v)$, we get that the asymptotic angle function $\omega$ of all K-surfaces of revolution is a solution to the system of PDEs $\omega_u=\omega_v$ and $\omega_{uv} = \sin(\omega)$ where the last is actually sine-Gordon equation in coordinates (see Theorem~\ref{theorem:classify:K-nets}). Solving the system using the \emph{pendulum equation} approach results in
\begin{equation}\label{eq:elliptic:modulus}
    \omega_k(u,v) = \arccos{\left(2\,k^2\,\mathrm{sn}^2\left(u+v\mid k\right) -1\right)},
\end{equation}
where $\mathrm{sn}\left(u+v \mid k\right) = \sin\left(\mathrm{am}\left(u+v \mid k\right)\right)$ and $k>0$ is the \emph{elliptic modulus}. Now, considering that the singularities of a pseudospherical surface happen whenever their principal curvatures, namely 
\begin{equation}\label{eq:principal:curvature:K:surf}
    \kappa_1 = \tan{\left(\frac{\omega}{2}\right)},\qquad\qquad
    \kappa_2 = -\cot{\left(\frac{\omega}{2}\right)}.
\end{equation}
vanish, the fundamental forms of \eqref{eq:I:II:psi} are regular on the following parameter domains 
\begin{equation}\label{eq:Omega:u:v}
    \Omega_k = 
    \left\{
    \begin{array}{ll}
    \bigl\{\, (u,v)\in\mathbb{R}^2 \,\bigm|\, 2m\,\mathcal{K}(k) < x < 2(m + 1)\,\mathcal{K}(k)\,\bigr\}, &  k \in (0,1) \\[0.3cm]
    \bigl\{\, (u,v)\in\mathbb{R}^2 \,\bigm|\, 0 < x\,\bigr\}\quad\text{or}\quad \bigl\{\, (u,v)\in\mathbb{R}^2 \,\bigm|\, 0 > x\,\bigr\}, &  k = 1 \\[0.3cm]
    \bigl\{\, (u,v)\in\mathbb{R}^2 \,\bigm|\, 2m\,k^{-1}\mathcal{K}\left(k^{-1}\right) < x < (2m+1)\,k^{-1}\mathcal{K}\left(k^{-1}\right)\,\bigr\}, &  k \in (1,\infty)
    \end{array}
    \right.
\end{equation}
where $m\in\Z$, $x = u + v$, $\operatorname{am}(- \mid k)$ is the \emph{Jacobi amplitude} and $\mathcal{K}(k)$ / $\mathcal{F}(- \mid k)$ for \emph{complete / incomplete elliptic integral of the first kind}. Having determined the asymptotic angle function and its regularity domains, we now give explicit immersion formulas for the corresponding K-nets of revolution.
\begin{proposition}\label{prop:K:net:immersion}
    Up to a box reparametrization, the $K$-nets of revolution are given by a $1$-parameter family $\psi_k\colon\Omega\to\R^3$ with the parameter domain $\Omega\subseteq\Omega_k$ and 
    \begin{equation}\label{eq:K:net:immersion}
    \begin{array}{l@{\qquad\qquad}l}
    \displaystyle
    \psi_{k}(u,v)
    = \frac{1}{k}\,
    \begin{pmatrix}
    \dn(x\mid k)\,\cos(y)\\[.15cm]
    \dn(x\mid k)\,\sin(y)\\[.15cm]
    \mathcal{E}\left(\mathrm{am}(x\mid k)\mid k\right) - x
    \end{pmatrix},
    &
    \left\{
    \begin{array}{l}
    x = u + v,\\[.3cm]
    y = k(u - v),
    \end{array}
    \right.
    \end{array}
    \end{equation}
    where $\mathcal{E}(-\mid-)$ is the incomplete elliptic integral of the second kind. Moreover, $(\psi_k,\Omega_k)$ corresponds to a hyperbolic pseudosphere for \(0<k<1\), to the parabolic pseudosphere for \(k=1\), and to the elliptic pseudosphere for \(k>1\).
\end{proposition}
\begin{proof}
    See Appendix~\ref{appendix:K:net:immersion}.
\end{proof}
\begin{remark}
    Since every K-net is a Chebyshev net and hence alignable \cite{AlignablePellis}, \eqref{eq:K:net:immersion} yields immersions of the alignable K-nets of revolution depending on the alignability parameter $k$.
\end{remark}

To construct the instantaneous wrapping and sliding for a K-surface of revolution we recall the Bour isometric deformation formula, discovered by Edmond Bour in 1862 (see \cite{Bour}). We use it with a slight modification that removes the need to assume the profile curve is parametrized by arc-length.
\begin{theorem}[\textit{Bour Theorem}]\label{theorem:T-net:II_X:preserving:ID}
    Let $\psi(x,y)$ be the curvature line net of a surface of revolution with nowhere vanishing Gaussian curvature and the profile curve $\gamma(x) = (f(x),0,g(x))$ where $f(x)>0$ and $g_x\neq0$. W.l.o.g assume that $\phi_y>0$ and $g(x_0) = 0$. Then the formula of immersion for the Bour family $\{\psi^{s,t}(x,y)\}$ with $\psi^{1,0}(x,y)=\psi(x,y)$ is of the following form:
    \begin{equation}\label{eq:bour:immersion}
       \psi^{s,t}(x,y)
        =
        \begin{pmatrix}
        \displaystyle
        \frac{\sqrt{s f(x)^2-t^2}}{s}
        \cos\bigl(\varphi^{s,t}(x,y)\bigr)
        \\[1em]
        \displaystyle
        \frac{\sqrt{s f(x)^2-t^2}}{s}
        \sin\bigl(\varphi^{s,t}(x,y)\bigr)
        \\[1em]
        \displaystyle
        \sqrt{s}\int_{x_0}^x
        \frac{
            f(z)\sqrt{
                \bigl(s f(z)^2-t^2\bigr)\,\|\gamma_z\|^2
                -f(z)^2 f_z^2
            }
        }{
            s f(z)^2-t^2
        }\,dz
        -\frac{t}{s}\,\varphi^{s,t}(x,y)
        \end{pmatrix},
    \end{equation}
    where $x$ and $y$ are the surface parameters and
    \begin{equation*}
        \varphi^{s,t}(x,y) =
        \sqrt{s}\left(
            y+t\int_{x_0}^x
            \frac{
                \sqrt{
                    \bigl(s f(z)^2-t^2\bigr)\,\|\gamma_z\|^2
                    -f(z)^2 f_z^2
                }
            }{
                f(z)\bigl(s f(z)^2-t^2\bigr)
            }\,dz
        \right).
    \end{equation*}
\end{theorem}
In \eqref{eq:bour:immersion}, the ISR frame is $(E_1,E_2) = (\partial_x/\sqrt{f_x^2 + g_x^2},\partial_y/\sqrt{G(x)})$ where \(G(x)=f(x)^2\). Thus, substituting the ISR coframe $(\theta^1,\theta^2)$ into \eqref{eq:L:N:u}, we obtain coefficients that agree with the fundamental forms computed from \eqref{eq:bour:immersion}.

\begin{proposition}\label{prop:rotation:field:formula}
    Let $\psi(x,y)$ be the curvature line net of a surface of revolution with the assumptions of Theorem~\ref{theorem:T-net:II_X:preserving:ID}. Then its rotation fields corresponding to instantaneous sliding and wrapping are respectively the curvature line net $\eta^+(x,y)$ of a surface of revolution and the asymptotic net $\eta^-(x,y)$ of a helicoid:
    \begin{equation}\label{eq:rotation:fields:SR:H}
        \begin{aligned}
            \eta^{+}(x,y) = \left(\begin{array}{cc}
                -f(x)^{-1}\,\cos{(y)}\\
                -f(x)^{-1}\,\sin{(y)}\\
                \int_{x_0}^x\,{g_z}{f(z)^{-2}}\,\mathrm{d}z
            \end{array}\right),\qquad
            &\eta^{-}(x,y) = \left(\begin{array}{cc}
            f_x\,(g_x)^{-1}\sin{(y)}\\
            -f_x\,(g_x)^{-1}\cos{(y)}\\
            y
            \end{array}\right).
        \end{aligned}
    \end{equation}
\end{proposition}
\begin{proof}
    From Proposition\,\ref{prop:categorizing:IIDs} we know that the induced asymptotic net and conjugate net infinitesimal isometric deformations of a surface of revolution are the velocity fields of the sliding and wrapping Bour isometries at $(s,t) = (1,0)$. Calculating the velocity fields of $\psi^{s,t}(x,y)$ with respect to $s$ and $t$ separately at $(s,t) = (1,0)$ after some manipulations yields:
    \begin{align*}
        \left(\frac{\partial \psi^{s,t}}{\partial t}\right)\Big|_{(1,0)} &=: \xi^{+}(x,y) = 
        \left(\begin{array}{cc}
                -f(x)\,\ell(x)\,\sin{(y)}\\
                 f(x)\,\ell(x)\,\cos{(y)}\\
                -y
              \end{array}\right),\\
        2\left(\frac{\partial \psi^{s,t}}{\partial s}\right)\Big|_{(1,0)} &=:  \xi^{-}(x,y) = \left(\begin{array}{cc}
        \mu(y)\,f(x)\,\cos{(\theta(y))}\\
        \mu(y)\,f(x)\,\sin{(\theta(y))}\\
        \tilde{g}(x)
        \end{array}\right),
    \end{align*}
    where 
    \begin{align*}
        \mu(y) &= \sqrt{y^2 + 1},\quad
        &\theta(y) &= y + \pi - \arctan{y},\\
        \tilde{g}(x) &= \int_{x_0}^x\,\frac{f_z^2}{g_z}\,\mathrm{d}z,\quad
        &\ell(x) &= \int_{x_0}^x\,\frac{g_z}{f(z)^2}\,\mathrm{d}z.
    \end{align*}
    Having the infinitesimal isometric deformations, from Section\,\ref{sec:IID:rotation:field:review} using the cross product matrix we have the following system
    \begin{equation}\label{eq:sys:nondeg}
        \xi^{\pm}_x = \eta^{\pm} \times \psi_x,\qquad\qquad
        \xi^{\pm}_y = \eta^{\pm} \times \psi_y,
    \end{equation}
    in $\eta_1^{\pm},\eta_2^{\pm},\eta_3^{\pm}$ variables which are the coordinates of $\eta^{\pm}$. Since $\xi^{\pm}_x \perp \psi_x$ the \eqrefi{eq:sys:nondeg}{1} has a $1$-parameter solution of the form $\eta^{\pm} = \eta^{\pm}_0 + \mu\,\psi_x$ with $\eta^{\pm}_0 = {(\psi_x\times \xi^\pm_x)}/{\|\psi_x\|^2}$ and $\mu \in \mathbb{R}$. Plugging it into \eqrefi{eq:sys:nondeg}{2} fixes the parameter $\mu = \frac{\langle\xi_y - \eta_0\times \psi_y, \psi_x \times \psi_y\rangle}{(\|\psi_x \times \psi_y\|^2)}$.
\end{proof}
By Proposition~\ref{prop:categorizing:IIDs}, every rotation field corresponding to a instantaneous sliding (wrapping) of $\psi$ differs from $\eta^\pm$ by a nonzero constant factor and a translation (see \eqref{eq:rot:comp}). We therefore use these normalized representatives in the following construction.
Now, substitute $f(x) = k^{-1}\dn\left(x\mid k\right)$ and $g(x) = k^{-1}\left(\mathcal{E}\left(\mathrm{am}(x\mid k)\mid k\right) - x\right)$ from \eqref{eq:K:net:immersion} into \eqref{eq:rotation:fields:SR:H} (w.l.o.g we consider $-\eta^\pm(x,y)$ of \eqref{eq:rotation:fields:SR:H}). Then, for the third coordinate of $\eta^+(x,y)$ we get
\begin{equation*}
    \frac{\d}{\d x}\eta^+_3(x) = -\frac{g_x}{f(x)^2} = k\,\left(\frac{1}{\dn^2(x\mid k)}-1\right) = k\,\left(\frac{1}{1-k^2\,\sin^2{(\phi)}}-1\right),
\end{equation*}
with $\phi = \mathrm{am}(x \mid k)$. Using $\d x=\d\mathcal F(\phi\mid k)=\d\phi/\sqrt{1-k^2\sin^2\phi}$, integration with the normalization $\eta^+_3(0) = 0$ gives
\begin{equation*}
    \eta^+_3(x) = k\,\Pi\left(k^2;\phi\mid k\right) - k\,\mathcal{F}\left(\phi\mid k\right) = k\,\left(\Pi\left(k^2;\mathrm{am}(x \mid k)\mid k\right)-x\right),
\end{equation*}
where $\Pi(-;-\mid-)$ is the \emph{incomplete elliptic integral of the third kind}. Then the final result is
\begin{equation}\label{eq:surfrevolution:immersion}
    \eta^{+}(x,y) = \left(\begin{array}{>{\displaystyle}c >{\displaystyle}c >{\displaystyle}c}
        \frac{k\cos(y)}{\dn(x\mid k)}&
        \frac{k\sin(y)}{\dn(x\mid k)}&
        k\,\left(\Pi\left(k^2;\mathrm{am}(x \mid k)\mid k\right) - x\right)
    \end{array}\right).
\end{equation}
In the case of $\eta^-(x,y)$, \eqrefi{eq:rotation:fields:SR:H}{2} gives
\begin{equation}\label{eq:helicoid:immersion}
    \eta^{-}(x,y) = \left(\begin{array}{>{\displaystyle}c>{\displaystyle}c>{\displaystyle}c}
    -\frac{\cn(x\mid k)}{\sn(x\mid k)}\,\sin(y)&
    \frac{\cn(x\mid k)}{\sn(x\mid k)}\,\cos(y) &
    -y 
    \end{array}\right).
\end{equation}
It remains to determine the V-net preserving isometric deformation family of $\eta^\pm(x,y)$. By Theorem~\ref{theorem:main:first}, such an isometric deformation is contained in a $1$-parameter subfamily of the Bour family of \eqref{eq:surfrevolution:immersion} and \eqref{eq:helicoid:immersion}. For $\eta^+(x,y)$ by plugging the profile curve to \eqref{eq:bour:immersion} we get the entire Bour family $\{(\eta^{+})^{s,t}\}$. As for the case of $\eta^-(x,y)$, since the Bour formula \eqref{eq:bour:immersion} generates directly from a surface of revolution. First, we find the corresponding surface of revolution in the Bour family by applying a sliding deformation to the helicoid \eqref{eq:helicoid:immersion}. Then we plug in the profile curve of the aforementioned surface of revolution to the Bour formula of \eqref{eq:bour:immersion}. 

In order to do so, w.l.o.g, let the helicoid of $\eta^-(x,y)$ be sitting at the \((s,t)=(1,1)\) member of the Bour family, and then compute the corresponding \((s,t)=(1,0)\) member. At $(s,t) = (1,1)$ the radius of \eqref{eq:bour:immersion} is of the form $r(x) = \sqrt{f(x)^2 -1}$. Now, using the radius function of \eqref{eq:helicoid:immersion} we get
\begin{equation*}
    f(x) = \sqrt{1 + r(x)^2} = \sqrt{1 + \frac{\cn^2(x\mid k )}{\sn^2(x\mid k )}} = \frac{1}{\sn(x\mid k )}.
\end{equation*}
Since the $(1,1)$ member of \eqref{eq:bour:immersion} is a helicoid, from Theorem~\ref{theorem:B:nets} we know that $II(\partial_x,\partial_x) = II(\partial_y,\partial_y) = 0$ which gives
\begin{equation*}
    -f_x^2 + g_x^2\,(f(x)^2-1)=0\,\quad\implies\quad g_x = \frac{\dn(x\mid k )}{\sn(x\mid k )}.
\end{equation*}
Now, integrating yields $g(x) = -\ln\left(\left|\frac{\sn(x\,\mid\,k )}{1 - \cn(x\,\mid\,k )}\right|\right)$. Therefore, we get the catenoid
\begin{equation}\label{eq:catenoid:immersion}
    (\eta^{-})^{1,0}(x,y) = \left(\begin{array}{>{\displaystyle}c>{\displaystyle}c>{\displaystyle}c}
    -\dfrac{\sin(y)}{\sn(x\mid k)} &
     \dfrac{\cos(y)}{\sn(x\mid k)} &
    -\ln\!\left(\left|\dfrac{\sn(x\mid k)}{1-\cn(x\mid k)}\right|\right)
    \end{array}\right).
\end{equation}
Now, plugging the profile curves of \eqref{eq:catenoid:immersion} into \eqref{eq:bour:immersion}, produces the $2$-parameter Bour family $\{(\eta^{-})^{s,t}\}$.
\begin{remark}
    Although the V-net preserving isometric deformation of the helicoid \eqref{eq:helicoid:immersion} lies inside the Bour family of the catenoid \eqref{eq:catenoid:immersion}, the catenoid itself is not a member of that $1$-parameter subfamily. In fact, one can verify by direct computation that a catenoid does not support any V-net. Hence, \eqref{eq:catenoid:immersion} is only a geodesic alignable net. Fig.~\ref{fig:alignable:deformation:all}-(b,f,j) depicts these nets on the catenoid.
\end{remark}
\begin{figure*}[t!]
\centering

\begin{tikzpicture}
    \def\labx{2mm}
    \def\laby{-2mm}
    \def\cellh{39mm}

    \matrix (m) [
        matrix of nodes,
        ampersand replacement=\&,
        row sep=2mm,
        column sep=1.5mm,
        nodes={
            anchor=north west,
            inner sep=0pt,
            outer sep=0pt
        }
    ]{
        \begin{minipage}[t][\cellh][t]{.235\textwidth}\centering
            \includegraphics[width=\linewidth,height=30mm,keepaspectratio]{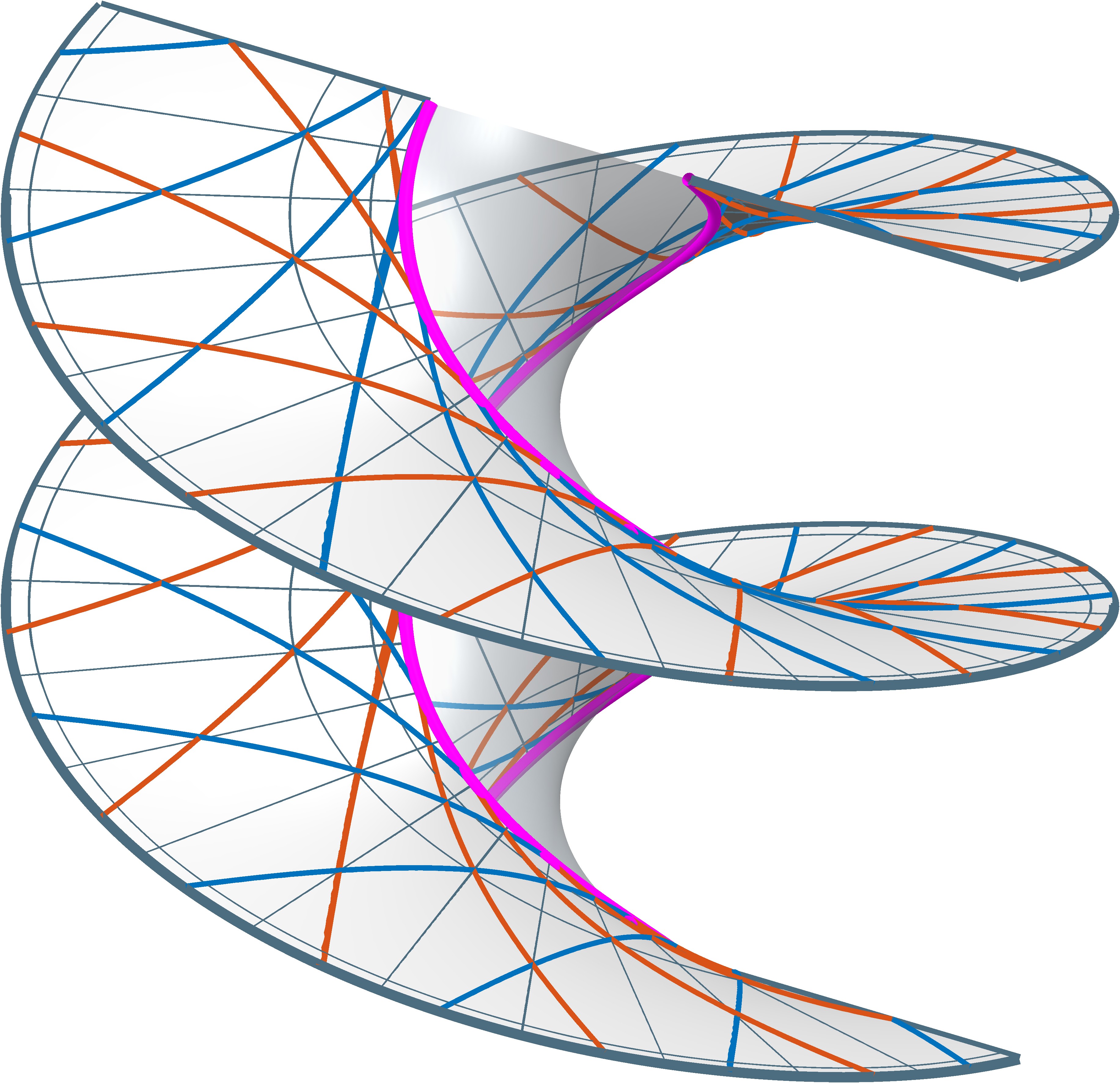}
        \end{minipage}
        \&
        \begin{minipage}[t][\cellh][t]{.235\textwidth}\centering
            \includegraphics[width=\linewidth,height=30mm,keepaspectratio]{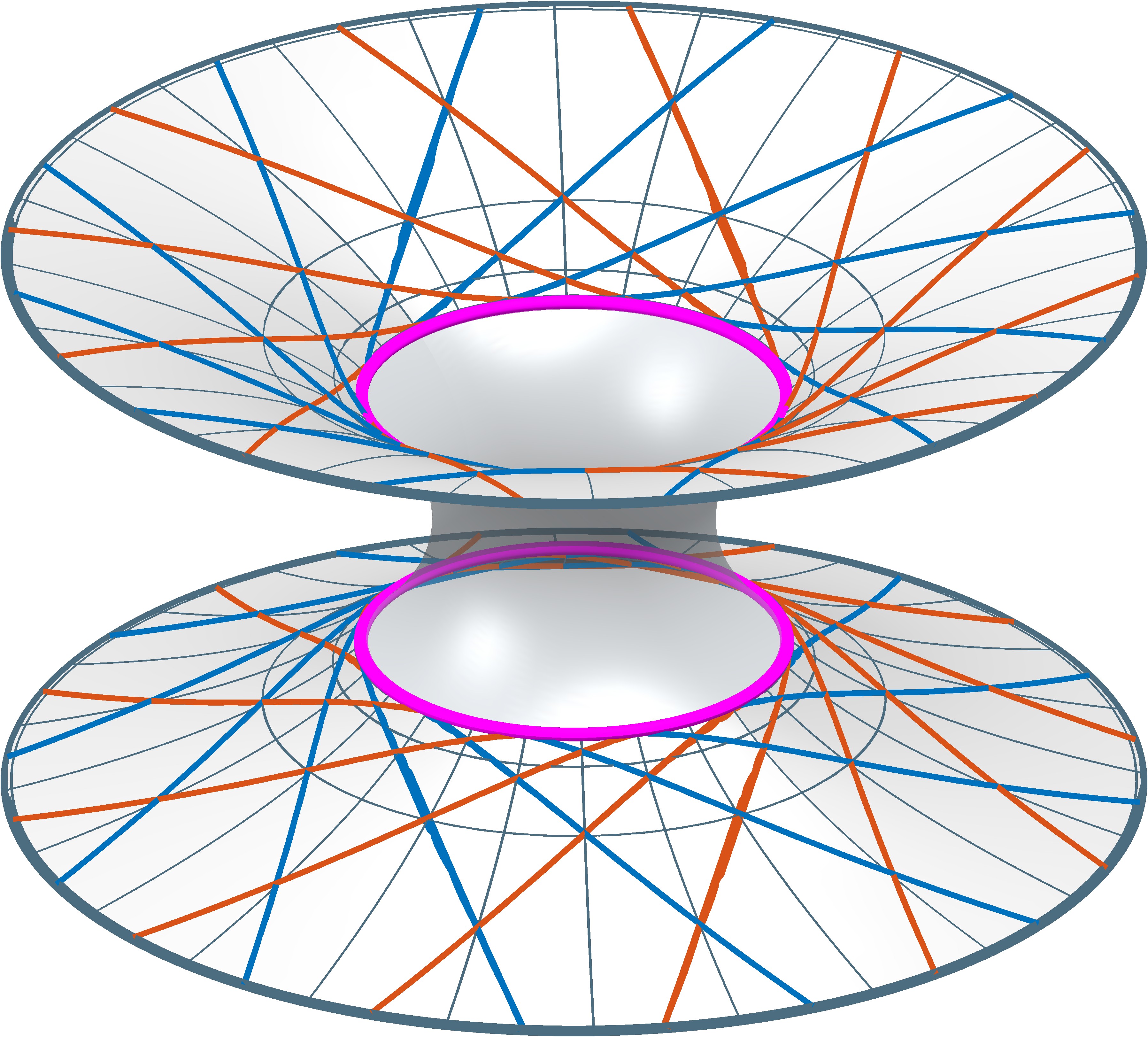}
        \end{minipage}
        \&
        \begin{minipage}[t][\cellh][t]{.235\textwidth}\centering
            \includegraphics[width=\linewidth,height=30mm,keepaspectratio]{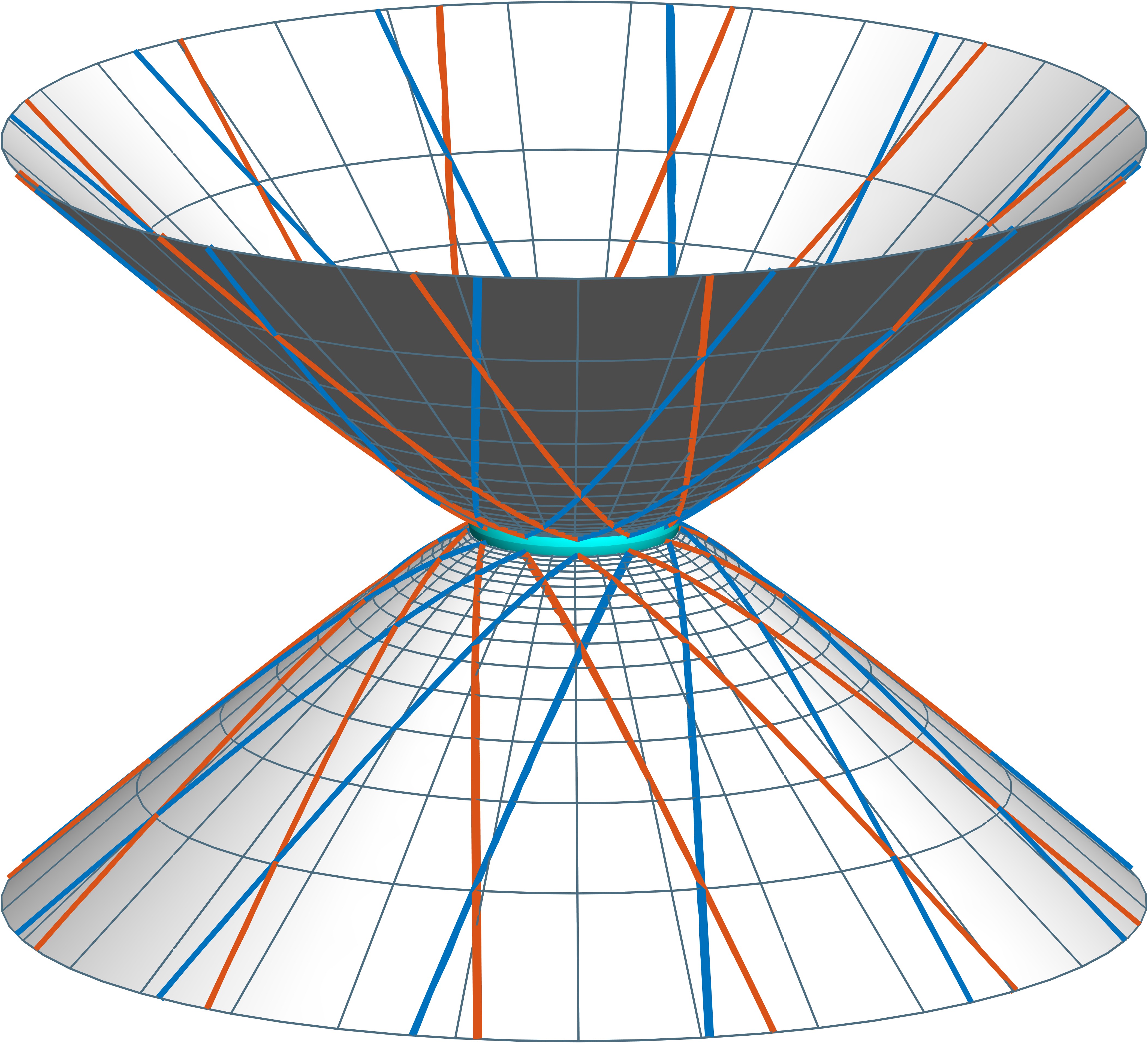}
        \end{minipage}
        \&
        \begin{minipage}[t][\cellh][t]{.235\textwidth}\centering
            \includegraphics[width=\linewidth,height=30mm,keepaspectratio]{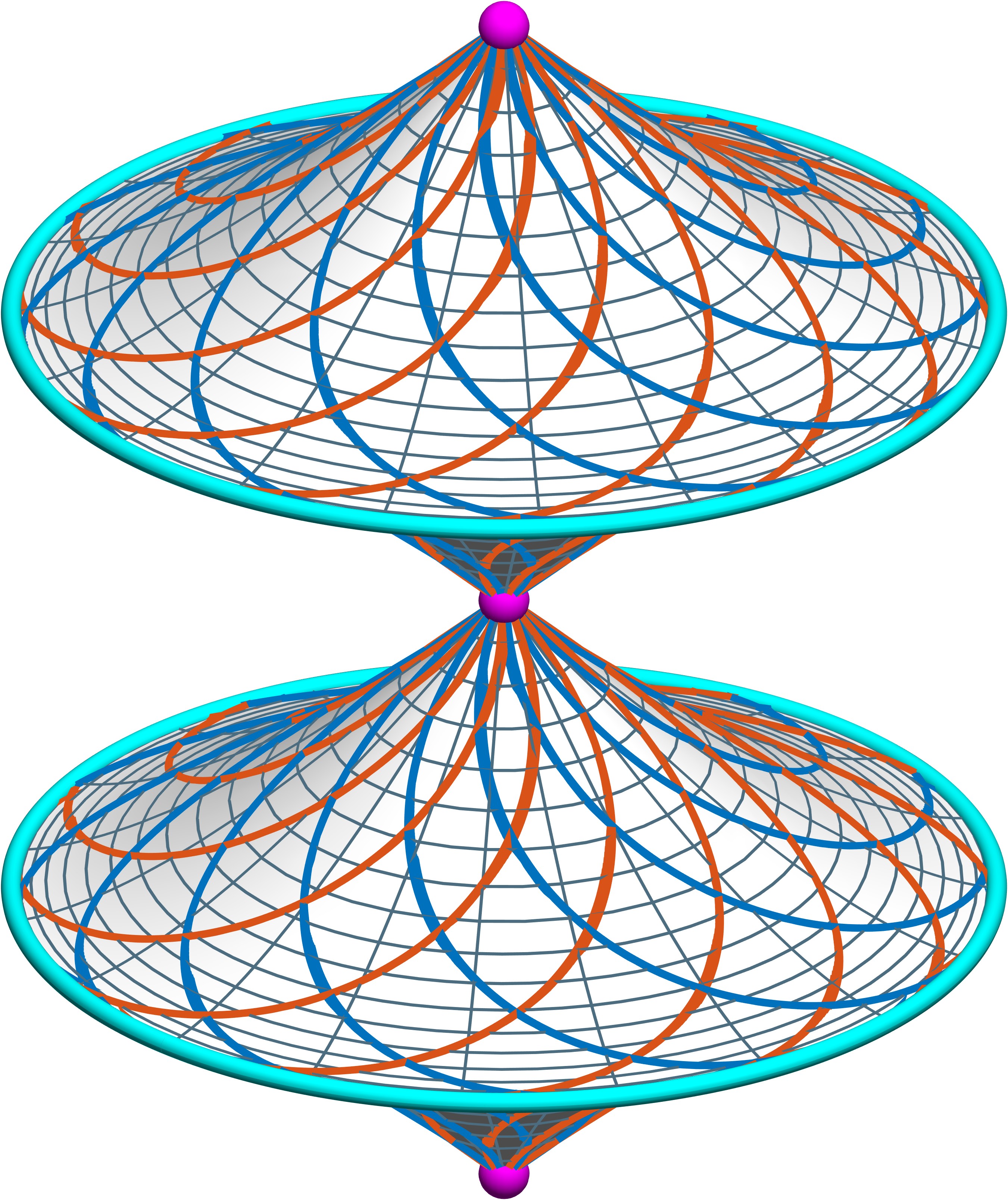}
        \end{minipage}
        \\

        \begin{minipage}[t][\cellh][t]{.235\textwidth}\centering
            \includegraphics[width=\linewidth,height=30mm,keepaspectratio]{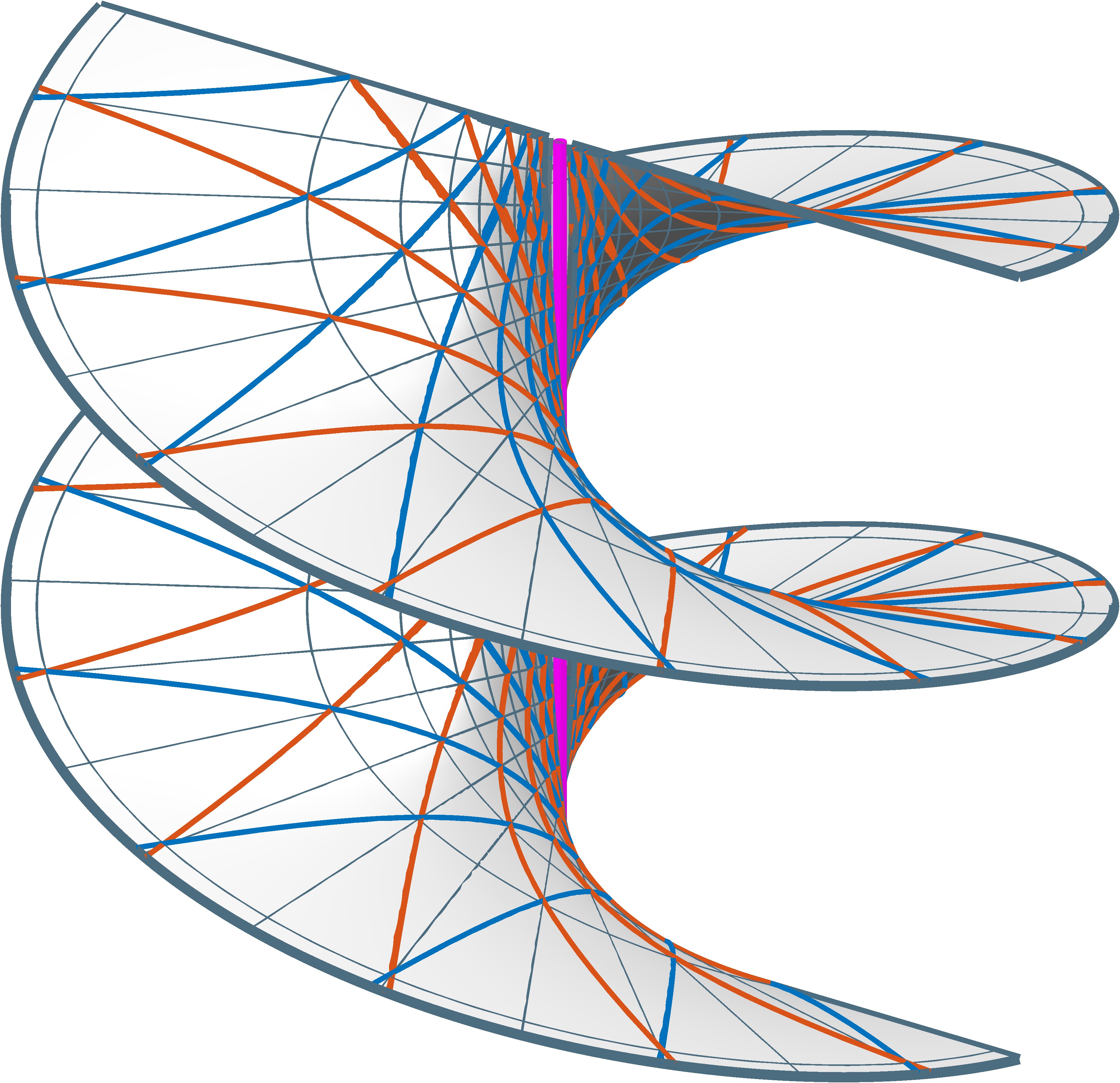}
        \end{minipage}
        \&
        \begin{minipage}[t][\cellh][t]{.235\textwidth}\centering
            \includegraphics[width=\linewidth,height=30mm,keepaspectratio]{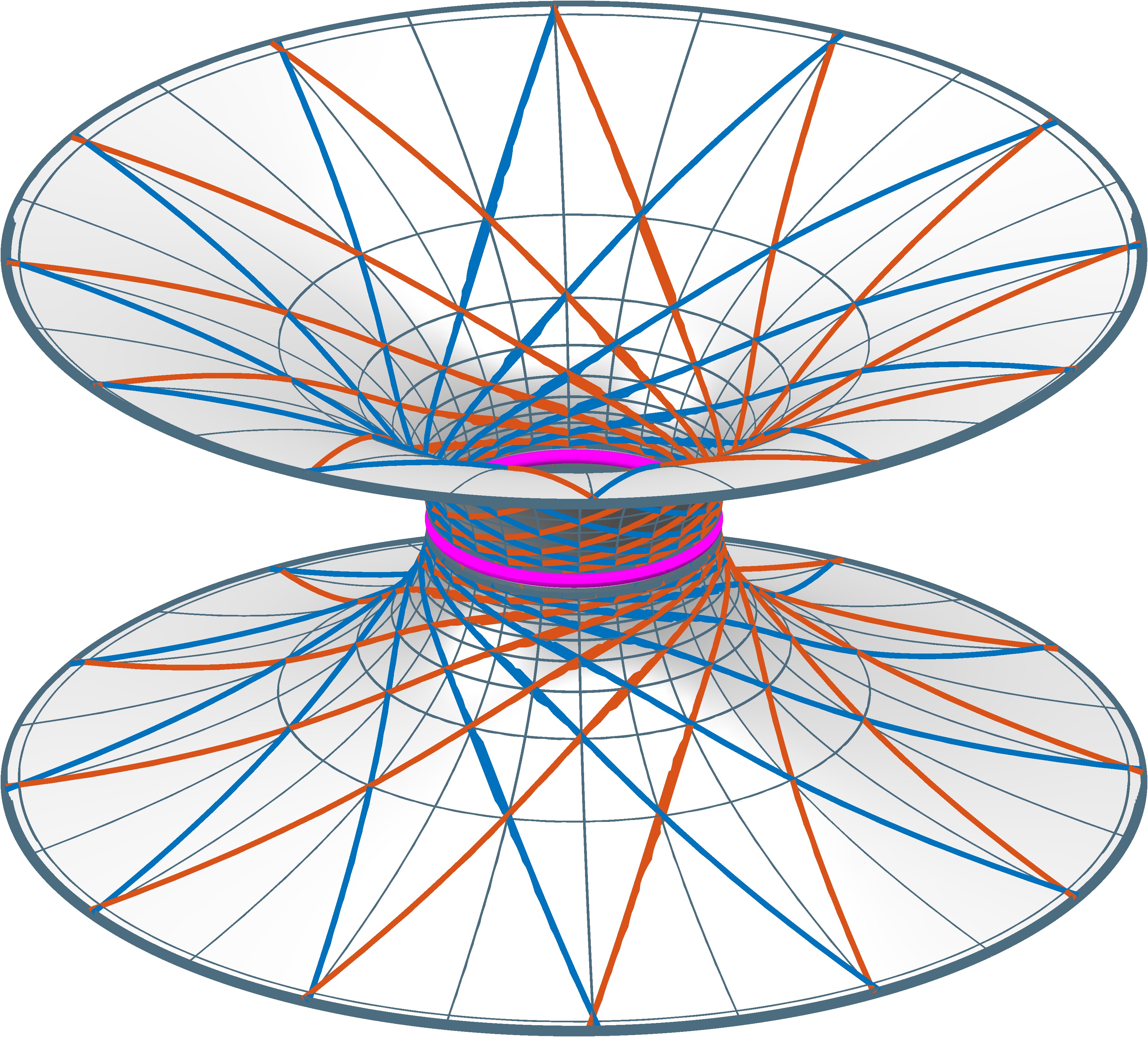}
        \end{minipage}
        \&
        \begin{minipage}[t][\cellh][t]{.235\textwidth}\centering
            \includegraphics[width=\linewidth,height=30mm,keepaspectratio]{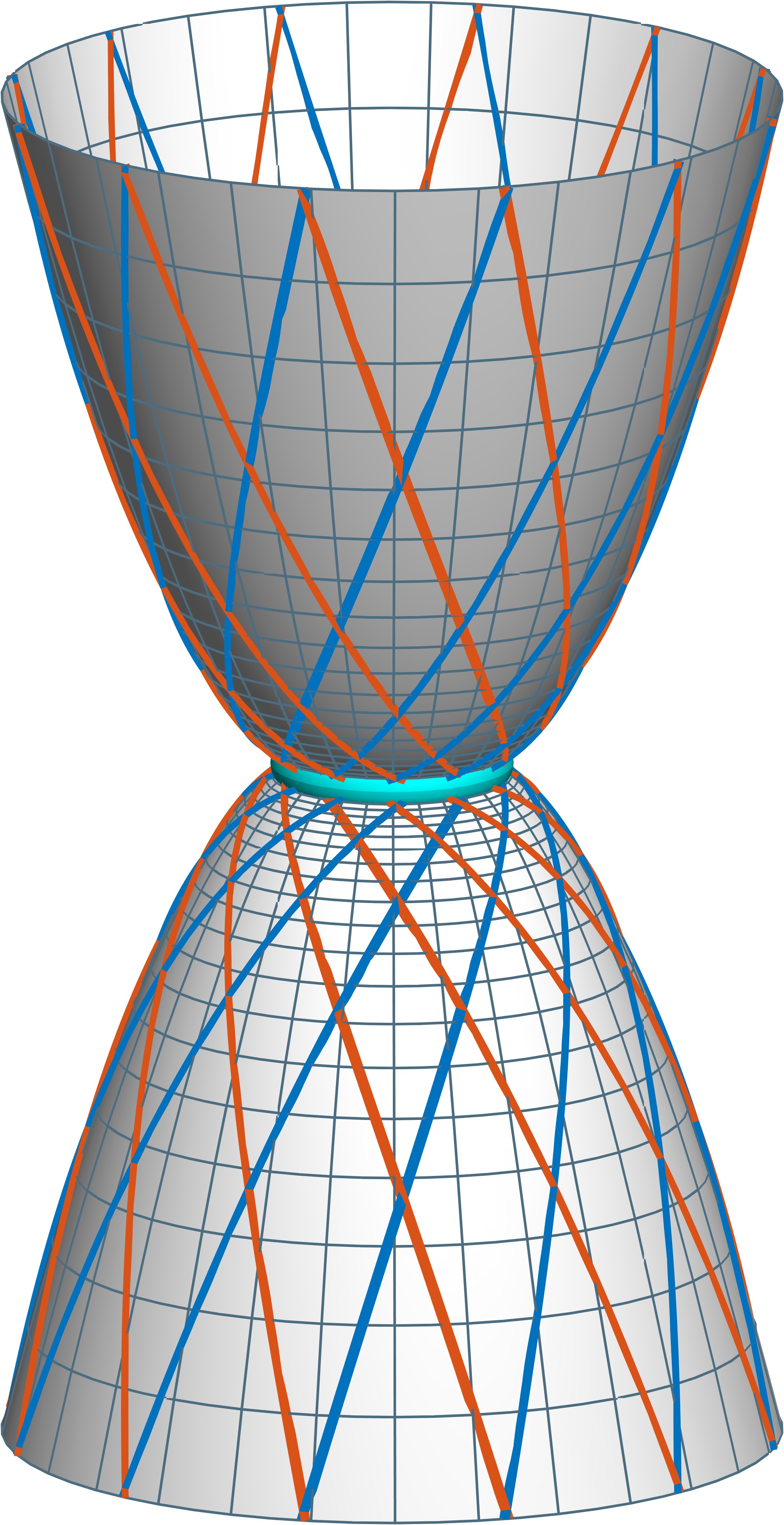}
        \end{minipage}
        \&
        \begin{minipage}[t][\cellh][t]{.235\textwidth}\centering
            \includegraphics[width=\linewidth,height=30mm,keepaspectratio]{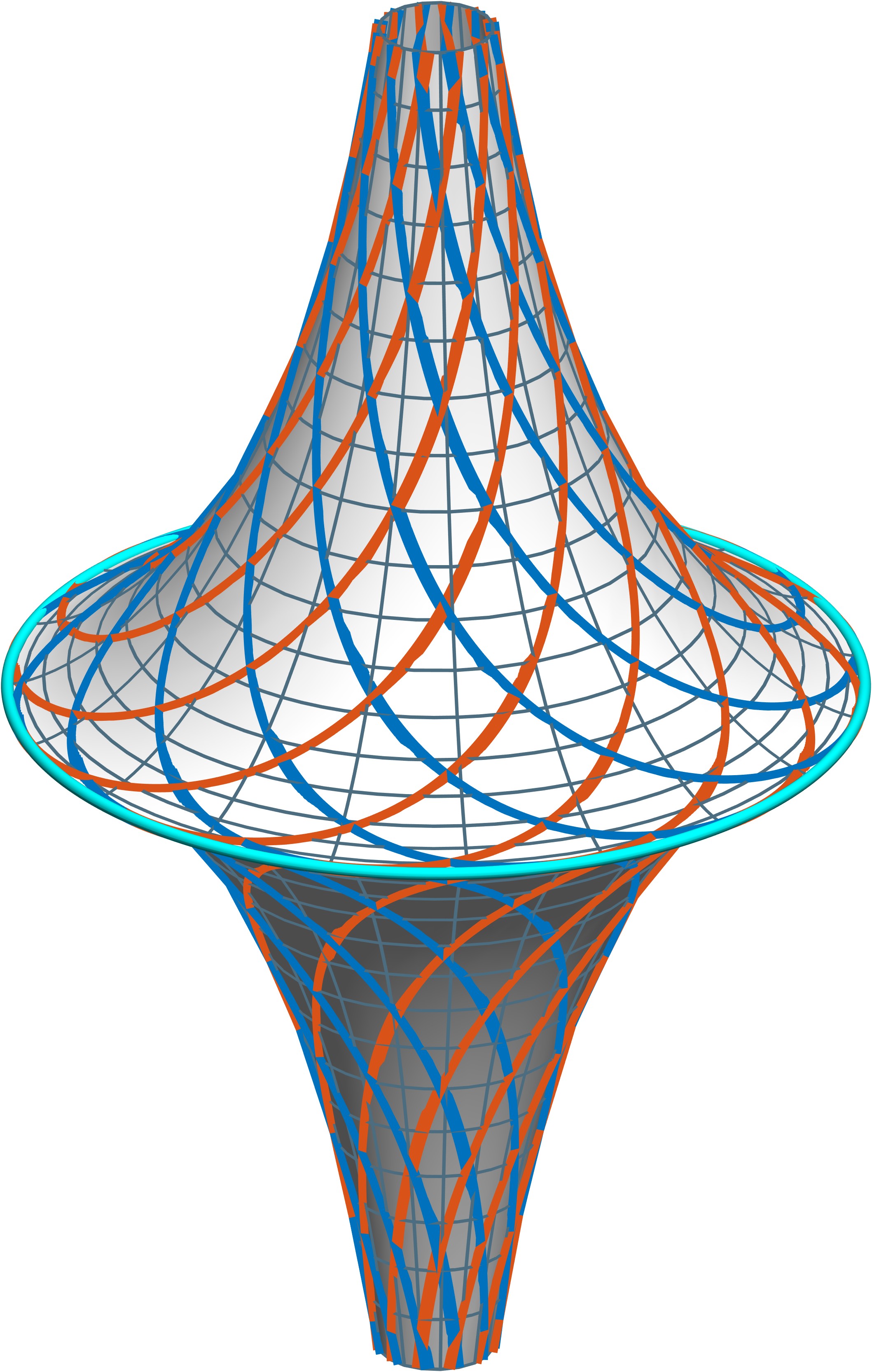}
        \end{minipage}
        \\

        \begin{minipage}[t][\cellh][t]{.235\textwidth}\centering
            \includegraphics[width=\linewidth,height=30mm,keepaspectratio]{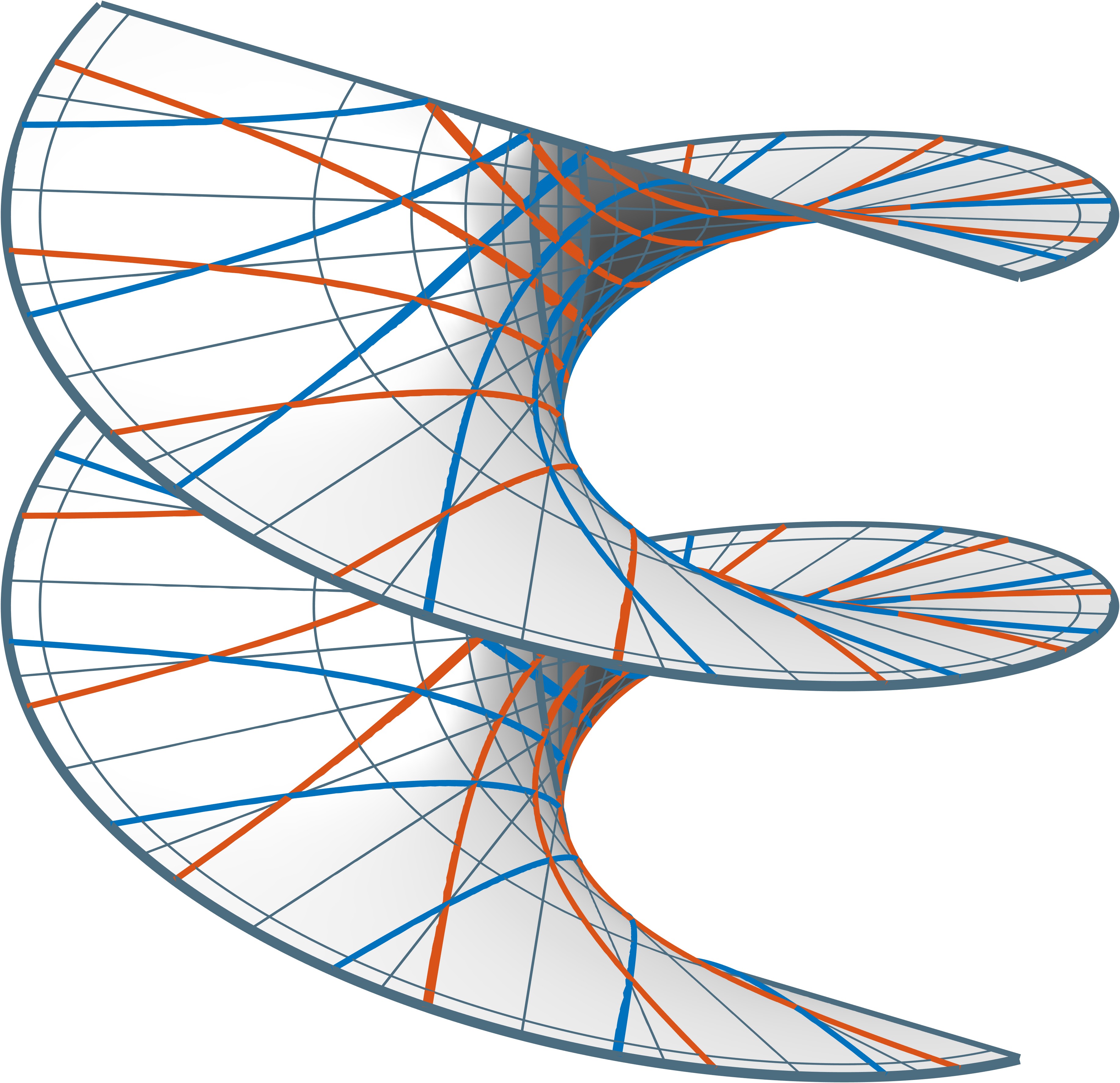}
        \end{minipage}
        \&
        \begin{minipage}[t][\cellh][t]{.235\textwidth}\centering
            \includegraphics[width=\linewidth,height=30mm,keepaspectratio]{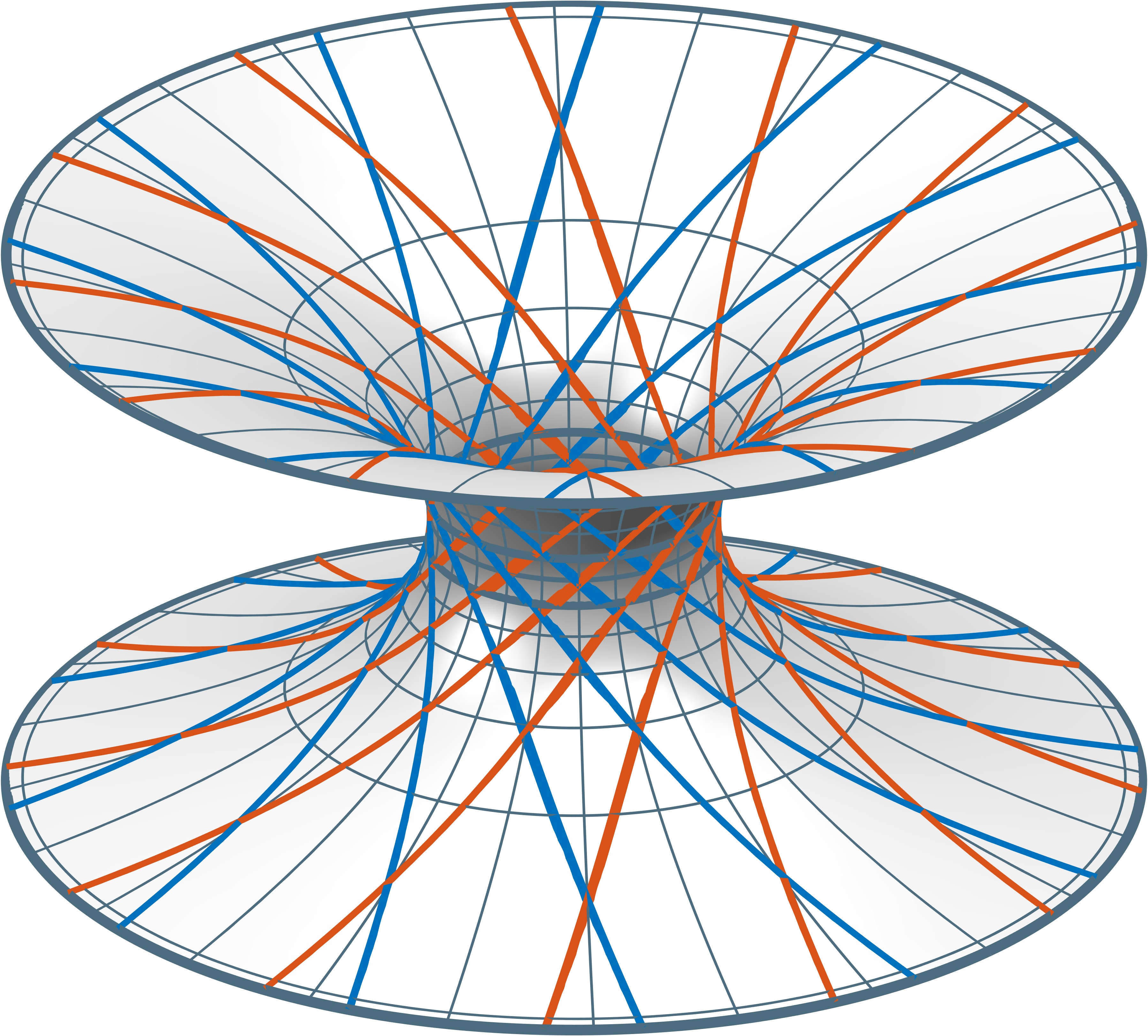}
        \end{minipage}
        \&
        \begin{minipage}[t][\cellh][t]{.235\textwidth}\centering
            \includegraphics[width=\linewidth,height=30mm,keepaspectratio]{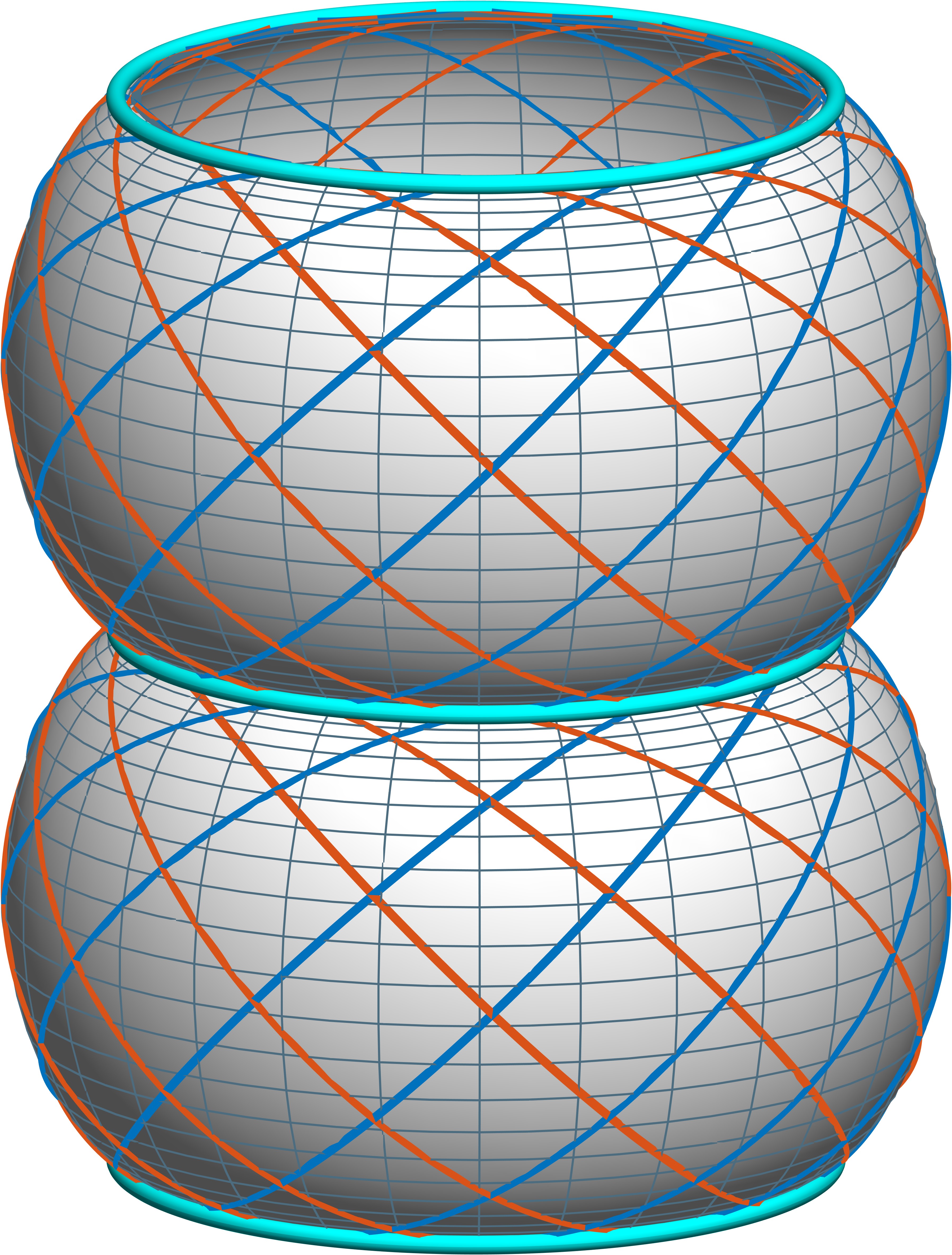}
        \end{minipage}
        \&
        \begin{minipage}[t][\cellh][t]{.235\textwidth}\centering
            \includegraphics[width=\linewidth,height=30mm,keepaspectratio]{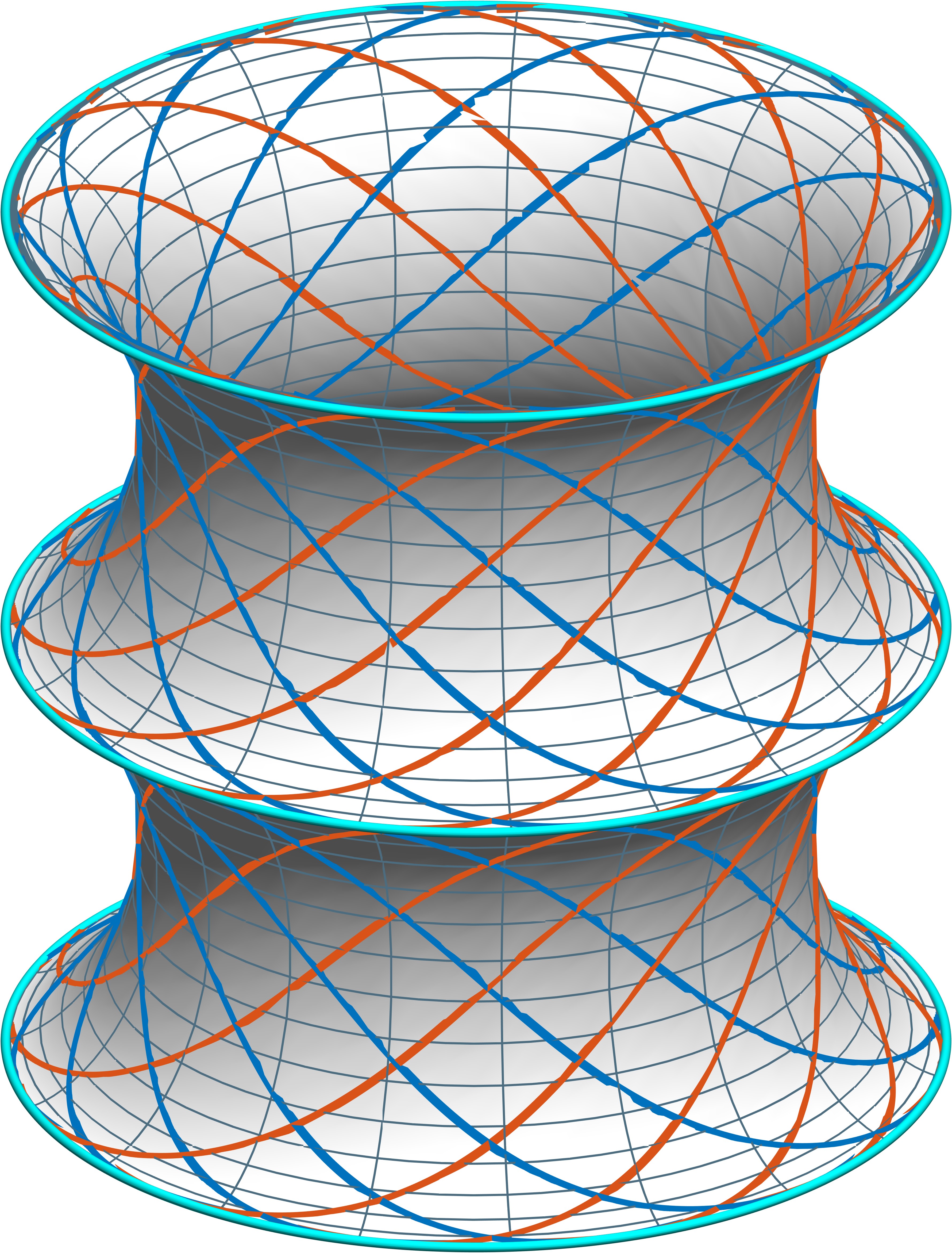}
        \end{minipage}
        \\
    };

    \foreach \r/\c/\lbl in {
        1/1/a, 1/2/b, 1/3/c, 1/4/d,
        2/1/e, 2/2/f, 2/3/g, 2/4/h,
        3/1/i, 3/2/j, 3/3/k, 3/4/l
    }{
        \node[
            draw,
            circle,
            fill=white,
            inner sep=1.2pt,
            font=\bfseries\small,
            anchor=north west
        ] at ([xshift=\labx,yshift=\laby]m-\r-\c.north west) {\lbl};
    }
\end{tikzpicture}

\caption{Illustration of the alignable deformation of alignable V-nets of the first kind for the negative-alignable type ($t = 1$) (a,e,i) and the corresponding catenoid (b,f,j), the positive-alignable type ($t = 0$) (c,g,k) and the corresponding K-nets of revolution (d,h,l). From top to bottom the alignable parameters, $k$, of each row is $1.5$, $1$ and $0.7$ respectively. The cyan curves correspond to the cusps of the pseudosphere while and the magenta curves correspond to its cone singularities.}
\label{fig:alignable:deformation:all}
\end{figure*}
The subfamily of $\{(\eta^{\pm})^{s,t}\}$ preserving the V-net is the one for which the V-frames of $(\eta^{+})^{1,0}$ and $(\eta^{-})^{1,1}$ remain conjugate. Let $I^\pm$ and $II^\pm$ denote the fundamental forms of $(\eta^{\pm})^{s,t}$ in the coordinates $(x,y)$. The ISR frame $(\bar{E}_3,\bar{E}_4)$ for the Bour family are the velocity fields of $x$- and $y$-lines. On the other hand, by Theorem~\ref{theorem:main} and Corollary~\ref{corollary:alignble:is:geodesic:parallel:surf}, the Killing field bisects the V-frame. This means that the ISR frame $(\bar{E}_3,\bar{E}_4)$ bisects the V-frame $(\bar{E}_1,\bar{E}_2)$. Consequently,
\begin{equation}\label{eq:X:Y:pm}
        \left(
        \frac{\bar{E}_1+\bar{E}_2}{2\cos\!\left(\frac{\omega_k}{2}\right)},
        \frac{\bar{E}_1-\bar{E}_2}{2\sin\!\left(\frac{\omega_k}{2}\right)}
        \right) 
        = \left(\bar{E}_3,\bar{E}_4\right) = \left(\frac{\partial_x}{\sqrt{I^+(\partial_x,\partial_x)}},\frac{\partial_y}{\sqrt{I^+(\partial_y,\partial_y)}}\right),
\end{equation}
for the positive alignable case. The negative alignable case is simply obtained by replacing $E_2\mapsto-E_2$, $\omega\mapsto\pi-\omega$ and $I^+\mapsto I^-$ in \eqref{eq:X:Y:pm}.
Using the relations \eqref{eq:X:Y:pm} and \eqref{eq:elliptic:modulus}, we compute $(\bar{E}_1,\bar{E}_2)$ according to $(\bar{E}_3,\bar{E}_4)$. A Bour isometric deformation is a V-net preserving isometric deformation of $(\eta^{+})^{1,0}$ and $(\eta^{-})^{1,1}$ if and only if $(\bar{E}_1,\bar{E}_2)$ remains conjugate. Consequently,
\begin{align*}
    0 = II^+(\bar{E}_1,\bar{E}_2) &\propto (s-1)\,k^2 - t^2,\quad\implies\quad s = k^{-2}\,(t^2 +k^2),\\
    0 = II^-(\bar{E}_1,\bar{E}_2) &\propto 1 - t^2 + (s-1)\,k^2,\quad\implies\quad s = k^{-2}\,(t^2 +k^2 -1).
\end{align*}
Since the \((u,v)\)-coordinates of the K-surface of revolution induce V-nets on the corresponding V-surfaces, it suffices to express \((x,y)\) in terms of \((u,v)\) using \eqrefi{eq:K:net:immersion}{2}. Finally, we can wrap up all our findings in the following theorem:
\begin{theorem}[\textit{Formulas of Immersion}]\label{theorem:main:immersion:V}
    The alignable V-surfaces of the first kind are represented by a positive- and a negative-alignable $2$-parameter families of helicoidal surfaces, denoted by $(\eta^{\pm,t}_k,\Omega)$, where $\Omega \subseteq \Omega_k$. Here, $t$ and $k$ are isometric deformation and alignability parameters while $\eta^{\pm,t}_k = \eta^\pm(u,v,k,t)$ with
    \begin{equation*}
        \eta^\pm \colon \Omega \times (0,+\infty) \times I^{\pm}_\epsilon \to \mathbb{R}^3\quad\text{with}\quad 
        I^{+}_\epsilon = \mathbb{R}\quad\text{and}\quad
        I^{-}_\epsilon = \left(-\infty,-1\right]\cup \left[1,+\infty\right).
    \end{equation*}
    Up to a rigid motion and a box reparametrization, these V-nets are given by \eqref{eq:bour:immersion} with the following substitutions for the positive-alignable case
    \begin{equation*}
        s = 1 + \left(\frac{t}{k}\right)^2,\qquad
        f(x) = k \dn^{-1}(x\mid k) ,\qquad
        g(x) = k\left(\Pi\left(k^2;\mathrm{am}(x \mid k)\mid k\right) - x\right),
    \end{equation*}
    and for the {negative-alignable case}
    \begin{equation*}
        s = 1 + \frac{t^2 - 1}{k^2},\qquad
        f(x) = \sn^{-1}(x\mid k),\qquad
        g(x) = -\ln\left(\left|\frac{\sn(x\mid k)}{1-\cn(x\mid k)}\right|\right),
    \end{equation*} 
    where $(x,y) = (u + v, k(u - v))$ respectively. 
\end{theorem}
\begin{figure*}[t!]
\centering

\begin{tikzpicture}
    \def\labx{2mm}      
    \def\laby{-2mm}     
    \def\cellh{39mm}    

    \matrix (m) [
        matrix of nodes,
        ampersand replacement=\&,
        row sep=2mm,
        column sep=1.5mm,
        nodes={
            anchor=north west,
            inner sep=0pt,
            outer sep=0pt
        }
    ]{
        \begin{minipage}[t][\cellh][t]{.235\textwidth}\centering
            \includegraphics[width=\linewidth,height=36mm,keepaspectratio]{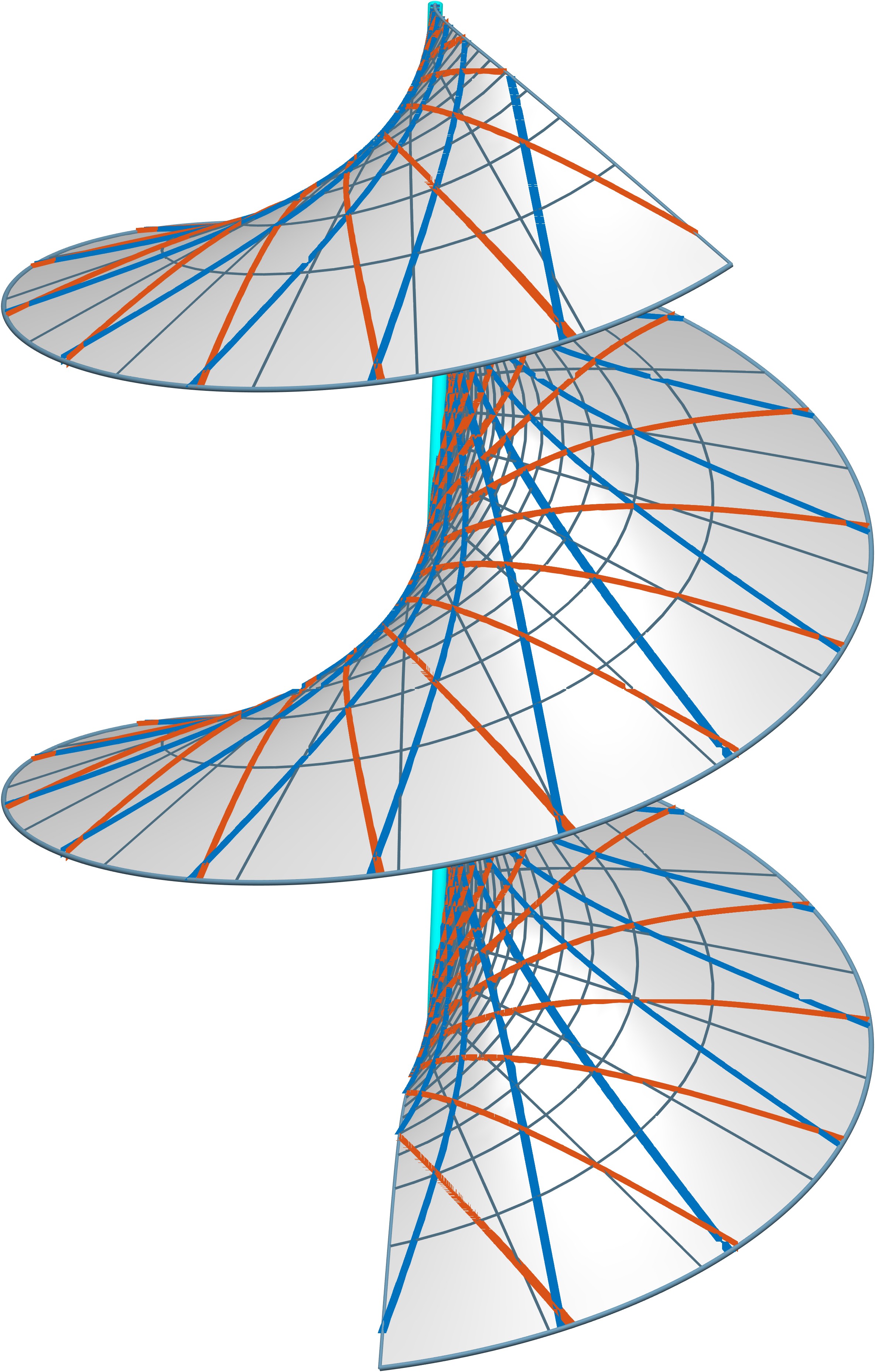}
        \end{minipage}
        \&
        \begin{minipage}[t][\cellh][t]{.235\textwidth}\centering
            \includegraphics[width=\linewidth,height=28mm,keepaspectratio]{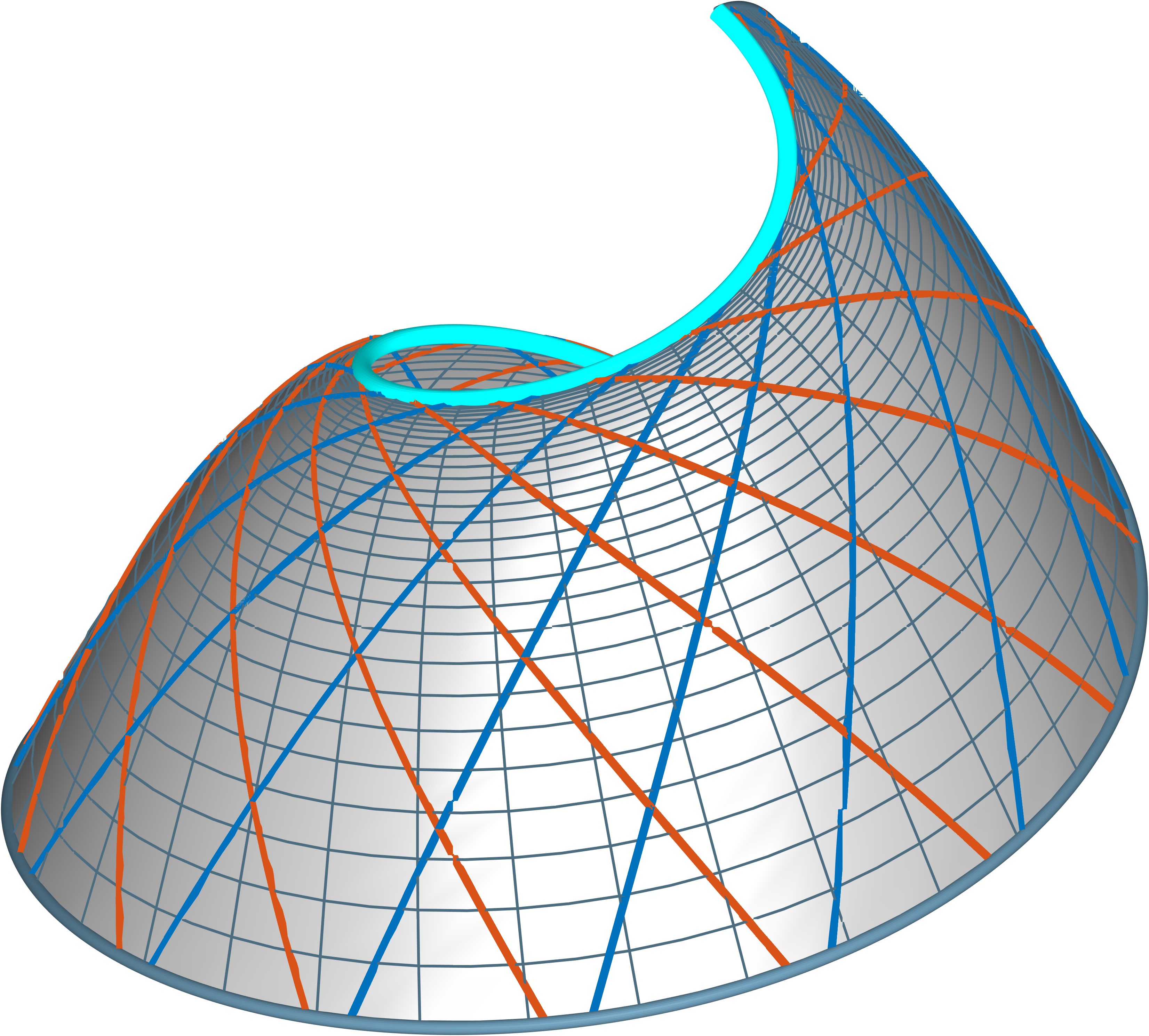}
        \end{minipage}
        \&
        \begin{minipage}[t][\cellh][t]{.235\textwidth}\centering
            \includegraphics[width=\linewidth,height=35mm,keepaspectratio]{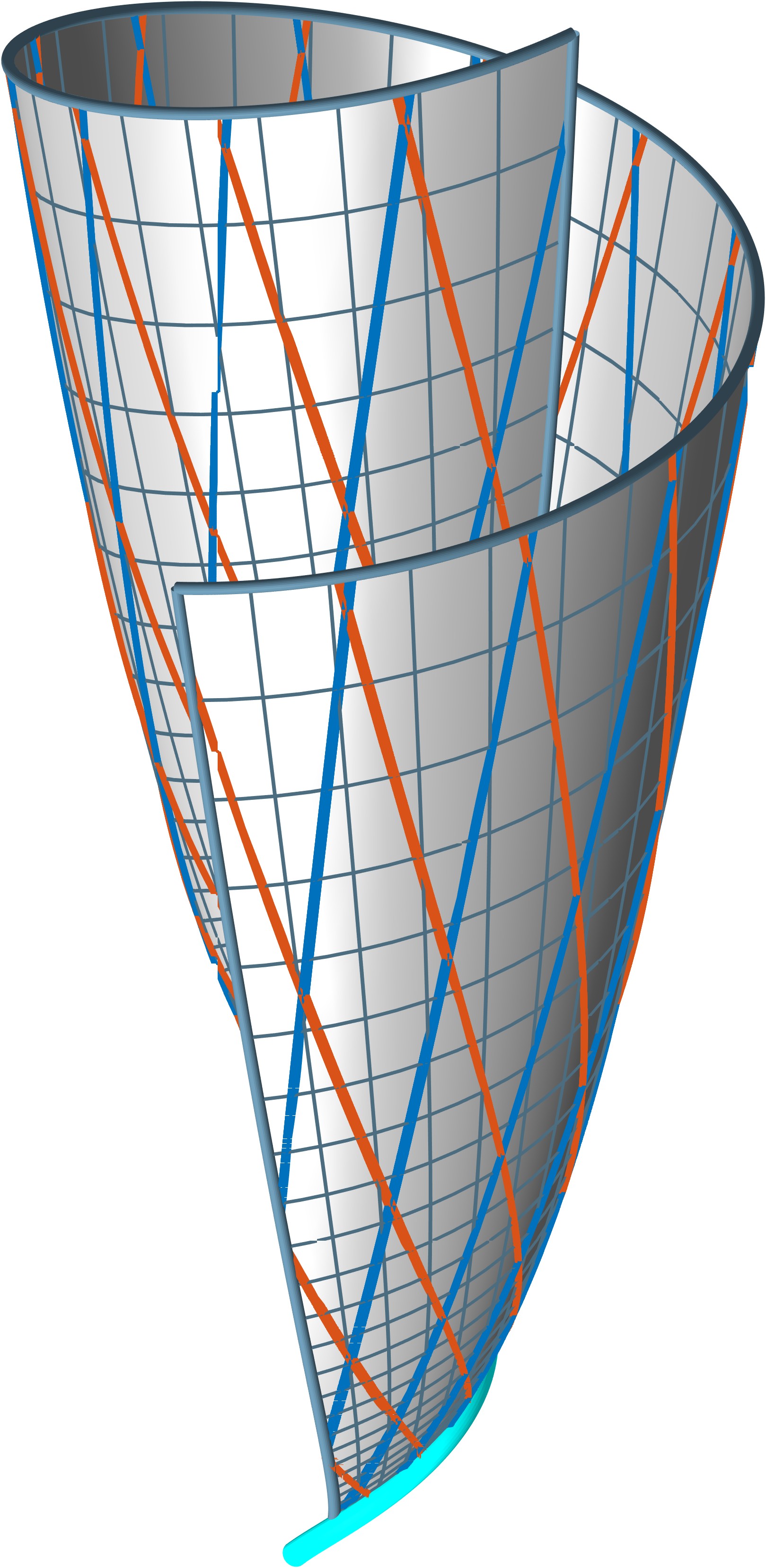}
        \end{minipage}
        \&
        \begin{minipage}[t][\cellh][t]{.235\textwidth}\centering
            \includegraphics[width=\linewidth,height=35mm,keepaspectratio]{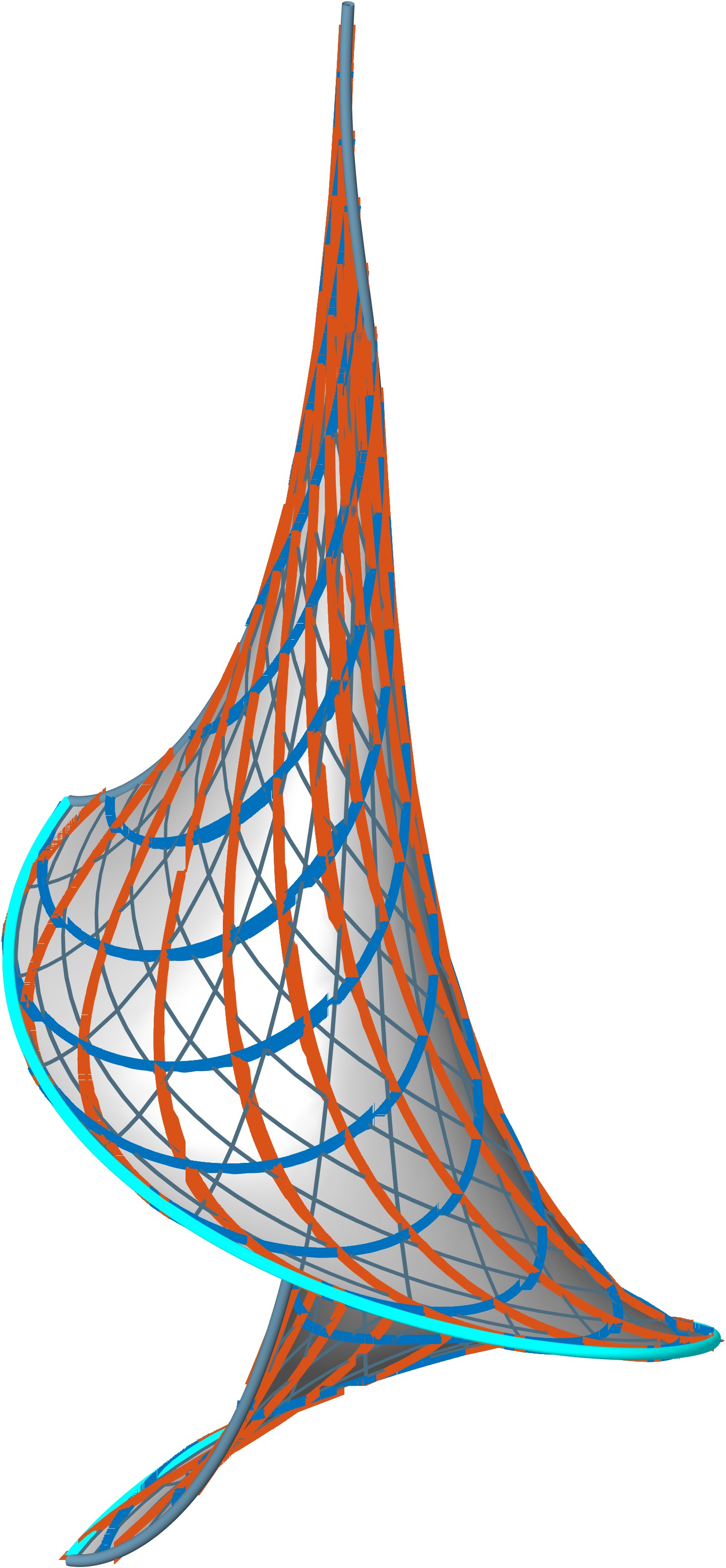}
        \end{minipage}
        \\

        \begin{minipage}[t][\cellh][t]{.235\textwidth}\centering
            \includegraphics[width=\linewidth,height=35mm,keepaspectratio]{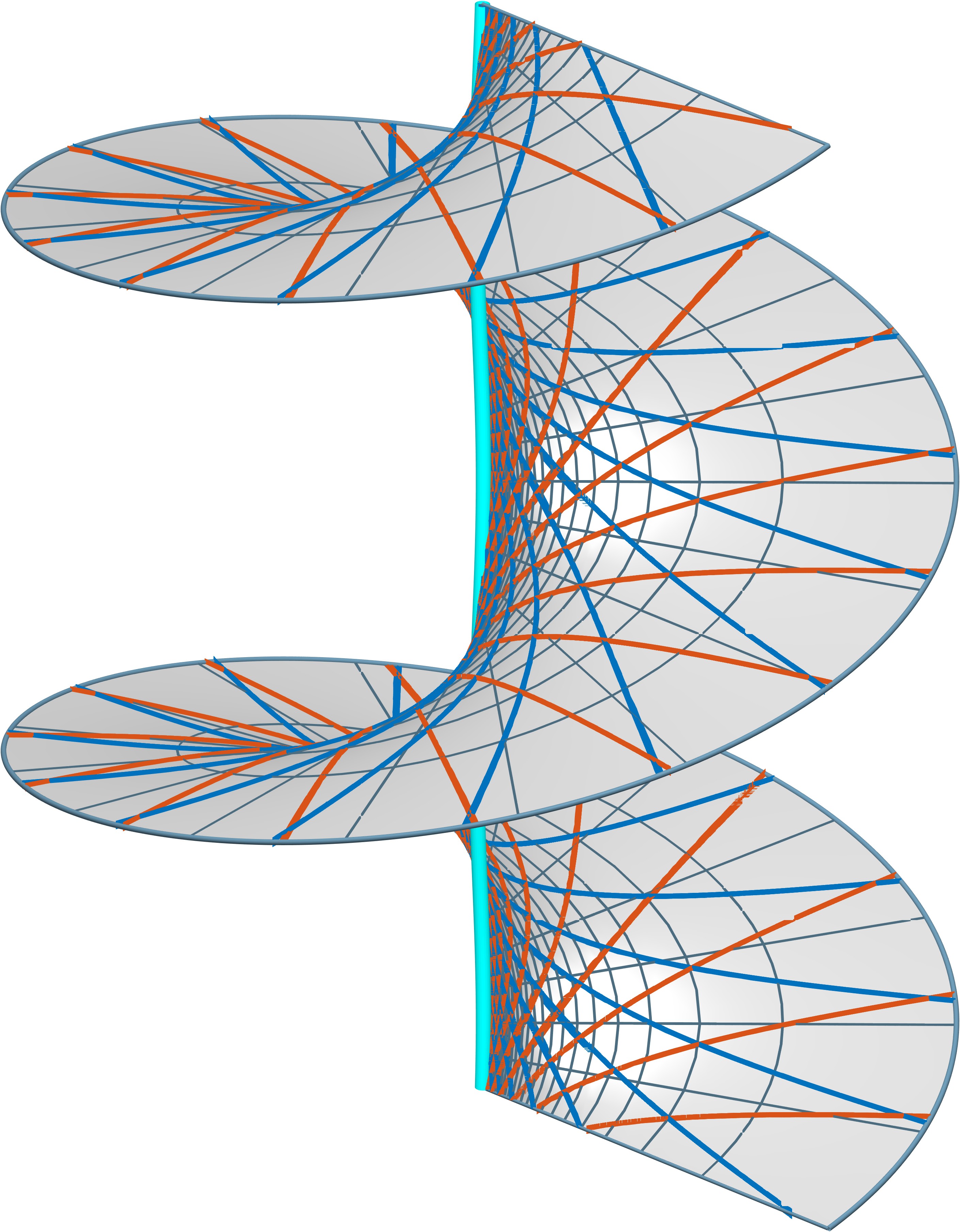}
        \end{minipage}
        \&
        \begin{minipage}[t][\cellh][t]{.235\textwidth}\centering
            \includegraphics[width=\linewidth,height=23mm,keepaspectratio]{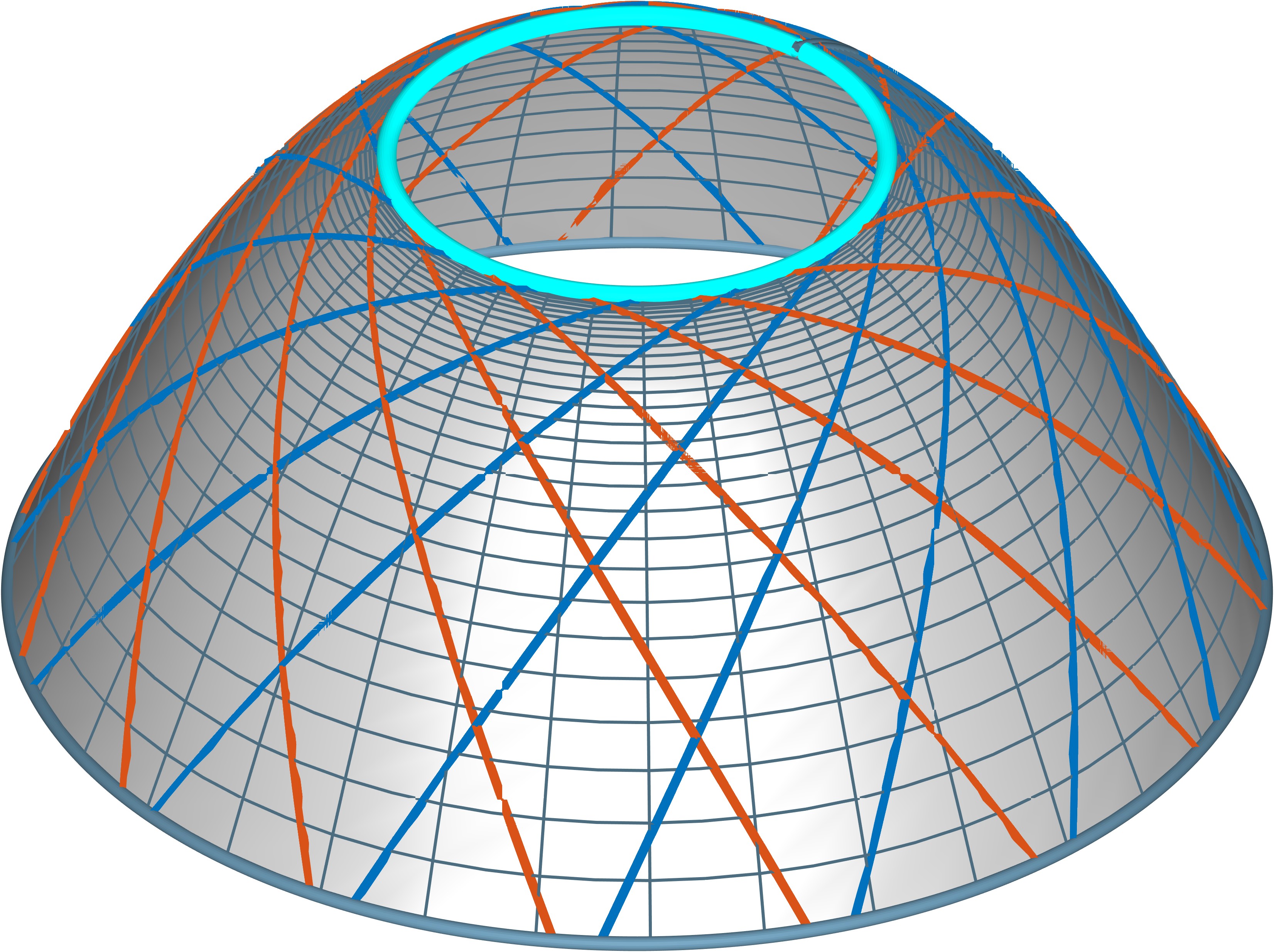}
        \end{minipage}
        \&
        \begin{minipage}[t][\cellh][t]{.235\textwidth}\centering
            \includegraphics[width=\linewidth,height=37mm,keepaspectratio]{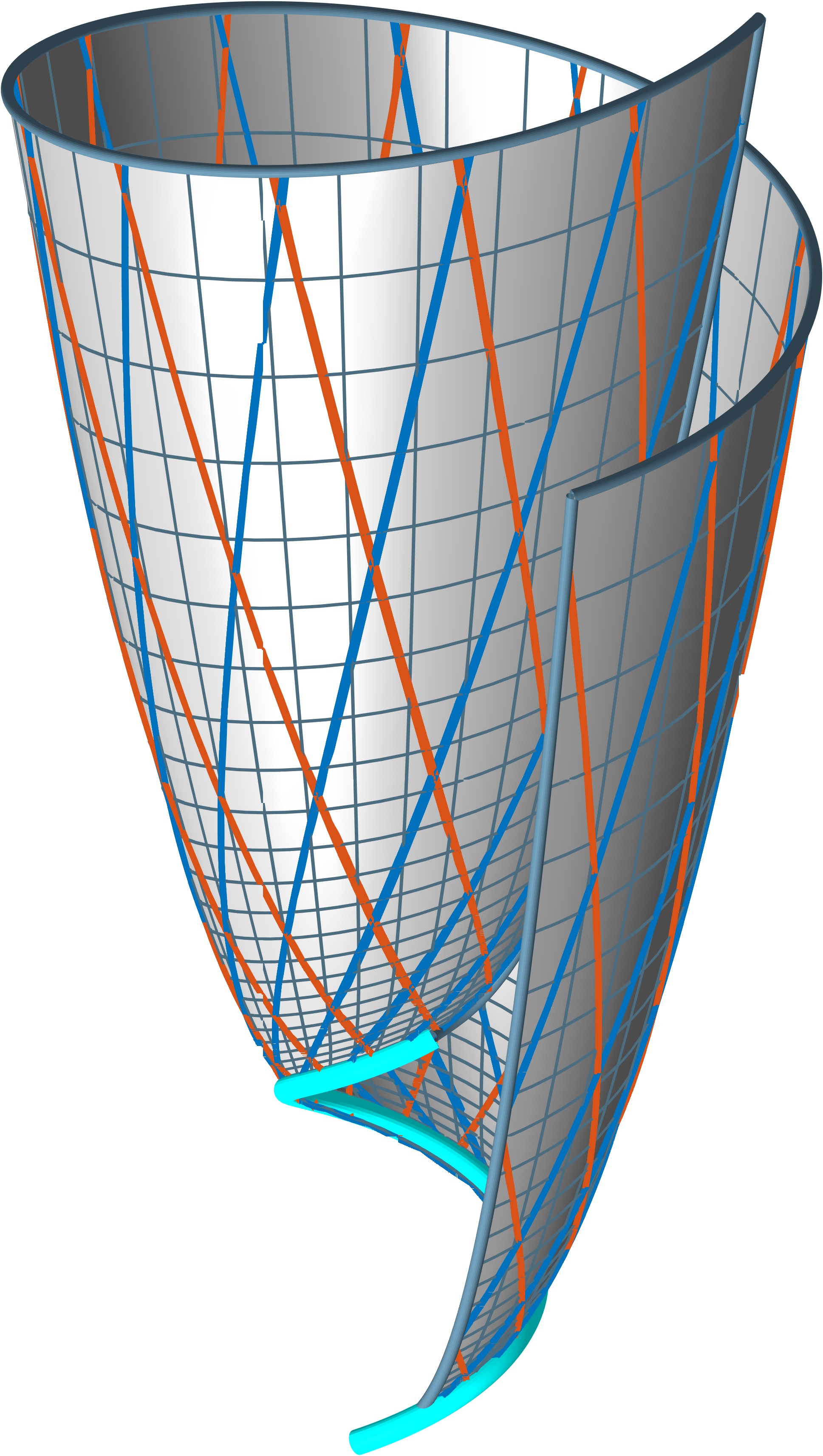}
        \end{minipage}
        \&
        \begin{minipage}[t][\cellh][t]{.235\textwidth}\centering
            \includegraphics[width=\linewidth,height=30mm,keepaspectratio]{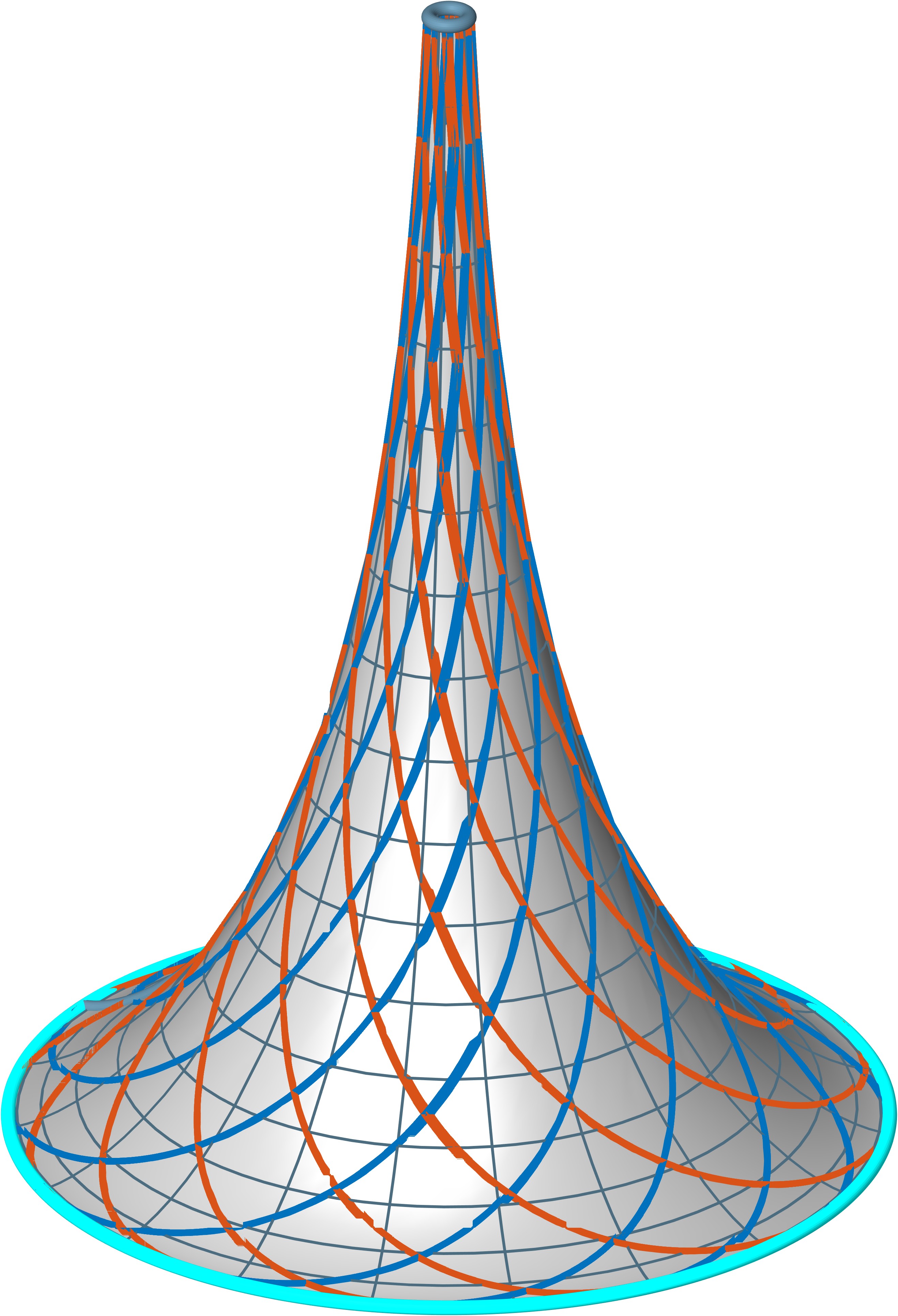}
        \end{minipage}
        \\

        \begin{minipage}[t][\cellh][t]{.235\textwidth}\centering
            \includegraphics[width=\linewidth,height=36mm,keepaspectratio]{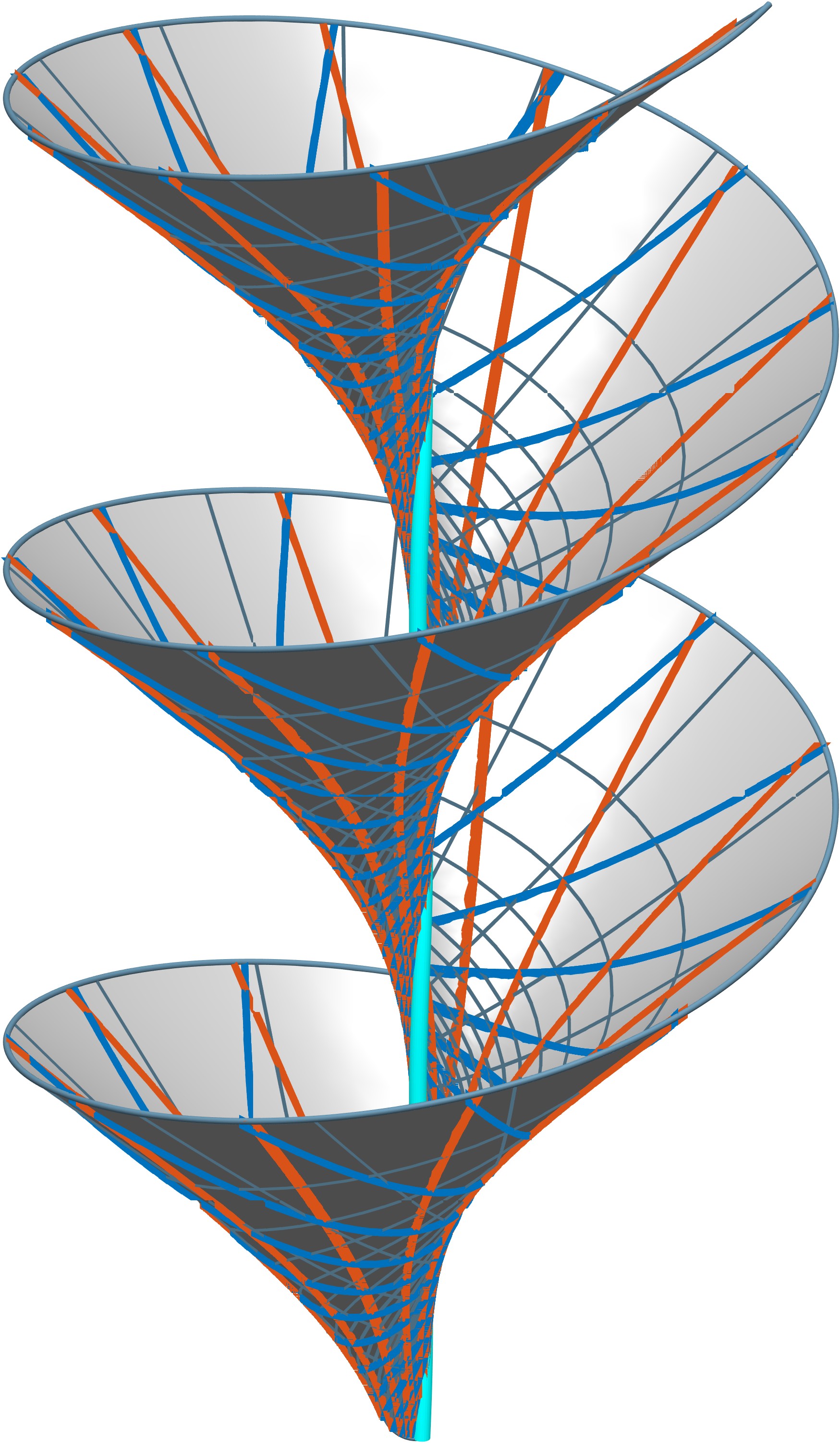}
        \end{minipage}
        \&
        \begin{minipage}[t][\cellh][t]{.235\textwidth}\centering
            \includegraphics[width=\linewidth,height=35mm,keepaspectratio]{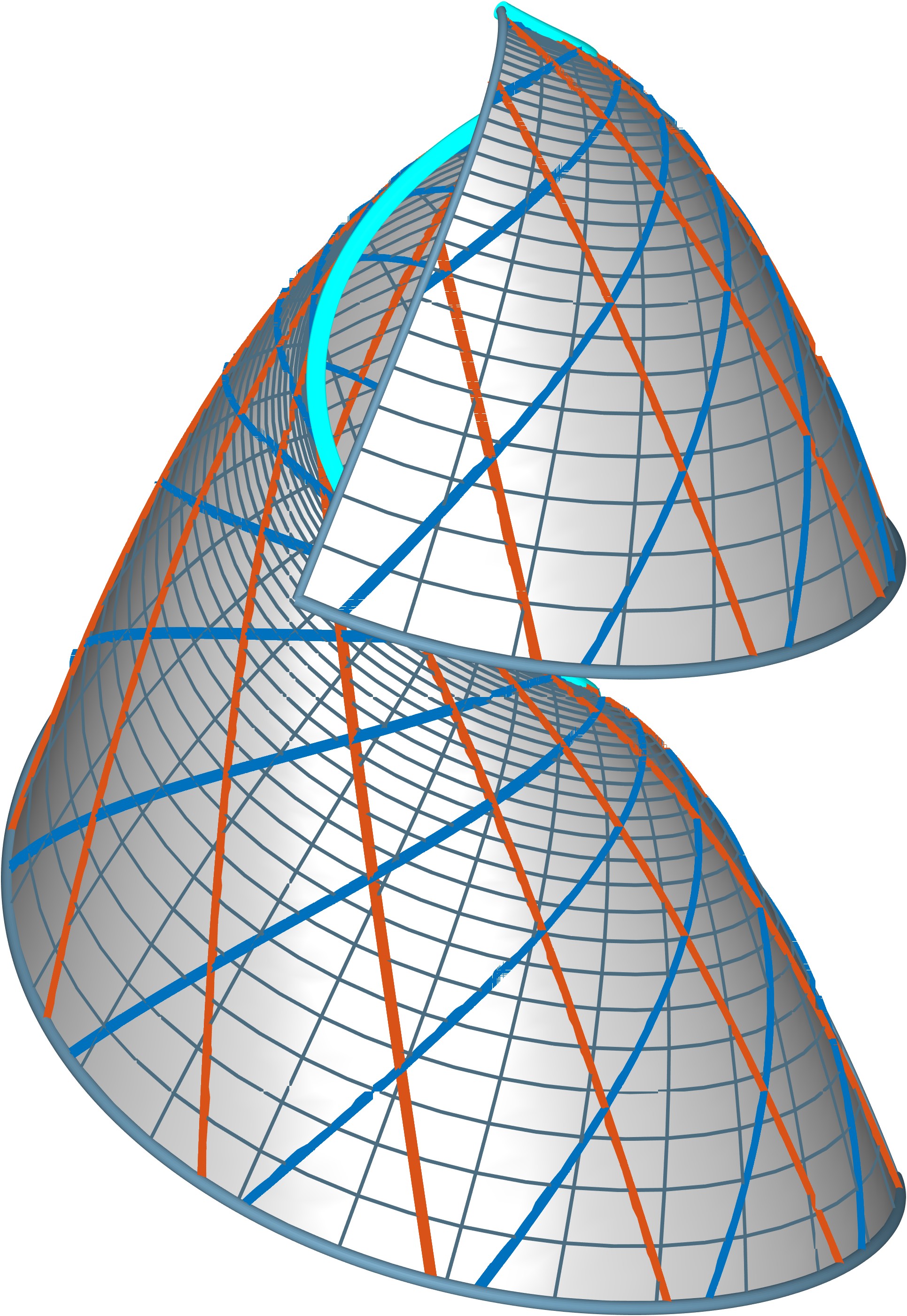}
        \end{minipage}
        \&
        \begin{minipage}[t][\cellh][t]{.235\textwidth}\centering
            \includegraphics[width=\linewidth,height=35mm,keepaspectratio]{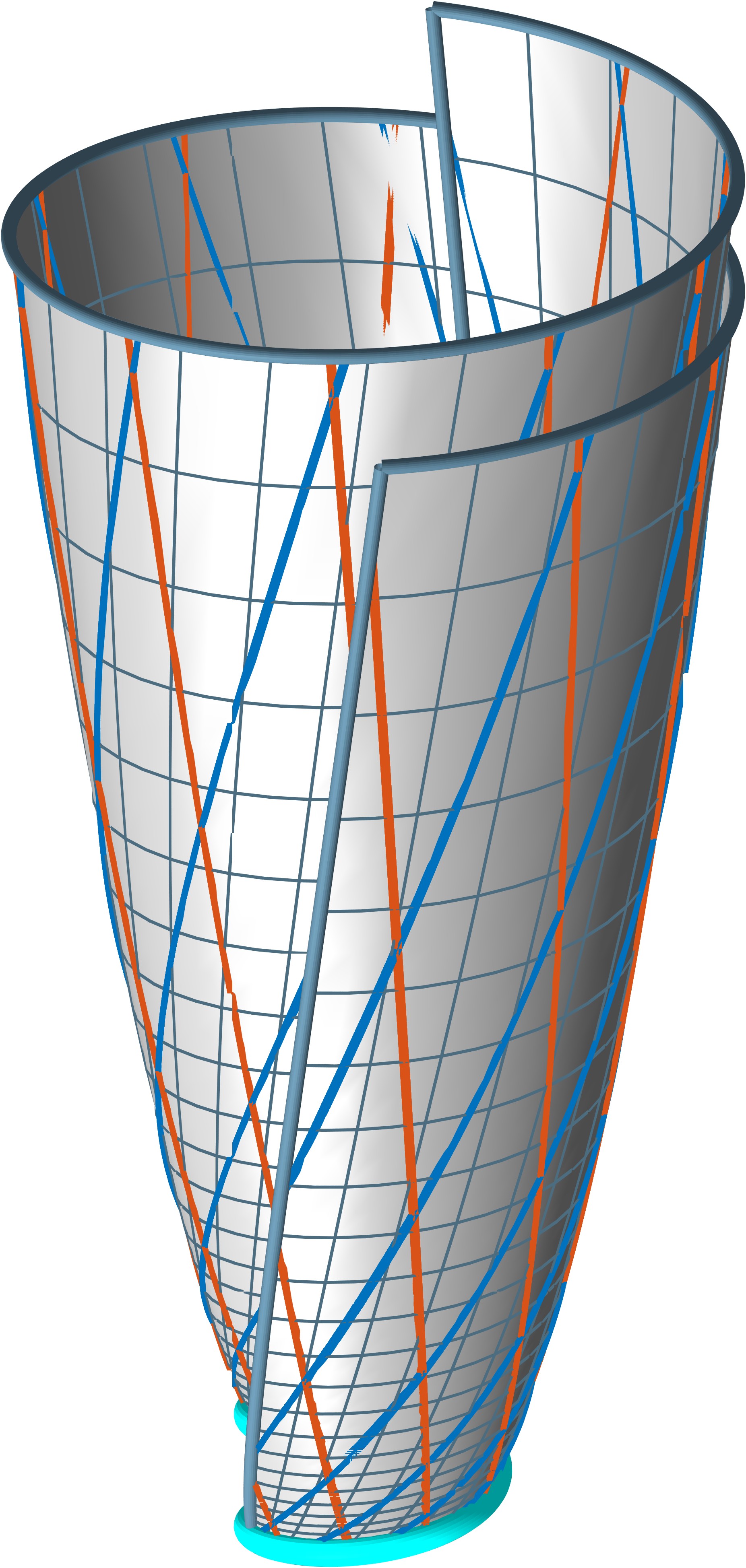}
        \end{minipage}
        \&
        \begin{minipage}[t][\cellh][t]{.235\textwidth}\centering
            \includegraphics[width=\linewidth,height=37mm,keepaspectratio]{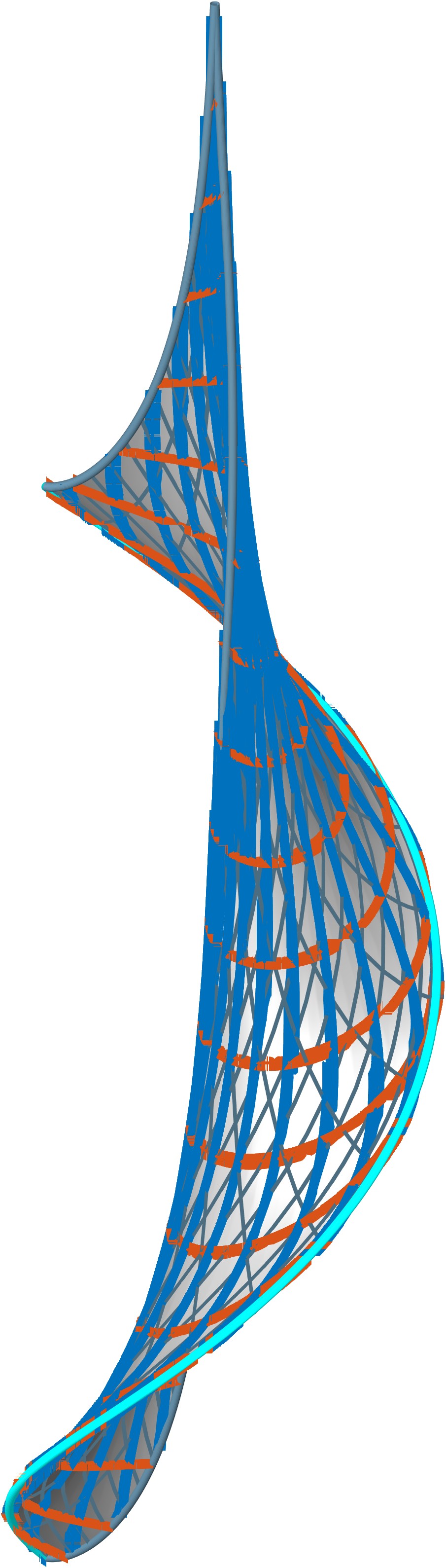}
        \end{minipage}
        \\
    };

    \foreach \r/\c/\lbl in {
        1/1/a, 1/2/b, 1/3/c, 1/4/d,
        2/1/e, 2/2/f, 2/3/g, 2/4/h,
        3/1/i, 3/2/j, 3/3/k, 3/4/l
    }{
        \node[
            draw,
            circle,
            fill=white,
            inner sep=1.2pt,
            font=\bfseries\small,
            anchor=north west
        ] at ([xshift=\labx,yshift=\laby]m-\r-\c.north west) {\lbl};
    }

\end{tikzpicture}

\caption{Illustration of the isometric deformation of alignable V-nets of the first kind with $k = 1$, of negative-alignable (a,e,i), positive-alignable (b,f,j) and linear combination (c,g,k) along with the spectral deformation of the K-net of the parabolic pseudosphere (d,h,l).}
\label{fig:iso:all}
\end{figure*}
Fig.~\ref{fig:alignable:deformation:all} depicts the alignable V-surfaces along with the corresponding V-nets while Fig.~\ref{fig:iso:all} illustrates their V-net preserving isometric deformation for the case of $k=1$. In the following discussion, we take $\Omega=\Omega_k$ and examine the limiting behavior of the V-surfaces as $x$ approaches the endpoints of the intervals of $\Omega_k$.

\paragraph{The case $k>1$}
Both $\psi$ and $\eta^+$ have cusp singularities along the lines $x=2m\,k^{-1}\mathcal K(k^{-1})$. At the intervening lines $x=(2m+1)\,k^{-1}\mathcal K(k^{-1})$, $\psi$ has cone singularities, while the radius function $k/\dn(x\mid k)$ of $\eta^+$ diverges. For $\eta^-$ and $(\eta^-)^{1,0}$, the lines $x=2m\,k^{-1}\mathcal K(k^{-1})$ correspond to ends at infinity while their finite boundaries correspond to the lines $\psi$ at $x=(2m+1)\,k^{-1}\mathcal K(k^{-1})$. For visualization purpose the cases of $\eta^-$ and $(\eta^-)^{1,0}$ are plotted for both $\pm g(x)$ (see Fig.~\ref{fig:alignable:deformation:all}-(a,b,c,d)).

\paragraph{The case $k=1$}
The line $x=0$ corresponds to a cusp singularities on both $\eta^+$ and $\psi$, whereas the corresponding curves on $\eta^-$ and $(\eta^-)^{1,0}$ escape to infinity as $x\to 0$. The cone singularities of $\psi$ (present for $k>1$) recede to $x=\pm\infty$, giving way to asymptotic ends. As $x\to\pm\infty$, the corresponding curves on $\eta^-$ approach the helicoid axis, while those on $(\eta^-)^{1,0}$ approach the catenoid's waist circle. For visualization purpose the cases of $\eta^-$ and $(\eta^-)^{1,0}$ are plotted for both $\pm g(x)$ (see Fig.~\ref{fig:alignable:deformation:all}-(e,f,g,h)). 

\paragraph{The case $0<k<1$}
Both $\eta^+$ and $\psi$ develop cusp singularities along the lines $x=2m\,\mathcal{K}(k)$, while on $\eta^-$ and $(\eta^-)^{1,0}$ the corresponding curves tend to infinity. Between $x=2m\,\mathcal{K}(k)$ and $x=(2m+2)\,\mathcal{K}(k)$ the radius of the helicoid, $\cn(x\mid k)/\sn(x\mid k)$, takes every real value and the net crosses the axis of the helicoid, and likewise the waist of the catenoid, regularly at $x=(2m+1)\,\mathcal{K}(k)$. Hence each component of $\Omega_k$ is mapped onto the entire helicoid and the entire catenoid meaning that $\eta^-$ and $(\eta^-)^{1,0}$ cover both surfaces globally. Even more, if $\tilde{\Omega}$ is a component of $\Omega_k$, then the map $\eta^-:\tilde{\Omega}\times\R\to\R^3$ is an embedding on the helicoid (see Fig.~\ref{fig:alignable:deformation:all}-(i,j,k,l)). 

\begin{remark}[\textit{Infinitely Many V-nets on a Helicoid}]\label{remark:infinitely:many:V:nets}
    As shown by Proposition~\ref{prop:rotation:field:formula} and Theorem~\ref{theorem:main:immersion:V}, every negative alignable V-net of the first kind at $t=1$ parametrizes an open set on the helicoid. Therefore, the helicoid carries infinitely many V-nets, as already noted by Voss and Gambier \cite{Gambier18,Voss}. The key point in such a phenomenon is the kinematic mechanism encoded in \eqrefi{eq:rotation:fields:SR:H}{2}; for any surface of revolution, the rotation field corresponding to the instantaneous wrapping is always an open set strip on a helicoid. In this mechanism the profile curve just determines where and how wide these strips should be. Consequently, by varying the alignability parameter $k$ in the profile curves of the K-nets of revolution (see \eqref{eq:K:net:immersion}), we effectively generate infinitely many alignable V-nets on the helicoid. 
    Considering this observation, a simply connected open patch of a helicoid, which is a V-surface, is therefore the rotation field of infinitely many K-surfaces of revolution.
\end{remark}

\begin{remark}\label{remark:linear:addition}
    By Proposition~\ref{prop:vector:space}, $\mathfrak{R}(K)$ is a vector space and hence any linear combination of rotation fields of $K$ is again a rotation field of $K$. Thus, for any scalars $\mu_1,\mu_2$, the rotation field $\eta := \mu_1\,\eta^{+}+\mu_2\,\eta^{-}$ yields a new V-net, provided that $\det(I_\eta)\neq 0$. Fig.~\ref{fig:iso:all}-(c,g,k) depicts a surface resulted form the aforementioned linear combination. 
\end{remark}

Computing the first and second fundmanetal forms of the alignable V-surfaces of the first kind from the immersion formulas of Theorem~\ref{theorem:main:immersion:V} and introducing $t = \frac{1}{2}(\lambda \mp \lambda^{-1})$ for the positive and negative cases respectively gives 
\begin{equation}\label{eq:fundamentalform:positive:first}
    \begin{aligned}
    I^+ &=
        k^4\csc^4 \left(\frac{\omega_k(u+v)}{2}\right)
        \left(
            \d u^2
            +
            2{\cos(\omega_k(u+v))}\,\d u\,\d v
            +
            \d v^2
        \right),\\
        II^+
        &=
        2k^2\cot \left(\frac{\omega_k(u+v)}{2}\right)
        \left(
            {\lambda}\,\d u^2
            +
            \frac{1}{\lambda}\,\d v^2
        \right),
    \end{aligned}
\end{equation}
for the positive alignable and 
\begin{equation}\label{eq:fundamentalform:negative:first}
    \begin{aligned}
    I^-&=
        k^4\sec^4\!\left(\frac{\omega_k(u+v)}{2}\right)
        \left(
            \d u^2
            -
            2{\cos(\omega_k(u+v))}\,\d u\,\d v
            +
            \d v^2
        \right),\\
    II^-&=
        2k^2\tan \left(\frac{\omega_k(u+v)}{2}\right)
        \left(
            {\lambda}\,\d u^2
            -
            \frac{1}{\lambda}\,\d v^2
        \right),
    \end{aligned}
\end{equation}
for the negative alignable case in the K-coframe $(\d u,\d v)$ where $\omega_k(u+v)$ is \eqref{eq:Omega:u:v}.
\begin{remark}\label{remark:V:rotation:field:V:I}
    Computing the third fundamental form of $\eta^\pm$ using \eqref{eq:fundamentalform:positive:first} and \eqref{eq:fundamentalform:negative:first}, one observes that they fulfill the criteria of Proposition~\ref{prop:rotation:field:iff} and consequently $\eta^+$ serves as a rotation field of $\eta^-$ and vice-versa.
\end{remark}
\section{Alignable V-surfaces of the Second Kind}\label{sec:2nd:kind}
By Definition~\ref{def:K:surf:patch:Amsler}, Amsler surfaces are pseudospherical surfaces admitting a $II$-Killing field $Y$ that preserves the asymptotic angle $\omega$ but not the first fundamental form. For the subsequent analysis, we use asymptotic coordinates $(u,v)$. After translating the coordinate origins and rescaling $Y$ by a nonzero constant, it reads as follows (see Remark~\ref{remark:Y}):
\begin{equation}\label{eq:Y:Amsler}
    Y=u\,\partial_u-v\,\partial_v.
\end{equation}
Since $\mathrm{d}\omega(Y)=0$, we obtain $u\,\omega_u=v\,\omega_v$. Thus, $\omega$ is constant along the curves $uv=\mathrm{const}$ of $Y$, and hence is of the form $\omega(uv)$. Combining this similarity relation with the sine-Gordon equation $\omega_{uv}=\sin\omega$, reduces the determination of the angle to an ODE. More precisely, on a region with $uv>0$, introducing $r=2\sqrt{uv}$ and regarding $\omega(uv) = w(r)$, we obtain Painlevé III equation in trigonometric form
\begin{equation}\label{eq:P:III}
     w_{rr} + r^{-1}\,w_r - \sin{(w)} = 0, \qquad\qquad r>0.
\end{equation}
By \cite[p.~512--513]{LinConte29}, all solutions of \eqref{eq:P:III} compatible with our immersion assumption of this article are the $2$-parameter transcendental ones.  
In this regard, at any fixed $r_0>0$, the initial conditions
\begin{equation}\label{eq:initial}
    w(r_0)=k_1,\qquad w_r(r_0)=k_2,\qquad 0<k_1<\pi,\qquad k_2\in\R,
\end{equation}
determine a $2$-parameter family of local analytic real solutions \cite{Prokhorov}. An Amsler surface develops cusps wherever $w$ reaches an integer multiple of $\pi$. As the singularity locus is approached from the regular part, one principal curvature tends to zero while the other becomes unbounded (see \eqref{eq:principal:curvature:K:surf}). The corresponding singular curves in the parameter domain are hyperbolas $uv={r^2_n}/{4}$, where $w(r_n)\in\pi\Z$. Accordingly, to remain within the framework of regular surface elements, we restrict our study to simply connected open sets in the parameter domain avoiding these curves.

We call an Amsler surface \emph{regular} if its angle function extends continuously to $r=0$ \cite{Popov}. The ODE of \eqref{eq:P:III} then forces $w_r(0)=0$, so this family is determined by
\begin{equation*}
    w(0) = k,\qquad\qquad
    w_r(0) = 0,\qquad\qquad
    0<k<\pi.
\end{equation*}
Thus, the regular family depends on one parameter. 
In \cite[eq.~3.14]{BobenkoAmsler}, a weakly regular asymptotic parametrization of this family is presented through the \emph{Sym formula}. This parametrization maps the lines $u=0$ and $v=0$ to two straight asymptotic lines whose intersection angle is $k$. Fig.~\ref{fig:alignable:deformation:all:2nd}-(c,f,i) shows a regular Amsler surface, using the aforementioned parametrization over the first quadrant, for different values of $k$. 

\subsection{Construction \& Isometric Deformation}

To identify the alignable V-surfaces of the second kind we resort to their corresponding solution pairs.
\begin{proposition}\label{prop:b:c}
    Let $V$ be an alignable V-surface of the second kind relative to $Y$. Then the solution pair corresponding to $V$ is $(b^+,c^+)$ in the positive-alignable and $(b^-,c^-)$ in the negative-alignable case, where
    \begin{equation}\label{eq:b:c:coordinates}
        \begin{aligned}
            (b^+,c^+)&=\left(\frac{1}{v}\csc^2\left(\frac{w(r)}{2}\right),\, \frac{1}{u}\csc^2\left(\frac{w(r)}{2}\right)\right),\\
            (b^-,c^-)&=\left(\frac{1}{v}\sec^2\left(\frac{w(r)}{2}\right),\,-\frac{1}{u}\sec^2\left(\frac{w(r)}{2}\right)\right).
        \end{aligned}
    \end{equation}
    where $w(r)$ with $r=2\sqrt{uv}$ is a solution of \eqref{eq:P:III}.
\end{proposition}
\begin{proof}
    In the $(u,v)$-coordinates adopted in this section, the V-frame takes the form $Y = u c \bar{E}_1 - v b \bar{E}_2$. From Theorem~\ref{theorem:main} we know that $Y$ bisects the V-frame. Therefore, we obtain $u\,c^\pm = \epsilon\,v\,b^\pm$ where $\epsilon = \sign(\mathrm{K}_V)$. W.l.o.g we consider the positive case in which $\epsilon = +1$:
    \begin{equation}\label{eq:bisect}
        u\,c^+(u,v) = v\,b^+(u,v).
    \end{equation}
    Substituting \eqref{eq:bisect} to the Codazzi relations of \eqref{eq:codazzi:VV}-\eqref{eq:codazzi:V} and using $u\,\omega_u=v\,\omega_v$ yields
    \begin{equation}\label{eq:sol}
        b^+(u,v) = \tilde{v}(u,v)\csc^2\left(\frac{\omega(uv)}{2}\right),\qquad
        c^+(u,v) = \tilde{u}(u,v)\csc^2\left(\frac{\omega(uv)}{2}\right),
    \end{equation}
    for some $\tilde{u},\tilde{v} \in C^\infty(\Omega)$ with $\partial_u\tilde{v} = \partial_v\tilde{u} = 0$.
    Substituting \eqref{eq:sol} back to \eqref{eq:bisect} gives $v\tilde{v} = u\tilde{u} =: \phi$.
    Using $\partial_u\tilde{v} = \partial_v\tilde{u} = 0$ and Remark~\ref{remark:Y} implies that $\d\phi(\partial_u) = \d\phi(\partial_v) = 0$ and hence $\d\phi = 0$. This gives $\phi = \const$ and therefore, w.l.o.g we have $v\tilde{v} = u\tilde{u} =1$. Solving it for $(\tilde{u},\tilde{v})$ and substituting it in \eqref{eq:sol} gives \eqrefi{eq:b:c:coordinates}{1}. The negative-alignable case follows analogously by taking $\epsilon = -1$, yielding \eqrefi{eq:b:c:coordinates}{2}.
\end{proof}

By the reconstruction discussed in Section~\ref{sec:V:surfaces:V:nets}, once an Amsler K-net $\psi\colon\Omega\to\R^3$ is given, the solution pairs of \eqref{eq:b:c:coordinates} determine, uniquely up to translation, immersions $\eta^\pm\colon\Omega\to\R^3$ satisfying
\begin{equation}\label{eq:immersion:2nd:kind}
    \d \eta^\pm = \d\psi \circ A^\pm,
\end{equation}
where $A^\pm$ is \eqref{eq:A:K:frame} with \eqref{eq:b:c:coordinates} being substituted in it.
Moreover, since $\psi$ and $A^\pm$ are real analytic on $\Omega$, so will be $\eta^\pm$.

ISR have the property that a surface patch can be moved along the surface without stretching or distorting it. 
In \cite{Pinkall}, several classical examples of ISR, including the \emph{Enneper surface} and \emph{parabolic pseudosphere}, are discussed. This property is of particular interest to us. As elaborated in the following proposition, if one chooses a V-net patch on an alignable V-surface of the second kind, the V-net preserving isometric deformation moves this patch along the ambient V-surface.  

\begin{proposition}\label{proposition:isometric:second}
    Let $V$ be an alignable V-surface of the second kind relative to $Y$. Then the flow of $Y$ induces a non-trivial isometric deformation of $V$ preserving the induced V-net.
\end{proposition}

\begin{proof}
    Let $V=(\eta,\Omega)$, and denote the local flow of $Y$ by $\phi^t$. Fix $p\in\Omega$. Choose a simply connected open neighborhood $U\subset\Omega$ of $p$ and $\varepsilon>0$ such that $\phi^t(U)\subset\Omega$ for every $t \in (-\varepsilon,\varepsilon)$, and define $\eta^t:=(\eta\circ\phi^t)|_U$. By Corollary~\ref{corollary:alignble:is:geodesic:parallel:surf}, $Y$ satisfies $\mathcal{L}_Y I = 0$ and $\mathcal{L}_Y II \neq 0$ in the second-kind case. Consequently, we get
    \begin{equation*}
        I^t=((\phi^{t})^\ast I)\left|_U\right. = I\left|_U\right.,
    \end{equation*}
    while the second fundamental form varies nontrivially. After decreasing $\varepsilon$ if necessary, the maps $\eta^t$ therefore form a nontrivial local isometric deformation of $V$. Moreover, \eqref{eq:Yb:Yc:one}, together with \eqref{eq:e1:e3:bare1:bare3}, give $\mathcal{L}_Y\bar E_1=\mathcal{L}_Y\bar E_2=0$.
    Hence, the flow preserves both families of $\bar{E}_1$- and $\bar{E}_2$-lines. Their geodesic property is preserved by the local isometry. Their conjugate property follows from the invariance of the frame, since
    \begin{equation*} 
        II^t(\bar{E}_1,\bar{E}_2)=\bigl((\phi^t)^\ast II\bigr)(\bar{E}_1,\bar{E}_2)=(\phi^t)^*\!\bigl(II(\bar{E}_1,\bar {E}_2)\bigr)=0.
    \end{equation*}
    Thus, the deformation preserves the induced V-net.
\end{proof}

Using Proposition~\ref{proposition:isometric:second}, we now classify these surfaces at the level of their fundamental forms.

\begin{theorem}\label{theorem:classification:second}
    The alignable V-surfaces of the second kind are locally represented by positive- and negative-alignable $3$-parameter families of V-nets $(\eta^{\pm,t}_{k_1,k_2},\Omega)$. Here, $t$ is the isometric deformation parameter, while $k_1$ and $k_2$ specify the initial data of \eqref{eq:initial}. Up to a box reparametrization, these V-nets are characterized, by the following fundamental forms:
    \begin{equation}\label{eq:fundamentalform:positive:second}
        \begin{aligned}
        I^+&=
        \csc^4 \left(\frac{w(r)}{2}\right)
        \left(
            \frac{1}{u^2}\,\d u^2
            +
            2\frac{\cos(w(r))}{uv}\,\d u\,\d v
            +
            \frac{1}{v^2}\,\d v^2
        \right),\\
        II^+ &=
        2\cot \left(\frac{w(r)}{2}\right)
        \left(
            \frac{\lambda}{u}\,\d u^2
            +
            \frac{1}{\lambda v}\,\d v^2
        \right),
        \end{aligned}
    \end{equation}
    while in the negative-alignable case they are given by
    \begin{equation}\label{eq:fundamentalform:negative:second}
        \begin{aligned}
        I^- &=
        \sec^4\!\left(\frac{w(r)}{2}\right)
        \left(
            \frac{1}{u^2}\,\d u^2
            -
            2\frac{\cos(w(r))}{uv}\,\d u\,\d v
            +
            \frac{1}{v^2}\,\d v^2
        \right),
        \\
        II^- &=
        2\tan \left(\frac{w(r)}{2}\right)
        \left(
            -\frac{\lambda}{u}\,\d u^2
            +
            \frac{1}{\lambda v}\,\d v^2
        \right),
        \end{aligned}
    \end{equation}
    where $\lambda = \exp(t)$, while $w(r)$ with $r=2\sqrt{uv}$ is a solution \eqref{eq:P:III}.
\end{theorem}

\begin{proof}
    For $t=0$, the result follows directly from Proposition~\ref{prop:b:c} by substituting \eqref{eq:b:c:coordinates} into \eqref{eq:I:sigma} and \eqref{eq:II:sigma}. 
    
    Now, by Proposition~\ref{proposition:isometric:second}, the flow of $Y$ induces a V-net preserving isometric deformation. In the coordinate setting discussed at the beginning of Section~\ref{sec:2nd:kind}, the local flow of $Y$ is 
    \begin{equation}\label{eq:squeeze}
        \phi^t(u,v) = (\lambda\,u,\lambda^{-1}\,v).
    \end{equation}
    where $\lambda = \mathrm{e}^t$.
    After restricting to a smaller parameter domain $\Omega$ and sufficiently small $t$, the deformation is represented on this fixed domain by
    $\eta^{\pm,t}_{k_1,k_2}=\eta^{\pm,0}_{k_1,k_2}\circ\phi^t$. 
    Since $(\lambda u)(\lambda^{-1}v)=uv$, both $r=2\sqrt{uv}$ and the angle $w(r)$ remain unchanged. Moreover,
    \begin{gather*}
        \phi^{t\ast}\!\left(\frac{\d u^2}{u^2}\right)
        =
        \frac{\d u^2}{u^2},
        \quad
        (\phi^{t})^\ast\!\left(\frac{\d u\,\d v}{uv}\right)
        =
        \frac{\d u\,\d v}{uv},
        \quad
        (\phi^{t})^\ast\!\left(\frac{\d v^2}{v^2}\right)
        =
        \frac{\d v^2}{v^2},\\
        (\phi^{t})^\ast\!\left(\frac{\d u^2}{u}\right)
        =
        \frac{\lambda}{u}\,\d u^2,
        \quad
        (\phi^{t})^\ast\!\left(\frac{\d v^2}{v}\right)
        =
        \frac{\lambda^{-1}}{v}\,\d v^2.
    \end{gather*}
    Consequently, pulling back the fundamental forms at $t=0$ therefore gives \eqref{eq:fundamentalform:positive:second}--\eqref{eq:fundamentalform:negative:second}, completing the classification.
\end{proof}

\begin{remark}\label{remark:V:rotation:field:V:II}
    Computing the third fundamental form of $\eta^\pm$ using the results of Theorem~\ref{theorem:classification:second}, one observes that they fulfill the criteria of Proposition~\ref{prop:rotation:field:iff} and consequently $\eta^+$ serves as a rotation field of $\eta^-$ and vice-versa.
\end{remark}

\begin{remark}
    Once the corresponding Amsler K-net is given, the immersion formulas for an alignable V-surface of the second kind, up to a single quadrature, are retrievable from \eqref{eq:immersion:2nd:kind}. The corresponding isometric deformation is obtained via applying the squeeze transformation \eqref{eq:squeeze}.
\end{remark}

\subsection{Missing Cases in Eisenhart's Classification}\label{subsec:eisenhart}

Theorem~\ref{theorem:classification:second} reveals a class of isothermal-conjugate V-nets that is absent from Eisenhart's classification. In \cite[p. 536]{EisenhartAssociate}, Eisenhart claims the following
\begin{quote}
    ``Minimal surfaces and certain associates of the pseudospherical surfaces of revolution are the only Voss surfaces with an isothermal-conjugate system of geodesic lines.''
\end{quote}
In Eisenhart’s terminology, the associate surfaces of a surface are precisely its rotation fields (see \cite{EisenhartAssociate, EisenhartTreatise, Eisenhart1914, Eisenhart1915, Eisenhart1917}). 
For each fixed $t$, consider a V-net from either family in Theorem~\ref{theorem:classification:second} and apply the box reparametrization
\begin{equation}\label{eq:Eisenhart:box}
    (u,v)\longmapsto(\tilde u,\tilde v) = (2\sqrt{\lambda u},2\sqrt{\lambda^{-1}v}).
\end{equation}
Then the resulting parametrizations are again V-nets and their second fundamental forms become
\begin{equation}\label{eq:tildeu:tilde:v}
    {II}^+ = 2\,\cot{\left(\frac{\tilde{\omega}(\tilde{u},\tilde{v})}{2}\right)}\,\left(\mathrm{d}\tilde{u}^2 +\mathrm{d}\tilde{v}^2\right),\quad 
    {II}^- = 2\,\tan{\left(\frac{\tilde{\omega}(\tilde{u},\tilde{v})}{2}\right)}\,\bigl(-\mathrm{d}\tilde{u}^2 + \mathrm{d}\tilde{v}^2\bigr),
\end{equation}
where $\widetilde{\omega}(\tilde u,\tilde v)= w\left(\frac{\tilde u\tilde v}{2}\right)$.
Thus, both families admit isothermal-conjugate V-nets.
These surfaces do not belong to the minimal class appearing in Eisenhart's statement since their mean curvatures are
\begin{equation}
    H^+ = \frac{\tilde u^2+\tilde v^2}{16} \tan \left(\frac{\widetilde{\omega}(\tilde{u}\tilde{v})}{2}\right),
    \qquad
    H^- = \frac{\tilde v^2-\tilde u^2}{16}\cot \left(\frac{\widetilde{\omega}(\tilde{u}\tilde{v})}{2}\right),
\end{equation}
which are not identically zero on a regular domain. Nor do these V-nets belong to Eisenhart's associates of pseudospherical surfaces of revolution.
In the $(\tilde u,\tilde v)$-coordinates the fundamental forms of the corresponding K-net are
\begin{equation}
    {I}_{\psi^t}
    = \frac{1}{4}\left(\tilde{u}^2\,\d\tilde{u}^{2}
    + 2\,\tilde{u}\tilde{v}\cos(\tilde{\omega})\,\d\tilde{u}\,\d\tilde{v}
    + \tilde{v}^2\,\,\mathrm{d}\tilde{v}^{2}\right),\quad
    {II}_{\psi^t}
    = \left(\frac{\tilde{u}\tilde{v}}{2}\right)\,\sin(\tilde{\omega})\,\d\tilde{u}\,\d\tilde{v},
\end{equation}
which implies that, this is a non-canonical K-net i.e., $\partial_{\tilde{u}}$ and $\partial_{\tilde{v}}$ are not unit vectors. This is where the source of the missing case lies. Eisenhart carries out his analysis using a canonical K-net. However, the property of a V-net being isothermal-conjugate may appear only after a box reparametrization of the corresponding V-net. Such a reparametrization preserves the geodesic-conjugate property but can neglect the canonical normalization of the K-net. Hence, restricting the analysis to canonical K-nets does not capture all V-nets that can become isothermal-conjugate after an admissible box reparametrization. The second kind families of Theorem~\ref{theorem:classification:second} therefore provide additional isothermal-conjugate V-nets beyond those appearing in Eisenhart's classification. Whether there exist additional such V-nets remains an open question for future study.

\begin{remark}
    Eisenhart’s classification also overlooks the fact that, aside from the helicoid and the surfaces of revolution of the isometric family of alignable V-nets of the first kind, their other members can also be transformed into isothermal-conjugate V-nets by the box reparametrization $(u,v)\mapsto (\sqrt{\lambda}\,u,\sqrt{\lambda^{-1}}\,v)$ (see \eqref{eq:fundamentalform:positive:first} and \eqref{eq:fundamentalform:negative:first}).
\end{remark}

\subsection{Alignable V-surfaces of the Second Kind in Large}

The preceding constructions were local. We now consider them on an entire connected component of the parameter domain. Fix a nonconstant real solution $w\colon(0,+\infty)\to\R$ of \eqref{eq:P:III} and let $\mathscr{D}$ be a connected component of  $\left\{(u,v)\in(0,\infty)^2 \mid \sin\bigl(w(2\sqrt{uv})\bigr)\ne0\right\}$. Each such component has the form 
\begin{equation}\label{eq:paramter:domain:in:large}
    \mathscr{D}=\left\{(u,v)\in(0,\infty)^2\mid uv\in J\right\},
\end{equation}
for an open interval $J\subset(0,\infty)$. Since the map $(u,v)\mapsto\left(\log\left(uv\right),\log\left(\frac{u}{v}\right)\right)$, identifies $\mathscr{D}$ diffeomorphically with $\log(J)\times\R$, $\mathscr{D}$ is simply connected.

For $\omega(u,v)=w(2\sqrt{uv})$, the fundamental forms \eqref{eq:I:II:psi} of the corresponding K-net are real analytic and satisfy the Gauss and Codazzi equations on $\mathscr{D}$. Moreover, the first fundamental form is positive definite because its determinant is $\sin^2(\omega)>0$. The \emph{fundamental theorem of surface theory} (see \cite{Recovery}) yields a real analytic K-net $\psi\colon\mathscr{D}\to\R^3$, unique up to a rigid motion. 

The solution pairs of \eqref{eq:b:c:coordinates} define the rotation operators $A^\pm$ on the whole $\mathscr{D}$. Since the one-forms $\d\psi\circ A^\pm$ are closed and real analytic, and $\mathscr{D}$ is simply connected, there exist real analytic maps $\eta^\pm\colon\mathscr{D}\to\R^3$, unique up to translation, satisfying \eqref{eq:immersion:2nd:kind}. These maps are immersions because $\psi$ is an immersion and $A^\pm$ are pointwise invertible. We call the surface elements $\mathscr{V}^\pm = (\eta^\pm,\mathscr{D})$, positive- and negative-alignable V-surfaces of the second kind \emph{in large}, respectively. 

\begin{proposition}\label{prop:V:V:t}
    Let $\mathscr{V}$ be an alignable V-surface of the second kind in large and let $\mathscr{V}^t$ be V-net preserving isometric to it. Then $\mathscr{V}^t$ is congruent to $\mathscr{V}$.
\end{proposition}
\begin{proof}
    By construction, the parameter domain $\mathscr{D}$ and Killing field $Y$ of $\mathscr{V}$ are \eqref{eq:paramter:domain:in:large} and \eqref{eq:Y:Amsler} respectively. Now, the flow $\phi^t$ of $Y$ \eqref{eq:squeeze}, evidently is a diffeomorphism and results in $\phi^t(\mathscr{D}) = \mathscr{D}$ for every $t \in \R$.
    The preceding characterization of the V-net preserving isometric family shows that, up to a rigid motions, its members are represented by $\eta^t = \eta \circ \phi^t$. Consequently, the immersions $\eta^t:\mathscr{D}\to\R^3$ and $\eta:\mathscr{D}\to\R^3$ represent the same surface element.
\end{proof}
Gambier already observed this self-congruence phenomenon in \cite[p.~105--106]{Gambier18}. However, we recorded it here in the form needed for the following proposition.

\begin{proposition}
    Let $\mathscr{V}^\pm = (\eta^\pm,\mathscr{D})$ be an alignable V-surface of the second kind in large. Then there exists a
    \begin{itemize}
    \item planar curvature line $\gamma^+ \subset \mathscr{V}^+$ such that $\mathscr{V}^+$ is invariant under reflection in the carrier plane of $\gamma^+$,
    \item straight line $\gamma^- \subset \mathscr{V}^-$ such that $\mathscr{V}^-$ is invariant under rotation by $\pi$ about the carrier line of $\gamma^-$.
    \end{itemize} 
\end{proposition}

\begin{proof}
    Due to Proposition~\ref{prop:V:V:t}, w.l.o.g, we restrict our attention to $\mathscr{V}^\pm$ with fundamental forms $I^\pm$ and $II^\pm$ given in \eqref{eq:fundamentalform:positive:second} and \eqref{eq:fundamentalform:negative:second} for $t = 0$.
    Since $\mathscr{D}$ is of the form \eqref{eq:paramter:domain:in:large}, exchanging $u$ and $v$ preserves $\mathscr{D}$.
    With that in mind, define $\sigma :\mathscr{D} \rightarrow \mathscr{D}$ by $\sigma(u,v) = (v,u)$. Then we get  
    \begin{equation}\label{eq:sigma:I:II:1}
        \sigma^\ast I^\pm = I^\pm,\qquad\qquad
        \sigma^\ast II^+ = II^+,\qquad\qquad
        \sigma^\ast II^- = -II^-.
    \end{equation}
    As a result of this, $\sigma$ is an isometry of $\mathscr{V}^\pm$. Its fixed point set $\Gamma := \{(u,v)\in \mathscr{D}\mid u = v\}$ is therefore totally geodesic. Consequently, $\gamma^\pm := \eta^\pm (\Gamma)$ are geodesics on $\mathscr V^\pm$. Along $\Gamma$, consider the unnormalized ISR frame 
    \begin{equation*}
        E_3 = -(\partial_u+\partial_v),
        \qquad\qquad
        E_4 = \partial_u-\partial_v.
    \end{equation*}
    From \eqref{eq:X:Y:pm} and the second fundamental forms we obtain $II^+(E_3,E_4)=0$ and $II^-(E_3,E_3)=0$. Thus $\gamma^+$ has vanishing geodesic torsion and is therefore a curvature line. Since it is also a geodesic, its torsion vanishes, and hence $\gamma^+$ is planar. Likewise, $\gamma^-$ has both vanishing geodesic curvature and vanishing normal curvature, and therefore zero curvature. Hence $\gamma^-$ is a straight line. Now, define $\tilde{\eta}^\pm = \eta^\pm\circ \sigma$. Since $\sigma$ reverses the orientation of $\mathscr{D}$, we have $\tilde{II}^\pm=-\sigma^\ast II^\pm$.
    Therefore, using \eqref{eq:sigma:I:II:1}, we retrieve
    \begin{equation}\label{eq:sigma:I:II}
        \tilde{I}^{\,\pm}=I^{\,\pm},
        \qquad\qquad
        \tilde{II}^{\,+}=-II^{\,+},
        \qquad\qquad
        \tilde{II}^{\,-}=II^{\,-}.
    \end{equation}
    Therefore, by the \emph{fundamental theorem of surface theory}, there exist rigid motions $\phi^{\pm}:\mathbb{R}^3 \rightarrow\mathbb{R}^3$ with $\phi^\pm(p)=\mathbf R^\pm p+\mathbf T^\pm$, $\mathbf R^\pm\in \mathrm{O}(3)$ and $\mathbf{T}^\pm \in \mathbb{R}^3$ such that $\eta^\pm\circ\sigma=\phi^\pm\circ \eta^\pm$ with $\det(\mathbf{R}^\pm)=\mp 1$.
    It remains to identify these rigid motions. Along $\Gamma$ we have $\d\sigma(E_3)=E_3$ and $\d\sigma(E_4)=-E_4$. Differentiating $\eta^\pm\circ\sigma=\phi^\pm\circ\eta^\pm$ along $\Gamma$ yields
    \begin{equation*}
        \mathbf R^\pm\,\d \eta^\pm(E_3)=\d \eta^\pm(E_3),
        \qquad\qquad
        \mathbf R^\pm\,\d \eta^\pm(E_4)=\mp\d \eta^\pm(E_4).
    \end{equation*}
    Thus, $\mathbf R^\pm$ has eigenvalue $+1$ in the direction $\d\eta^\pm(E_3)$ and eigenvalue $\mp1$ in the direction $\d\eta^\pm(E_4)$. Moreover, since $\sigma$ fixes $\Gamma$ pointwise, for each $p\in\Gamma$, we have $\phi^\pm\bigl(\eta^\pm(p)\bigr)=\eta^\pm(p)$. Since $\det(\mathbf{R}^+)=-1$ and fixes the direction $\d\eta^+(E_3)$, it is a plane reflection. Together with the fixed-point condition, this shows that $\phi^+$ is a reflection in a plane containing $\gamma^+$. Since $\gamma^+$ is nonstraight, this is its carrier plane, and $\mathscr{V}^+$ is invariant under reflection in this plane. Similarly, since $\det(\mathbf{R}^-)=1$ and $\phi^-$ fixes $\gamma^-$ pointwise, $\phi^-$ is a rotation about its carrier line. The relation $\mathbf R^-\d\eta^-(E_4)=-\d\eta^-(E_4)$, with $\d\eta^-(E_4)$ perpendicular to this line, shows that the rotation angle is $\pi$. Thus, $\mathscr{V}^-$ is invariant under rotation by $\pi$ about the carrier line of $\gamma^-$.
\end{proof}
\begin{figure*}[t!]
\centering
\begin{tikzpicture}

    \def\labx{2mm}
    \def\laby{-2mm}
    \def\cellh{42mm}

    \def\maincx{5mm}
    \def\maincy{10.9mm}
    \def\rmain{3.5mm}
    \def\rinset{8mm}

    \matrix (m) [
        matrix of nodes,
        ampersand replacement=\&,
        row sep=2mm,
        column sep=1.5mm,
        nodes={
            anchor=north west,
            inner sep=0pt,
            outer sep=0pt
        }
    ]{
        \begin{minipage}[t][\cellh][t]{.31\textwidth}\centering
            \includegraphics[
                width=\linewidth,
                height=20mm,
                keepaspectratio
            ]{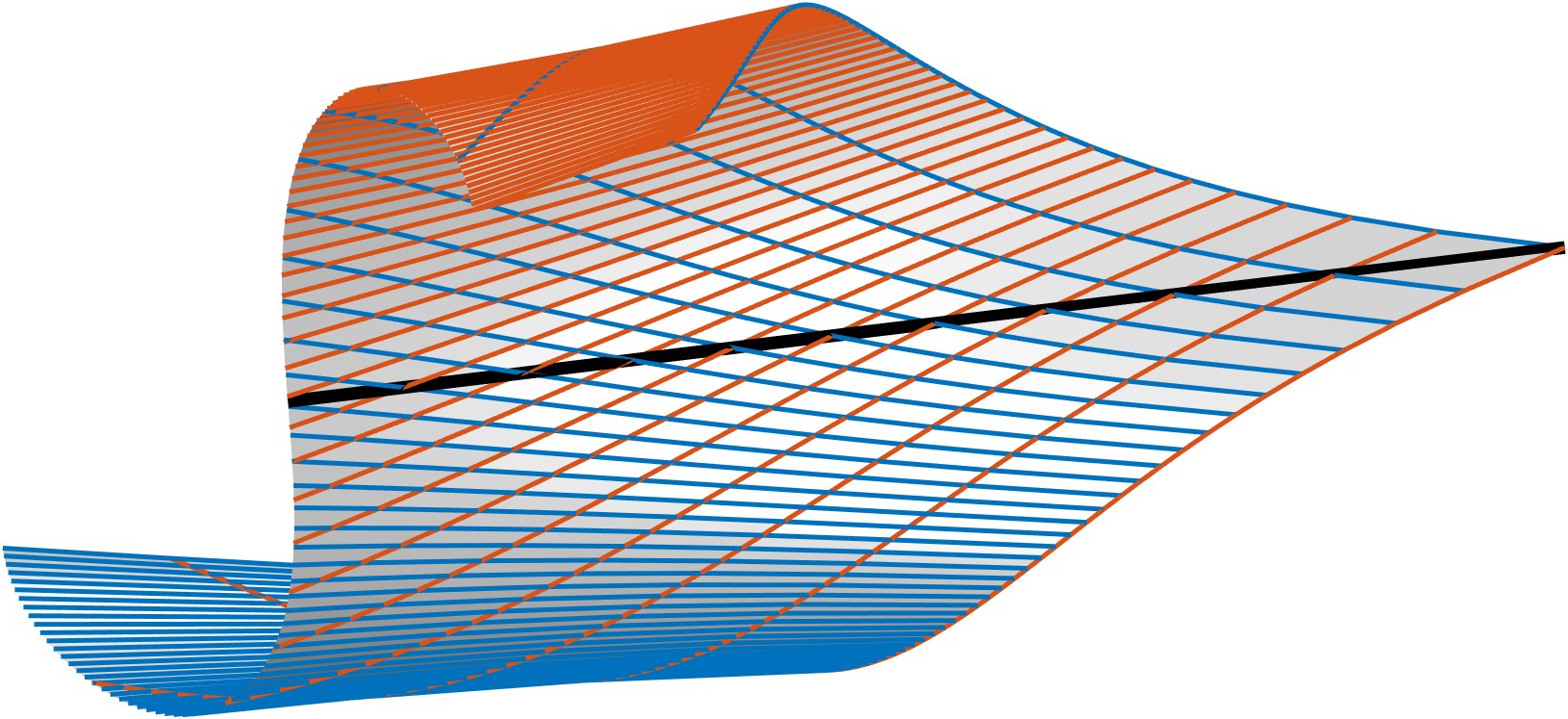}
        \end{minipage}
        \&
        \begin{minipage}[t][\cellh][t]{.31\textwidth}\centering
            \includegraphics[
                width=\linewidth,
                height=25mm,
                keepaspectratio
            ]{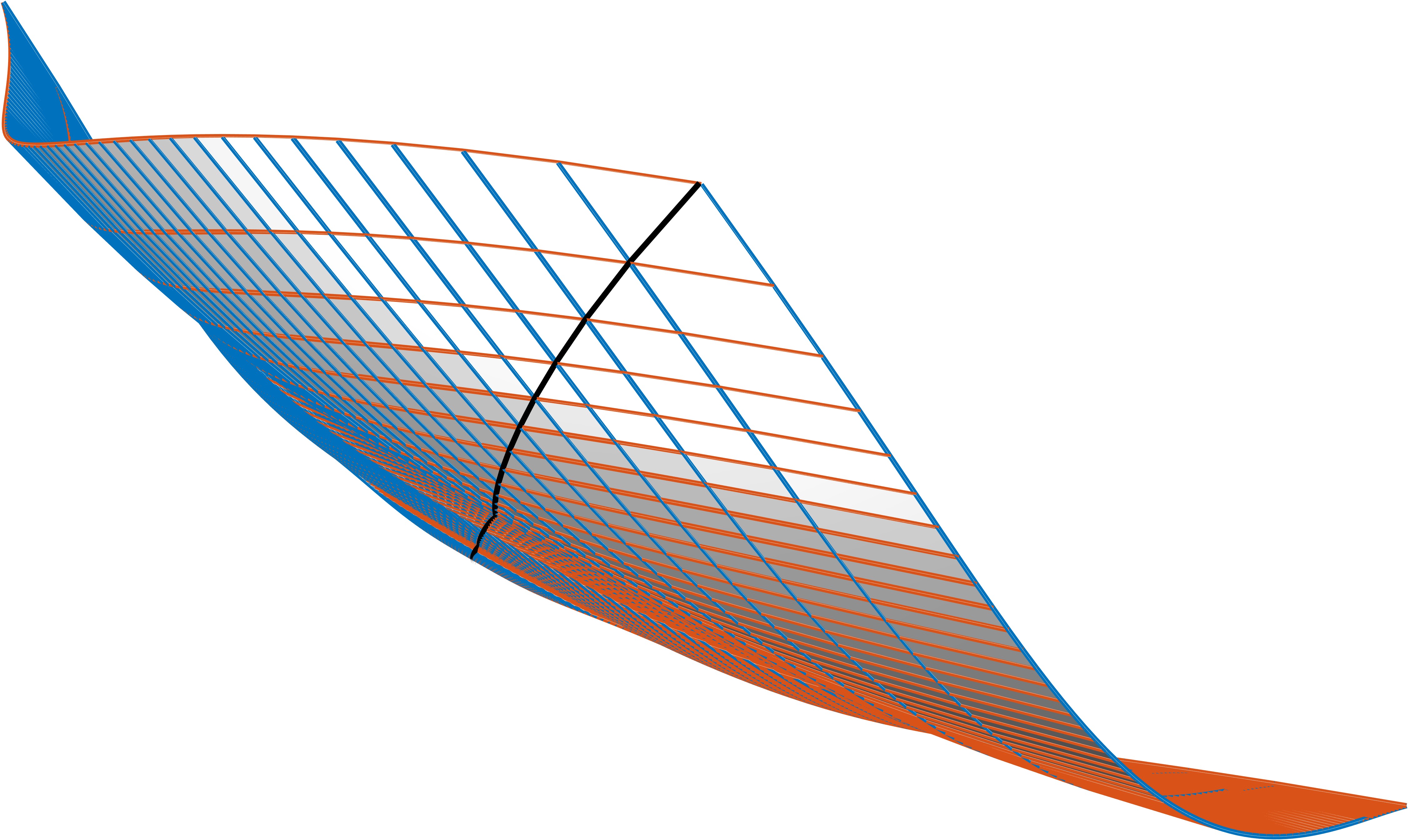}
        \end{minipage}
        \&
        \begin{minipage}[t][\cellh][t]{.31\textwidth}\centering
            \includegraphics[
                width=\linewidth,
                height=28mm,
                keepaspectratio
            ]{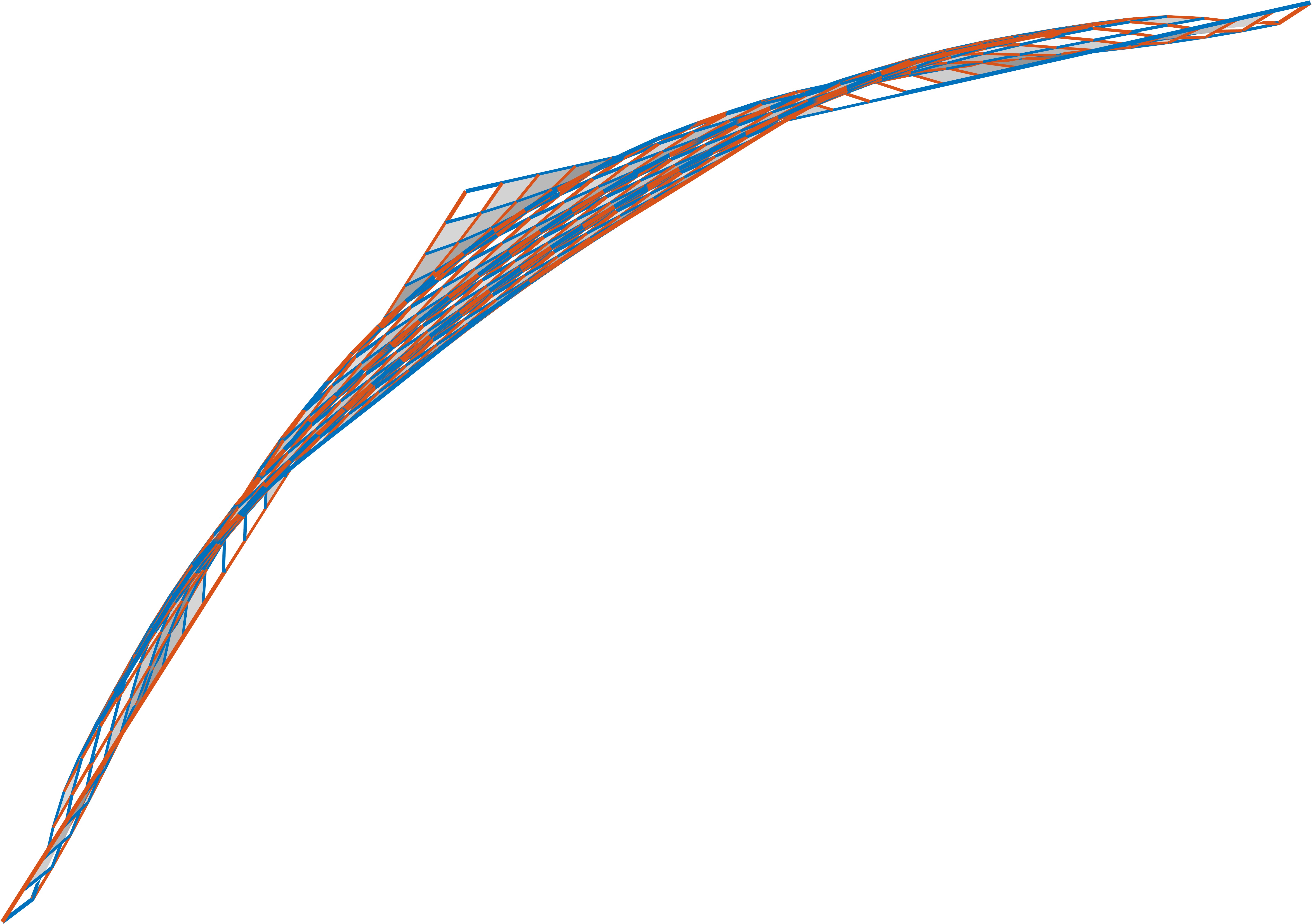}
        \end{minipage}
        \\

        \begin{minipage}[t][\cellh][t]{.31\textwidth}\centering
            \includegraphics[
                width=\linewidth,
                height=25mm,
                keepaspectratio
            ]{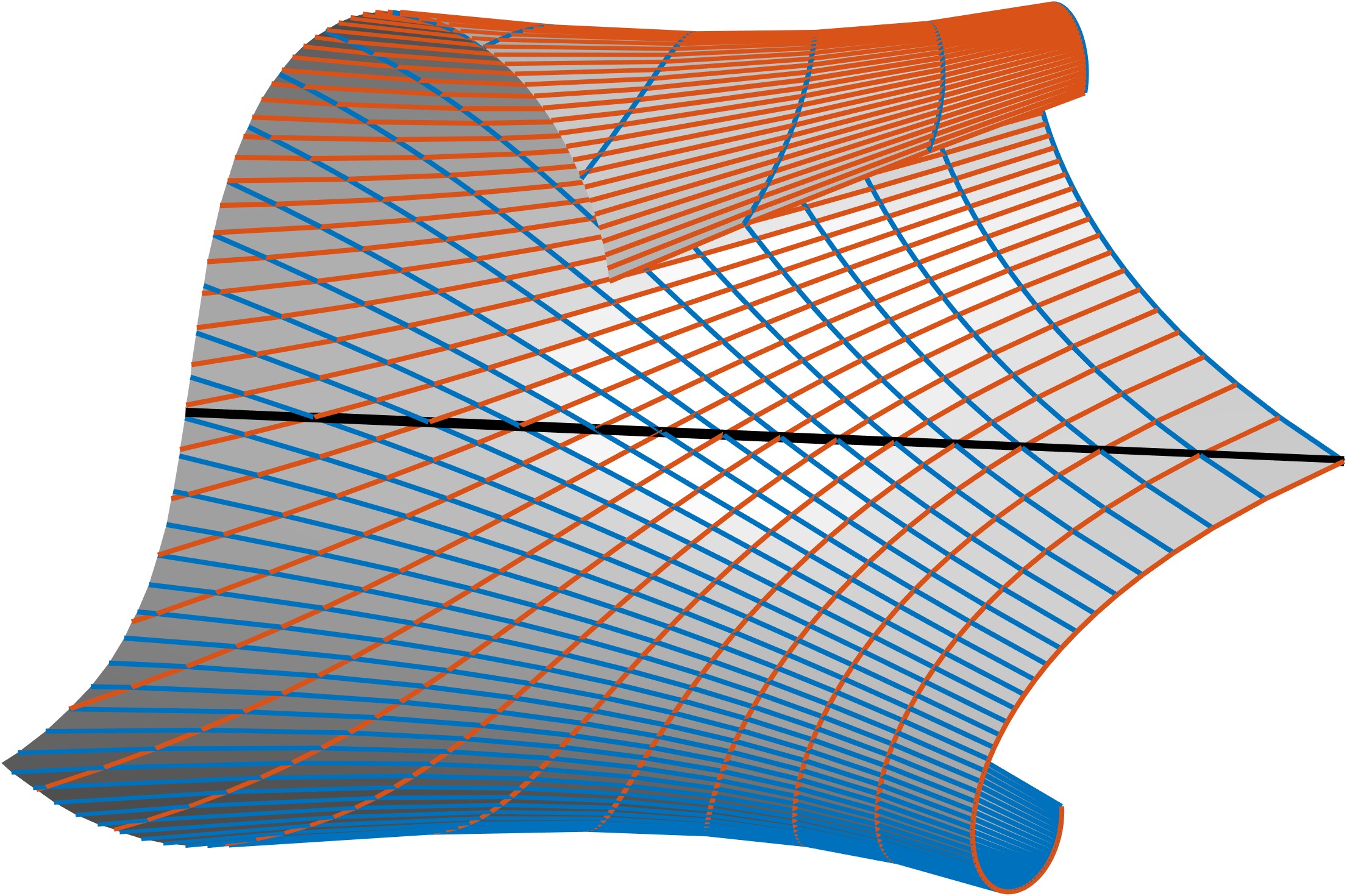}
        \end{minipage}
        \&
        \begin{minipage}[t][\cellh][t]{.31\textwidth}\centering
            \includegraphics[
                width=\linewidth,
                height=30mm,
                keepaspectratio
            ]{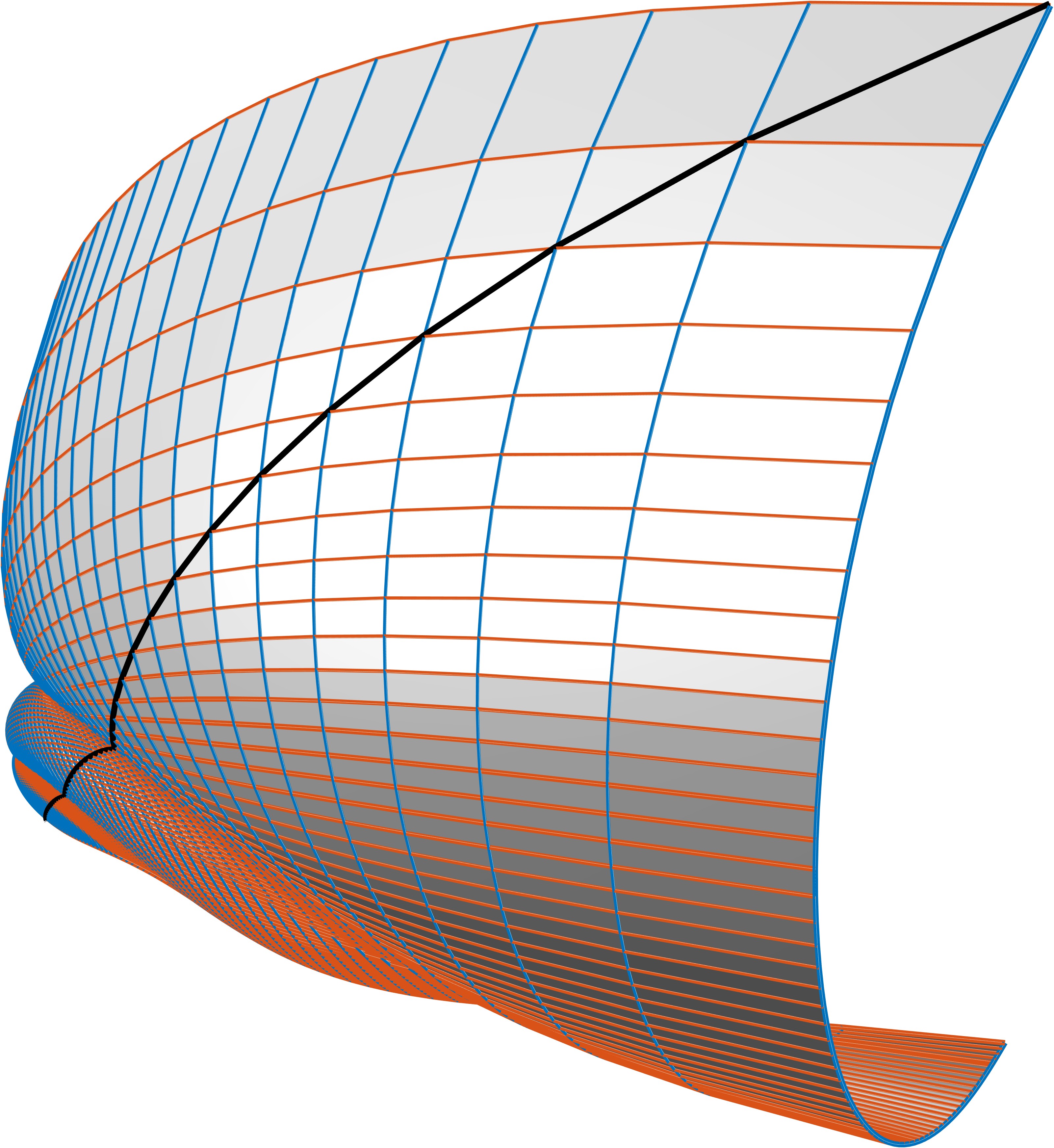}
        \end{minipage}
        \&
        \begin{minipage}[t][\cellh][t]{.31\textwidth}\centering
            \includegraphics[
                width=\linewidth,
                height=30mm,
                keepaspectratio
            ]{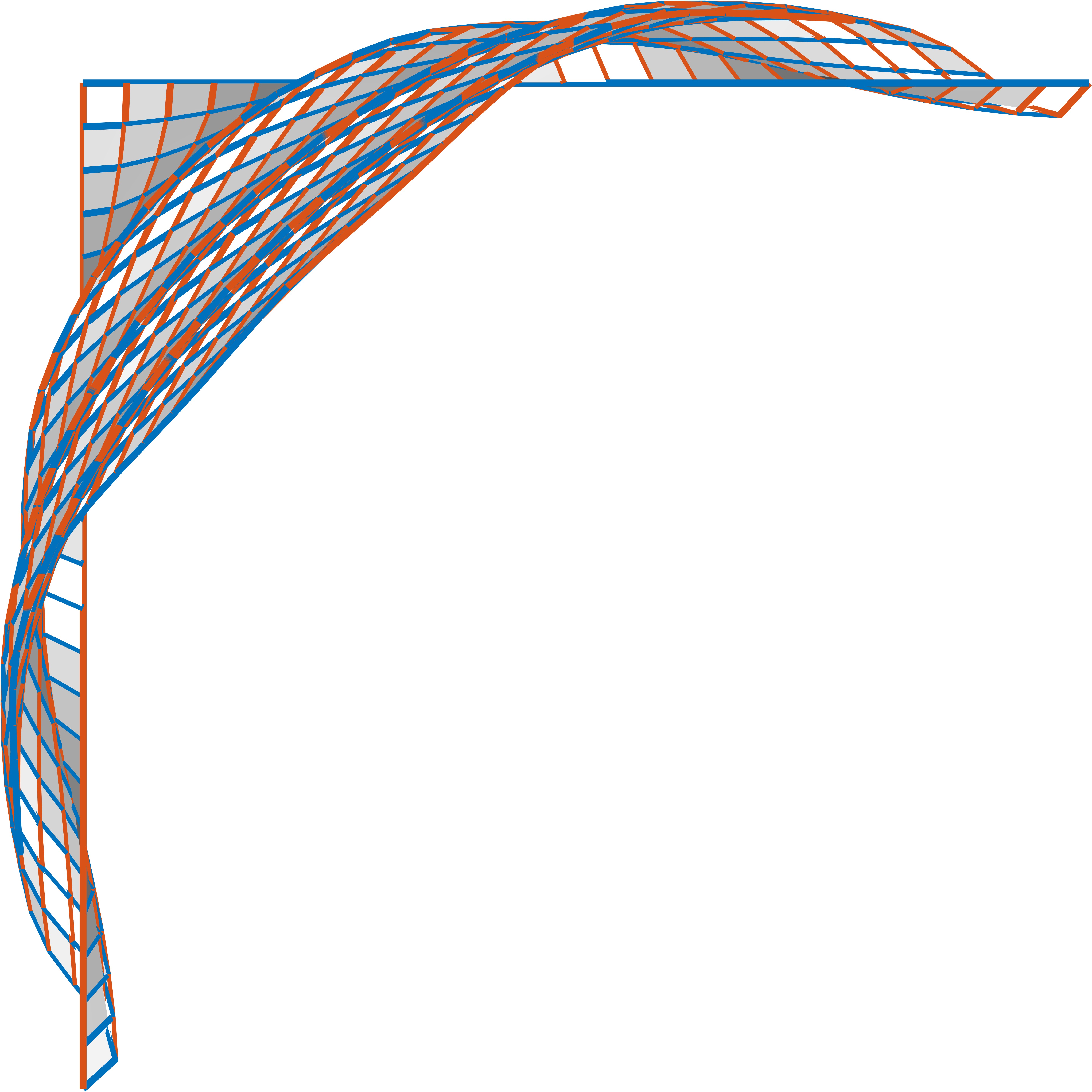}
        \end{minipage}
        \\

        \begin{minipage}[t][\cellh][t]{.31\textwidth}\centering
            \includegraphics[
                width=\linewidth,
                height=32mm,
                keepaspectratio
            ]{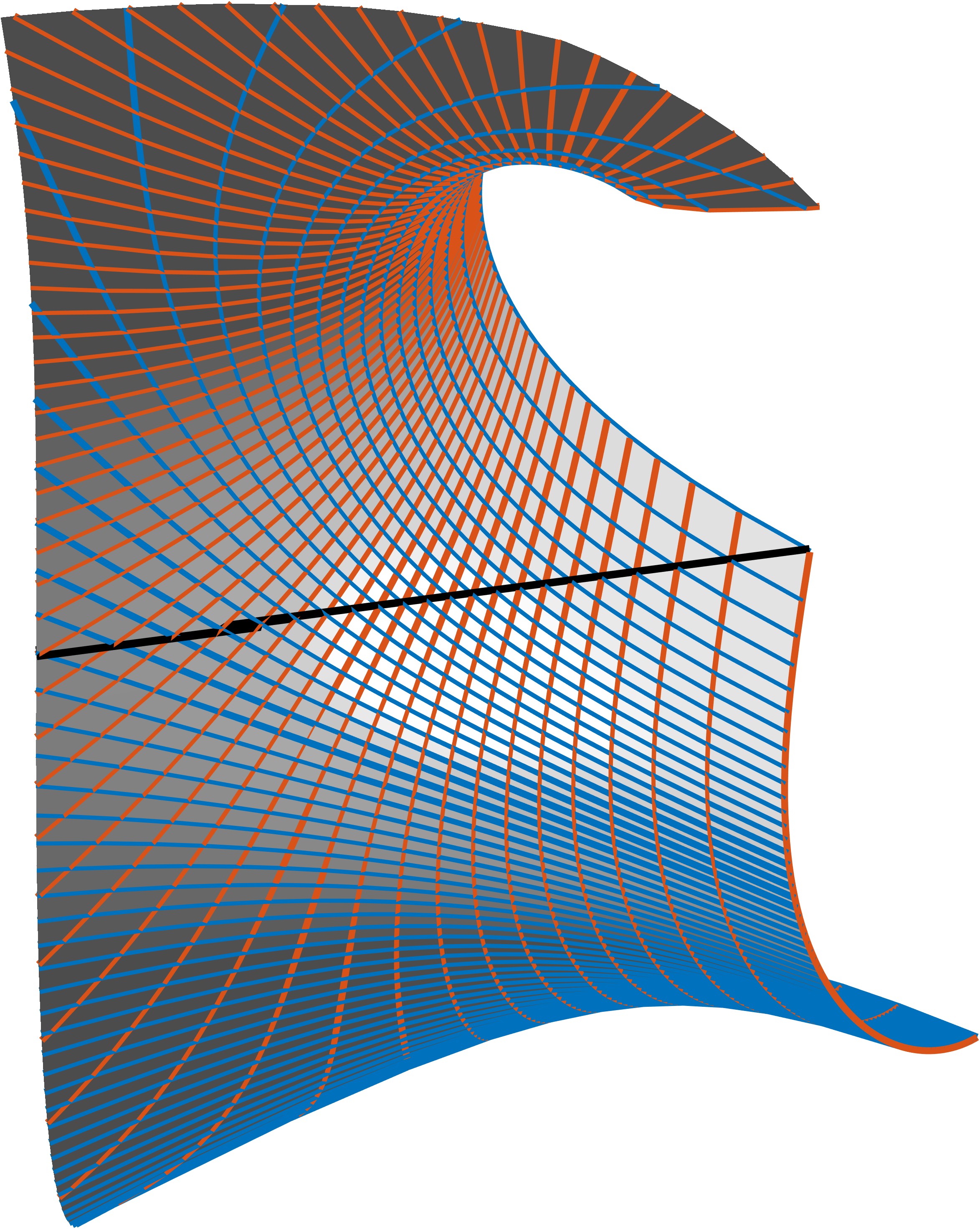}
        \end{minipage}
        \&
        \begin{minipage}[t][\cellh][t]{.31\textwidth}\centering
            \includegraphics[
                width=\linewidth,
                height=30mm,
                keepaspectratio
            ]{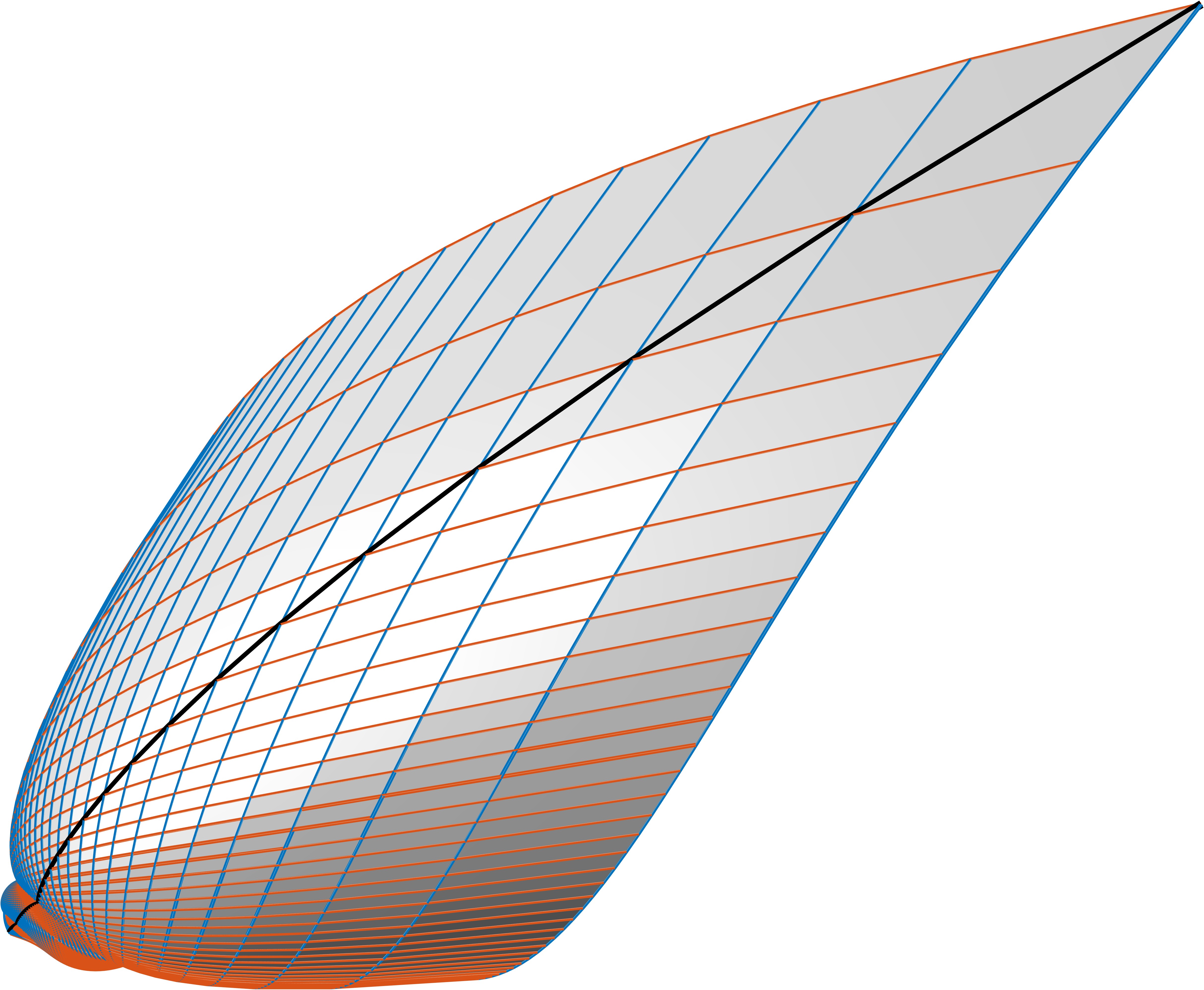}
        \end{minipage}
        \&
        \begin{minipage}[t][\cellh][t]{.31\textwidth}\centering
            \includegraphics[
                width=\linewidth,
                height=30mm,
                keepaspectratio
            ]{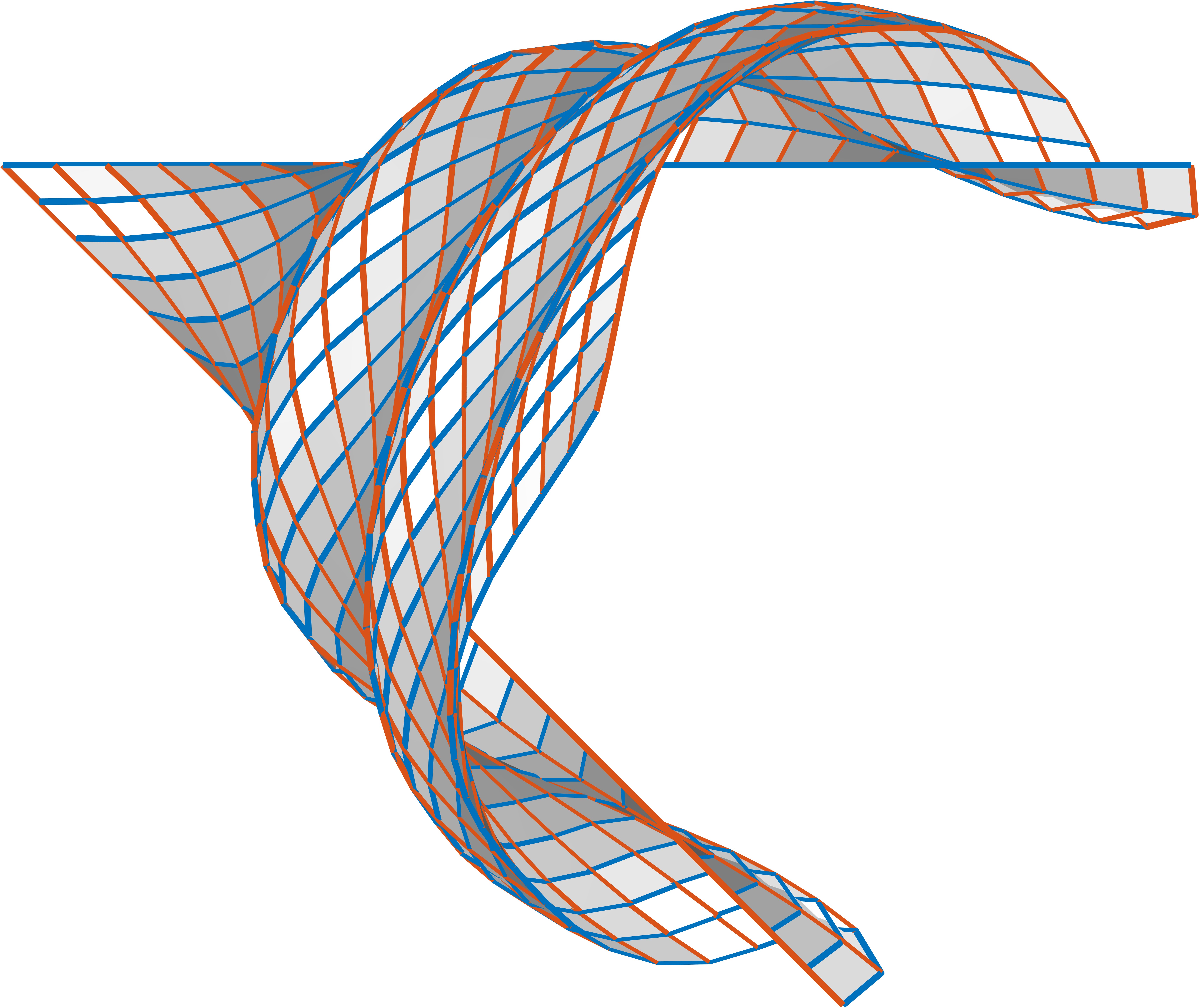}
        \end{minipage}
        \\
    };

    \coordinate (V22SW) at
        ([xshift=3mm,yshift=10mm]m-2-2.south west);

    \coordinate (Cmain) at
        ([xshift=\maincx,yshift=\maincy]V22SW);

    \draw[dashed,line width=0.7pt]
        (Cmain) circle[radius=\rmain];

    \coordinate (Cinset) at
        ([xshift=37mm,yshift=22mm]m-2-2.south west);

    \begin{scope}
        \clip (Cinset) circle[radius=\rinset];
        \node[inner sep=0pt] at (Cinset)
            {\includegraphics[height=19mm]{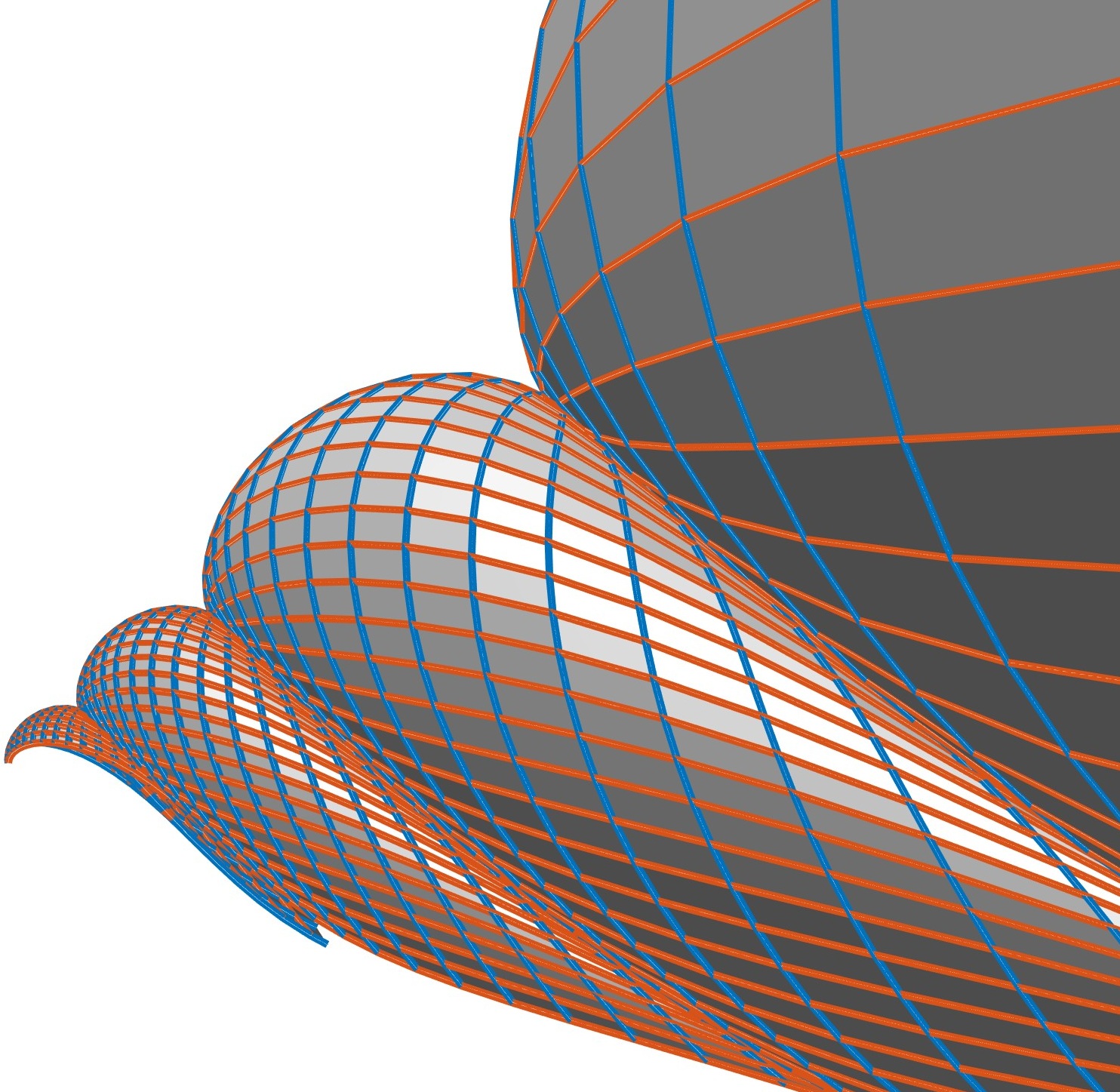}};
    \end{scope}

    \draw[dashed,line width=0.8pt]
        (Cinset) circle[radius=\rinset];

    \draw[dashed,line width=0.7pt]
        ($(Cmain)!\rmain!(Cinset)$)
        --
        ($(Cinset)!\rinset!(Cmain)$);

    \foreach \r/\c/\lbl in {
        1/1/a, 1/2/b, 1/3/c,
        2/1/d, 2/2/e, 2/3/f,
        3/1/g, 3/2/h, 3/3/i
    }{
        \node[
            draw,
            circle,
            fill=white,
            inner sep=1.2pt,
            font=\bfseries\small,
            anchor=north west
        ] at ([xshift=\labx,yshift=\laby]m-\r-\c.north west) {\lbl};
    }

\end{tikzpicture}

\caption{Illustration of the alignable deformation of alignable V-nets of the second kind of the negative-alignable type (a,d,g), the positive-alignable type (b,e,h) in large, together with the corresponding regular Amsler K-nets (c,f,i). In cases (a,c), the portions of the surface extending beyond the boundaries of the figure have been cut off. The rows (a,b,c), (d,e,f), and (g,h,i) stand for $k>\frac{\pi}{2}$, $k=\frac{\pi}{2}$, and $k<\frac{\pi}{2}$, respectively.}
\label{fig:alignable:deformation:all:2nd}
\end{figure*}
\begin{remark}
    In $(u,v)$ parameters, a smooth map $\eta$ is called \emph{weakly regular} if $\eta_u$ and $\eta_v$ are nowhere vanishing, although they need not necessarily be linearly independent. For a regular solution $w(0)=k\in(0,\pi)$ of \eqref{eq:P:III}, The map $\eta^+$ of \eqref{eq:immersion:2nd:kind} extends smoothly to the entire open first quadrant $Q$ as a weakly regular map. $\eta^+(Q)$ develops cusps along the same parameter lines as that correspond to the cusps of the first quadrant of the regular Amsler surface (see Fig.~\ref{fig:alignable:deformation:all:2nd}-(b,e,h)). In contrast, $\eta^-$ becomes unbounded as these parameter lines are approached and therefore admits no continuous extension across them into $\R^3$. Consequently, each connected component of their complement in $Q$ determines a separate regular negative-alignable V-surface of the second kind in large. Fig.~\ref{fig:alignable:deformation:all:2nd}-(a,d,g) illustrates $\eta^-(\mathscr{D})$ for the case where $\mathscr{D}$ is the simply connected open adjacent to the coordinates axes.
\end{remark}
\section*{Declaration of Competing Interest}
The author declares that he has no known competing financial interests or personal relationships that could have
appeared to influence the work reported in this article.
\section*{Acknowledgements}
This work is financed by the European Union, Sächsische Aufbaubank -- Förderbank (application number 100750317). For open access purposes, the author has applied a CC BY public copyright license to any author-accepted manuscript arising from this submission. During the preparation of this manuscript, ChatGPT (OpenAI) was used solely for language editing and stylistic improvement. All mathematical content was developed by the author. The author also wishes to thank Andrew Sageman-Furnas, Hui Wang, Jannik Steinmeier and Ru\v{z}ica Rasoulzadeh-Miji\'{c} for constructive scientific discussions. 
\appendix
\section{Proof of Proposition~\ref{prop:bour:family}}\label{appendix:Bour:I:II}
For the proof we need the following:
\begin{lemma}\label{lem:parallel:geodesics}
    Let $P = (\psi,\Omega)$ be an ISR relative to $Y$. Furthermore, let $(E_1,E_2)$ be the ISR frame adapted to $Y$. Then $E_1$-lines are geodesics while $E_2$-lines have constant geodesic curvature.
\end{lemma}

\begin{proof}
    For the proof, we refer the reader to the supplementary material.
\end{proof}

\textbf{Proof of Proposition~\ref{prop:bour:family}:}
By Lemma~\ref{lem:parallel:geodesics}, the adapted ISR coframe satisfies
\begin{equation}\label{eq:g1:g2}
    \d\theta^1 = \kappa^g_1\,\theta^1\wedge\theta^2 = 0,\qquad\qquad
    \d\theta^2 = \kappa^g_2\,\theta^1\wedge\theta^2 = \frac{G_1}{2 G}\,\theta^1\wedge\theta^2.
\end{equation}
Furthermore, since $\mathcal{L}_Y I = \mathcal{L}_Y II =0$ and $\mathcal{L}_YE_i =0$ for $i=1,2$, we get 
\begin{equation}\label{eq:Y:invariance}
    E_2(L)=E_2(M)=E_2(N)=0.
\end{equation}
Substituting \eqref{eq:g1:g2} in the Gauss-Codazzi equations of \eqref{eq:codazzi:cartan} we arrive at the following 
\begin{equation}\label{eq:sys:bour}
    M_1 + \frac{G_1}{G}\,M = 0,\qquad
    N_1 + \frac{G_1}{2\,G}\,N = \frac{G_1}{2\,G}\,L,\qquad
    LN - M^2 = \frac{G_1^2 - 2\,G\,G_{11}}{4\,G^2}.
\end{equation}
The \eqrefi{eq:sys:bour}{1} gives $E_1(GM)=0$. Together with \eqref{eq:Y:invariance}, this implies $M = \frac{t}{{G}}$ for a constant $t\in\mathbb R$. Multiplying \eqrefi{eq:sys:bour}{2} by $2GN$ and using \eqrefi{eq:sys:bour}{3}, we obtain
\begin{equation*}
        E_1(GN^2)
        = G_1LN
        = \frac{G_1\bigl(G_1^2+4t^2-2GG_{11}\bigr)}{4G^2}
        = E_1\!\left(-\frac{G_1^2}{4G}-\frac{t^2}{G}\right).
\end{equation*}
Consequently, we get $GN^2=s-\frac{G_1^2}{4G}-\frac{t^2}{G}$ for a constant $s\in\mathbb{R}$.
On a connected patch where $4sG-4t^2-G_1^2>0$, solving for $N$ and then applying the Gauss equation gives
\begin{equation*}
    N=\varepsilon\frac{\sqrt{4sG-4t^2-G_1^2}}{2G},
    \qquad
    L=\varepsilon\frac{G_1^2+4t^2-2GG_{11}}{2G\sqrt{4sG-4t^2-G_1^2}},
    \qquad
    \varepsilon\in\{1,-1\}.
\end{equation*}

Since, these coefficients satisfy the Gauss-Codazzi equations, the \emph{fundamental theorem of surface theory} provides the corresponding immersions, unique up to rigid motions. Their first fundamental form is independent of $s,t$, and their second fundamental forms satisfy \eqref{eq:Y:invariance}, giving the Bour family. Finally, $t=0$ implies $M=0$, so the corresponding members are surfaces of revolution by Theorem~\ref{theorem:B:nets} and the ISR frame, due to $I(E_1,E_2)=II(E_1,E_2) = 0$, becomes the principal frame.
\section{Proof of Proposition~\ref{prop:categorizing:IIDs}}\label{appendix:categorizing:IIDs}
For the proof we need the following:
\begin{lemma}[\cite{izmestiev}]\label{lem:DotII}
    Let $\xi$ be an IID of an immersion $\psi$ with the rotation field $\eta$, and let $\psi^\tau$ be any smooth deformation of $\psi$ whose instantaneous velocity at $\tau=0$ is $\xi$. Then the derivatives at $\tau=0$ of the fundamental forms of $\psi^\tau$ are as follows:
    \begin{equation}\label{eq:It:IIt}
        \left. \frac{\d}{\d \tau} \right|_{\tau=0}\mathbf{I}^{\,\tau} = 0, \qquad\qquad \left. \frac{\d}{\d \tau} \right|_{\tau=0}\mathbf{II}^{\,\tau} = - \sqrt{\det (\mathbf{I})}\,\mathbf{A}\,\mathbf{J},
    \end{equation}
    where $\mathbf{A}$ is the matrix of rotation operator, and $\mathbf{J}$ is the matrix of rotation by $\frac{\pi}2$.
\end{lemma}

\textbf{Proof of Proposition~\ref{prop:categorizing:IIDs}:}
Let $(E_1,E_2)$ be the ISR frame of $(\psi,\Omega)$ with the dual coframe $(\theta^1,\theta^2)$. Then the second fundamental form of $\psi$ reads as
\begin{equation*}
    II_\psi = L\,\theta^1\otimes\theta^1+N\,\theta^2\otimes\theta^2.
\end{equation*}
Since $\mathrm{K}_\psi=LN\neq0$, both $L$ and $N$ are nowhere zero. Now, fix an arbitrary point of $\Omega$ and work on a sufficiently small simply connected open neighborhood $U$ on which the Bour family $\psi^{s,t}$ is defined for parameters near $(1,0)$, with $\psi^{1,0}=\psi|_U$. In the following local argument, all vector fields and forms are understood to be restricted to $U$. Denote the instantaneous sliding and wrapping at $(1,0)$ by $\xi^+$ and $\xi^-$, respectively, and let $\eta^\pm$ and $A^\pm$ be their rotation fields and rotation operators. Differentiating the second fundamental form of
\eqrefi{prop:bour:family}{2} with respect to $t$ and $s$, respectively, and using \eqref{eq:It:IIt}, gives
\begin{equation}\label{eq:A:plus:minus}
    A^+
    =\frac{1}{G}\left(E_1\otimes\theta^1-E_2\otimes\theta^2\right),
    \qquad
    A^-
    =\frac{1}{2GN^2}
    \left(N\,E_1\otimes\theta^2+L\,E_2\otimes\theta^1\right).
\end{equation}
Both operators are invertible. Hence, $\d\eta^\pm=\d\psi\circ A^\pm$ (see \eqref{eq:rotation:operator}) shows that $\eta^\pm$ are immersions.

\medskip
\noindent
$(2)\Rightarrow(1)$:
Assume that $\eta$ is a conjugate (asymptotic) net. Let $A$ be the rotation operator of $\xi$ and write
\begin{equation*}
    A(E_1)=aE_1+cE_2,\qquad A(E_2)=bE_1-aE_2.
\end{equation*}
Since $A$ is self-adjoint with respect to $II_\psi$, we have $Lb=Nc$.
Using \eqrefi{eq:I:II:K}{2}, we obtain the following alternatives:
\begin{itemize}
    \item If $\eta$ is a conjugate net, then
    $II_\eta(E_1,E_2)=0$, so $Lb=Nc=0$.
    Thus $b=c=0$ and $A=(aG)A^+$.
    \item If $\eta$ is an asymptotic net, then
    $II_\eta(E_1,E_1)=II_\eta(E_2,E_2)=0$, so $La=Na=0$.
    Thus $a=0$, and together with $Lb=Nc$, this yields
    $A=(2GNb)A^-$.
\end{itemize}
In either case, $A=fA^\pm$ for some $f\in C^\infty(U)$, and therefore $\d\eta=f\,\d\eta^\pm$. Taking exterior derivatives gives $\d f\wedge\d\eta^\pm=0$.
Since $\eta^\pm$ is an immersion, the vectors $\d\eta^\pm(E_1)$ and $\d\eta^\pm(E_2)$ are linearly independent.
Consequently, $\d f=0$, and hence $f=C$ for a constant $C\in\R$.
It follows that $\eta=C\eta^\pm+r$ for a constant $r\in\R^3$. Since $\eta$ is an immersion, $C\neq0$.
Integrating $\d\xi=\eta\times\d\psi$ then gives $\xi=C\xi^\pm+r\times\psi+q$, where $q\in\R^3$ is constant. Set $(\psi^+)^\tau=\psi^{1,\tau}$ and $(\psi^-)^\tau=\psi^{1+\tau,0}$. Choose a smooth family of rigid motions
$\phi^\tau\colon\R^3\to\R^3$ satisfying $\left.\frac{\d}{\d\tau}\right|_{\tau=0}\phi^\tau(x)=r\times x+q$ and $\phi^0=\operatorname{id}$. Then $\xi|_U
    =
    \left.\frac{\d}{\d\tau}\right|_{\tau=0}
    \left(\phi^\tau\circ(\psi^\pm)^{C\tau}\right)$.
Since time rescaling and ambient rigid motions preserve the sliding or wrapping condition, $\xi|_U$ is an instantaneous sliding or wrapping, respectively. It is nontrivial because its rotation field is an immersion. As the initial point was arbitrary, condition~(1) follows.

\medskip
\noindent
$(1)\Rightarrow(2)$:
Suppose that $\xi$ is locally a nontrivial instantaneous sliding
(wrapping) of $\psi$. On each sufficiently small neighborhood $U$,
the Bour description gives
\[
    A=CA^+
    \qquad\text{or}\qquad
    A=CA^-,
\]
respectively, where $C\neq0$ is constant. The invertibility of
$A^\pm$ shows that $\eta|_U$ is an immersion.
In the sliding case, $A$ is diagonal, and
\eqrefi{eq:I:II:K}{2} gives $II_\eta(E_1,E_2)=0$.
In the wrapping case, $A$ is off-diagonal, and the same relation gives
$II_\eta(E_1,E_1)=II_\eta(E_2,E_2)=0$. Thus $\eta$ is a conjugate (asymptotic) net. Since these properties
hold locally throughout $\Omega$, condition~(2) follows.

\medskip
\noindent
Finally, let $\xi$ and $\tilde\xi$ satisfy these equivalent conditions
and be of the same type. Denote their rotation operators by
$A$ and $\tilde A$. The preceding classification gives
$A=h\tilde A$ for a smooth nowhere vanishing function $h$ on $\Omega$.
Hence, $\d\eta=h\,\d\tilde\eta$, which itself gives $\d h\wedge\d\tilde\eta=0$.
Since $\tilde\eta$ is an immersion, the same argument as above yields
$\d h=0$. Thus $h=C\in\R^\ast$ is constant on the connected domain
$\Omega$, and integration gives $\eta=C\tilde\eta+r$ for a constant
$r\in\R^3$. Consequently, $\d\bigl(\xi-C\tilde\xi-r\times\psi\bigr)=0$, which yields \eqref{eq:rot:comp} with a constant $q\in\R^3$.
\section{Proof of Proposition~\ref{prop:K:net:immersion}}\label{appendix:K:net:immersion}
For the proof we resort to the quaternion approach of \cite{BobenkoQuaternion}. 
Furthermore, it suffices to consider the component $m=0$ of \eqref{eq:Omega:u:v}, choosing $u+v>0$ when $k=1$. The formulas on the remaining components follow by box reparametrizations and rigid motions. Now, consider the reparametrization $x = u + v$, $y = k(u - v)$. The Lax pair becomes
\begin{equation}\label{eq:lax:pair:X:Y}
    \Phi_x=\frac{1}{2}(\mathbf U+\mathbf V)\,\Phi,\qquad\qquad
    \Phi_y=\frac{1}{2k}(\mathbf U-\mathbf V)\,\Phi,
\end{equation}
with $\mathbf{U}$ and $\mathbf{V}$ now being evaluated at $(x,y)$ points. In the new coordinates $(x,y)$ we have $\omega_y = 0$. Consequently, using \cite[Eq.~8.14]{BobenkoQuaternion} for $\lambda=1$, we get
\begin{equation}
\mathbf U+\mathbf V=\sin\!\left(\frac{\omega(x)}{2}\right)\mathbf J,\qquad\qquad
\mathbf U-\mathbf V=i\,\mathbf M(x),
\end{equation}
where $\mathbf{J} = -i\boldsymbol{\sigma}_2$ and $\mathbf{M}(x) = \left(\frac{\omega_x}{2}\right) \boldsymbol{\sigma}_3 - \cos{(\frac{\omega}{2})} \boldsymbol{\sigma}_1$. Now, using identities from \eqref{eq:elliptic:modulus}
\begin{equation}\label{eq:jacobi:identities}
    \sin{\left(\frac{\omega}{2}\right)} = \dn(x \mid k),\qquad
    \cos{\left(\frac{\omega}{2}\right)} = k\,\sn(x \mid k),\qquad
    \omega_x = -2\,k\,\cn(x \mid k),
\end{equation}
results in  
\begin{equation*}
    \mathbf{M}(x) = -k\,\left(\begin{array}{cc}
    \cn(x \mid k)  & \sn(x \mid k) \\
    \sn(x \mid k)  & -\cn(x \mid k)
    \end{array} \right),\quad\quad\quad\quad
    \mathbf{M}(x)^2 = k^2\,\id.
\end{equation*}
Since $\mathbf{M}(x)$ is independent of $y$, solving \eqrefi{eq:lax:pair:X:Y}{2} yields
\begin{equation}\label{eq:Phi}
    \Phi(x,y)=\exp\!\left(\frac{iy}{2k}\,\mathbf M(x)\right)\,\Phi(x,0).
\end{equation}
In order to fix $\Phi(x,0)$, simplifying \eqrefi{eq:lax:pair:X:Y}{1} and using the identities of \eqref{eq:jacobi:identities} gives
\begin{equation}\label{eq:phix:II}
    \Phi_x=\frac12(\mathbf U+\mathbf V)\Phi
    =\frac12\sin\!\left(\frac{\omega(x)}{2}\right)\mathbf J\,\Phi
    =\frac12\dn(x\mid k)\,\mathbf J\,\Phi.
\end{equation}
Now, solving \eqref{eq:phix:II} along the line $y=0$ produces
\begin{equation}\label{eq:Phi:x:0}
    \Phi(x,0)=\exp\!\big(\upsilon(x)\,\mathbf J\big)\,\Phi(0,0),\qquad
    \upsilon(x)=\frac12\int_0^x \dn(z\mid k)\,dz=\frac12\,\mathrm{am}(x\mid k).
\end{equation}
We fix the base point $\Phi(0,0) = \id$. On the other hand, differentiating the Sym formula (see \cite[eq.~8.15]{BobenkoQuaternion}) with respect to $(u,v)$ and using $\Phi_u=\mathbf{U}\Phi$, $\Phi_v=\mathbf{V}\Phi$ and $\Psi_x=\frac12(\Psi_u+\Psi_v)$, $\Psi_y=\frac{1}{2k}(\Psi_u-\Psi_v)$ we get the following tangent formulas at $\lambda = 1$
\begin{align*}
    \Psi^{1}_x
    &=\Phi^{-1}\,\big((U_\lambda+V_\lambda)|_{\lambda=1}\big)\,\Phi
    =-i\cos\left(\frac{\omega(x)}{2}\right)\,\Phi^{-1}\boldsymbol{\sigma}_{1}\Phi,
    \\
    \Psi^{1}_y
    &=\frac{1}{k}\Phi^{-1}\,\big((U_\lambda-V_\lambda)|_{\lambda=1}\big)\,\Phi
    =\frac{1}{k}\sin\left(\frac{\omega(x)}{2}\right)\,\Phi^{-1}\,\mathbf{J}\,\Phi.
\end{align*}
Since the Lax system is compatible, the $\mathfrak{su}(2)$-valued 1-form $\Psi_x\,\d x+\Psi_y\,\d y$ is closed. Thus, $\Psi(x,y)$ is obtainable (up to translation) by integrating $\Psi_x$ and $\Psi_y$ from a base point $\Psi^1(0,0)=0$.
Now, plugging in \eqref{eq:Phi} with the earlier step \eqref{eq:Phi:x:0}, setting $\Psi(0,0) = 0$ and using $\frac{\d}{\d x}\mathcal{E}\left(\mathrm{am}(x\mid k)\mid k\right) = \dn^2(x\mid k)$, gives us the immersion $\Psi^1(x,y)$. Taking the pre-image of $\Psi^1(x,y)$ under the identification $(x,y,z) \mapsto x\,i\,\boldsymbol{\sigma}_1 + y\,i\,\boldsymbol{\sigma}_2 + z\,i\,\boldsymbol{\sigma}_3$ and subsequently applying the fixed Euclidean isometry $(x,y,z)\mapsto(x-\frac{1}{k},-y,z)$ yields the immersion formula for $\psi(x,y)$. Upon reparametrizing back to $(u,v)$ coordinates we get the K-net.

\bibliographystyle{siamplain}
\bibliography{references}

\end{document}